\documentclass{article}
\usepackage[a4paper,left=2.5cm,right=2.5cm, top=3cm, bottom=3cm]{geometry}

\usepackage{cuted}
\usepackage{enumitem}
\usepackage{amsfonts}
\usepackage{comment}
\usepackage{graphicx}
\usepackage{epstopdf}
\usepackage{scalerel}
\usepackage{framed,multirow}
\usepackage{etoolbox}
\usepackage{amscd,amsmath,amssymb,mathrsfs,bbm,listings}
\usepackage{amsthm}%AM
\usepackage{tikz} 
\usetikzlibrary{arrows,automata,calc}
\usepackage{epsfig}
\usepackage{multirow}
\usepackage{subfig}
\usepackage{dirtytalk}

\newtheorem{theorem}{Theorem}[section]

\newtheorem{lemma}[theorem]{Lemma}
\newtheorem{proposition}[theorem]{Proposition}
\newtheorem{corollary}[theorem]{Corollary}
\newtheorem{remark}[theorem]{Remark}
\newtheorem{definition}[theorem]{Definition}

\usepackage[hidelinks]{hyperref}
\definecolor{myblue}{rgb}{0,0,0.6}   
\hypersetup{
    colorlinks,
    linktoc=section,
    citecolor=red,
    linkcolor=myblue,
    pdfborder={0 0 0}
}
\usepackage{orcidlink}

\numberwithin{equation}{section}

\ifpdf
  \DeclareGraphicsExtensions{.eps,.pdf,.png,.jpg}
\else
  \DeclareGraphicsExtensions{.eps}
\fi

\usepackage[normalem]{ulem}
\newcounter{corr}
\definecolor{violet}{rgb}{0.580,0.,0.827}
\newcommand{\corr}[3]{\typeout{Warning : a correction remains in page \thepage}
  \stepcounter{corr}
 	            {\color{blue}\ifmmode\text{\,\sout{\ensuremath{#1}}\,}\else\sout{#1}\fi}
              {\color{red}#2}
              {\color{violet} #3}
}

\renewcommand{\div}{\mathrm{div}}
\DeclareMathOperator{\Div}{div}
\newcommand{\Nabla}{\boldsymbol{\nabla}}
\renewcommand{\u}{\boldsymbol{u}}
\renewcommand{\v}{\boldsymbol{v}}
\newcommand{\z}{\boldsymbol{z}}
\newcommand{\w}{\boldsymbol{w}}
\newcommand{\uo}{\boldsymbol{u}_0}
\newcommand{\f}{\boldsymbol{f}}
\newcommand{\ST}{\Sigma_T}
\newcommand{\SO}{\Sigma_0}
\newcommand{\SD}{\Sigma_{\mathcal{D}}}
\newcommand{\Z}{\boldsymbol{\mathcal{Z}}}
\newcommand{\V}{\boldsymbol{\mathcal{V}}}

\newcommand{\Q}{\mathcal{Q}}

\newcommand{\bnOmega}{\Uu{n}_{\Omega}}
\newcommand{\bnF}{\Uu{n}_{F}}

\newcommand{\Qm}{Q_m}
\newcommand{\Qn}{Q_n}
\newcommand{\Ru}{{\cal R}_{\u}}
\newcommand{\In}{I_n}
\newcommand{\Inmo}{I_{n-1}}
\newcommand{\tn}{t_n}
\newcommand{\tnmo}{t_{n-1}}
\newcommand{\Sn}{\Sigma_n}
\newcommand\Vh{\boldsymbol{{\mathcal V}_{h}}}
\newcommand{\Vstar}{\boldsymbol{V_{*}}}
\newcommand\Vht{\boldsymbol{{\mathcal V}_{h, \tau}}}
\newcommand\Zht{\boldsymbol{{\mathcal Z}_{h, \tau}}}
\newcommand\Sht{\boldsymbol{{\mathcal S}_{h, \tau}}}
\newcommand\Qh{\mathcal{Q}_h}
\newcommand\Qht{\mathcal{Q}_{h, \tau}}
\newcommand{\Th}{{\calT_h}}
\newcommand{\Tt}{\calT_{\tau}}
\newcommand{\uht}{\u_{h, \tau}}
\newcommand{\uhtt}{\widetilde{\u}_{h, \tau}}
\newcommand{\vhtt}{\widetilde{\v}_{h, \tau}}
\newcommand{\vht}{\v_{h, \tau}}
\newcommand\vh{\v_h}
\newcommand{\zht}{\z_{h, \tau}}
\newcommand{\wht}{\w_{h, \tau}}
\newcommand{\whtt}{\widetilde{\w}_{h, \tau}}
\newcommand{\pht}{p_{h, \tau}}
\newcommand{\qht}{q_{h, \tau}}
\newcommand{\zh}{z_h}
\newcommand{\Nablah}{\Nabla_h}
\newcommand{\Bht}{\mathcal{B}_{h, \tau}}

\newcommand{\mht}{m_{h, \tau}^{\partial_t}}
\newcommand{\aht}{a_{h, \tau}}
\newcommand{\calA}{\mathcal{A}}
\newcommand{\calB}{\mathcal{B}}
\newcommand{\calH}{\mathcal{H}}
\newcommand{\bht}{b_{h, \tau}}
\newcommand{\cht}{c_{h, \tau}}
\newcommand{\chtb}{\overline{c}_{h, \tau}}
\newcommand{\chth}{\widehat{c}_{h, \tau}}
\newcommand{\chtt}{\widetilde{c}_{h, \tau}}

\makeatletter
\newcommand{\hatbar}[1]{% 
\begingroup%
  \let\macc@kerna\z@%
  \let\macc@kernb\z@%
  \let\macc@nucleus\@empty%
  \widehat{\mathchoice%
    {\raisebox{.3ex}{\vphantom{\ensuremath{\displaystyle #1}}}}%
    {\raisebox{.3ex}{\vphantom{\ensuremath{\textstyle #1}}}}%
    {\raisebox{.16ex}{\vphantom{\ensuremath{\scriptstyle #1}}}}%
    {\raisebox{.14ex}{\vphantom{\ensuremath{\scriptscriptstyle #1}}}}%
    \smash{\overline{#1}}}%
\endgroup%
}
\makeatother
\newcommand{\chtbh}{%widehat{\widehat{c}}
\hatbar{c}_{h, \tau}}
\newcommand{\Fh}{\calF_h}
\newcommand{\FhI}{\calF_{h, \circ}}
\newcommand{\FhD}{\calF_{h, \mathcal{D}}}
\newcommand{\Fhtime}{\calF_h^{\mathrm{time}}}
\newcommand{\Fhtimen}{\calF_h^{\mathrm{time},(n)}}
\newcommand{\FhItime}{\calF_{h, \circ}^{\mathrm{time}}}
\newcommand{\FhItimen}{\calF_{h, \circ}^{\mathrm{time}, (n)}}
\newcommand{\FhDtime}{\calF_{h, \mathcal{D}}^{\mathrm{time}}}
\newcommand{\FhDtimen}{\calF_{h, \mathcal{D}}^{\mathrm{time},(n)}}
\newcommand{\ItR}{\mathcal{I}_{\tau}^{\mathcal{R}}}
\newcommand{\QtR}{\mathsf{Q}_{\tau}^{\mathcal{R}}}
\newcommand{\QtRn}{\mathsf{Q}_{\tau}^{\mathcal{R}, (n)}}
\newcommand{\QtRm}{\mathsf{Q}_{\tau}^{\mathcal{R}, (m)}}
\newcommand{\RTk}{\mathrm{RT}_k}
\newcommand{\BDMk}{\mathrm{BDM}_k}
\newcommand\Pp[2]{\mathbb{P}_{#1}(#2)}
\newcommand\Ppo[2]{\mathbb{P}^0_{#1}(#2)}
\newcommand{\Kp}{K_+}
\newcommand{\Km}{K_-}
\newcommand\cA{C_{\mathcal{A}}}
\newcommand\CSI{C_{S}^{\mathcal{I}}}
\newcommand{\jump}[1]{[\![#1]\!]}

\newcommand{\Clin}{C_{\mathrm{stab}}^{\mathrm{lin}}}
\newcommand{\Tnorm}[2]{|\!|\!|#1|\!|\!|_{\mathcal{E};#2}}

\newcommand{\Id}{\mathsf{Id}}
\newcommand{\Pt}{\mathcal{P}_{\tau}}
\newcommand{\IRT}{\mathcal{I}_h^{\mathsf{RT}}}
\newcommand{\Piht}{\Pi_{h, \tau}}
\newcommand{\eu}{\boldsymbol{e_u}}
\newcommand{\epi}{\boldsymbol{e_\Pi}}
\newcommand{\eremk}{\hbox{}\hfill\rule{0.8ex}{0.8ex}}
\newcommand\dx{\,\mathrm{d}\bx} 
\newcommand\dS{\,\mathrm{d}S}
\newcommand\ds{\,\mathrm{d}s}
\newcommand\dt{\,\mathrm{d}t}
\newcommand\dV{\,\mathrm{d}V}

\newcommand\QT{Q_T}

\newcommand\dpt{\partial_t}

\newcommand\p{{\uu{p}}}

\newcommand{\red}[1]{{{\color{black} #1}}}

\numberwithin{equation}{section}

\newlength{\dhatheight}

\allowdisplaybreaks[4]
\newcommand{\uu}[1]{\hbox{\boldmath$#1$}}
\newcommand{\Uu}[1]{{\mathbf{#1}}}

\newcommand{\calT}{{\mathcal T}}

\newcommand{\calF}{{\mathcal F}}

\newcommand{\hK}{h_{K}}

\newcommand{\IN}{\mathbb{N}}
\newcommand{\IR}{\mathbb{R}}

\newcommand{\bx}{{\Uu x}}

\newcommand{\bn}{{\Uu n}}

\newcommand{\IZ}{\mathbb{Z}}

 \newcommand{\mvl}[1]{\{ \!\!\{#1\}\!\!\}}

\newcommand{\Cinv}{C_{\mathrm{inv}}}

\DeclareMathOperator{\diam}{diam} 
\DeclareMathOperator{\esssup}{ess\,sup}

\newcommand*{\SemiNorm}[2]{\left|#1\right|_{#2}}
\newcommand*{\Norm}[2]{\|#1\|_{#2}}

\newcommand{\errU}{\texttt{err}(\mathbf{u})}

\title{\red{Fully and semi-implicit robust space--time DG methods for the incompressible Navier--Stokes equations}}

\author{L. Beir\~ao da Veiga\,\orcidlink{0000-0001-5895-469X}\thanks{Department of Mathematics and Applications, University of Milano-Bicocca, Via Cozzi 55, 20125 Milan, Italy (\href{mailto:lourenco.beirao@unimib.it}{lourenco.beirao@unimib.it}, \href{mailto:franco.dassi@unimib.it}{franco.dassi@unimib.it}, \href{mailto:sergio.gomezmacias@unimib.it}{sergio.gomezmacias@unimib.it})} \footnote{ IMATI-CNR, 27100 Pavia, Italy} \and F. Dassi\,\orcidlink{0000-0001-5590-3651}\footnotemark[1] \and S. G\'omez\,\orcidlink{0000-0001-9156-5135
}\footnotemark[1]}
\date{}

\begin{document}

\maketitle

\begin{abstract}
\noindent 
We carry out a stability and convergence analysis of a fully discrete scheme for the time-dependent Navier--Stokes equations resulting from combining an~$H(\div, \Omega)$-conforming discontinuous Galerkin spatial discretization, and a discontinuous Galerkin time stepping scheme. Such a scheme is proven to be pressure robust and Reynolds semi-robust.
Standard techniques can be used to analyze only the case of lowest-order approximations in time. Therefore, we use some nonstandard test functions to prove existence of discrete solutions, unconditional stability, and quasi-optimal convergence rates for any degree of approximation in time. In particular, a continuous dependence of the discrete solution on the data of the problem, and quasi-optimal convergence rates for low and high Reynolds numbers are proven in an energy norm including the term~$L^{\infty}(0, T; L^2(\Omega)^d)$ for the velocity.
\red{In addition to the standard discontinuous Galerkin time stepping scheme, which is fully implicit, we propose and analyze a novel high-order semi-implicit version that avoids the need of solving nonlinear systems of equations after the first time slab, thus significantly improving the efficiency of the method. }
Some numerical experiments validating our theoretical results are presented.
\end{abstract}



\section{Introduction}
This work focuses on the numerical approximation of the solution to the time-dependent incompressible Navier--Stokes equations.

Let the space--time domain~$\QT$ be given by~$\Omega \times (0, T)$, where~$\Omega \subset \IR^d$ ($d \in \{2, 3\}$) is an open, bounded domain with Lipschitz boundary~$\partial \Omega$, and~$T > 0$ is the final time. We define the surfaces~$\SO := \Omega \times \{0\}$, $\ST:= \Omega \times \{T\}$, and~$\SD := \partial \Omega \times (0, T)$.  For given force term~$\f: \QT \to \IR^d$, initial datum~$\uo : \Omega \to \IR^d$, and constant viscosity~$\nu > 0$, we consider the following incompressible Navier--Stokes IBVP: find the velocity~$\u: \QT \to \IR^d$ and the pressure~$p : \QT \to \IR$, such that
\begin{equation}
\label{eq:model-problem}
\begin{cases}
\dpt \u - \nu \Delta \u + (\Nabla \u) \u + \nabla p = \f & \text{ in } \QT, \\
\Div \u = 0 & \text{ in } \QT, \\
\u = {\textbf{0}} & \text{ on } \SD, \\
\u = \uo & \text{ on } \SO.
\end{cases}
\end{equation}

% --------------------------------------------------
%                   PREVIOUS WORKS
% --------------------------------------------------

\paragraph{Previous works.}
The literature on numerical methods for the approximation of solutions to the Navier--Stokes equations is extensive, which evidences the high interests in the simulation of incompressible fluid flows. We focus on methods satisfying the pressure- and Reynold-semi-robustness properties.

A method is said to be \emph{pressure robust} \cite{John_etal:2017} if changes in the data that do not affect the continuous velocity solution (but only the pressure) retain the same property also at the discrete level. An important consequence of pressure robustness is that the error estimates for the velocity~$\u$ are not influenced by the pressure~$p$.  
%%
% The Helmholtz-Hodge decomposition (see~\cite[Lemma 2.6]{John_etal:2017} and~\cite[Thm.~3.3]{Gauger_Linke_Schroeder:2019}) is a fundamental result for determining the pressure-robustness of a scheme. 
% Given~$\w \in L^2(\Omega)^d$, it implies the existence of a vector field~$\w_0 \in H(\div, \Omega)$ and a scalar function~$\varphi \in H^1(\Omega)/\IR$, such that: \emph{i)} $\w = \w_0 + \nabla \phi$; \emph{ii)} $\Div \w_0 = 0$; \emph{iii)} $(\w_0, \nabla \phi)_{\Omega} = 0$.
% Since such a decomposition is unique, it is possible to define a Helmholtz-Hodge operator~$\calH : L^2(\Omega)^d \rightarrow \Z$ as~$\calH(\w) := \w_0$.
%%
An effective way to obtain pressure robustness is to use an inf-sup stable pair $(\Vh, \Qh)$ of spaces for the velocity/pressure with $H(\div, \Omega)$-conforming velocities and such that $\div \Vh \subseteq Q_h$, thus allowing for the ``elimination" of the gradient component of the source term~$\f$ in the computation of the velocities.
The pressure robustness (or lack of it) of several schemes was studied for the first time in~\cite{John_etal:2017}. $H(\div, \Omega)$-conforming spaces have also been used to design pressure-robust hybridizable discontinuous Galerkin (HDG) methods~\cite{Rhebergen_Wells:2018,Lehrenfeld_Schroeberl:2016}.
An alternative approach consists of using~$H(\div,\Omega)$-reconstructions of the test functions (see \cite{linke2014role} and also, for instance, \cite[\S5.2]{John_etal:2017} or \cite{Quiroz_DiPietro:2025,Li_Rui:2023}). 

Another important property for a method is that of being \emph{Reynolds semi-robust}, which means that (assuming a regular solution) the stability and error constants are independent of the Reynolds number~$\mathrm{Re} = \nu^{-1}$, and (possibly) an extra power~$h^{1/2}$ is gained in the convection-dominated regime (i.e., when~$\nu \lesssim h$). This is particularly significant for large Reynolds numbers (or, equivalently, small viscosities~$\nu$), see for instance \cite{Schroeder_etal:2018}.
In order to obtain Reynolds-semi-robustness, several stabilization terms have been proposed, such as streamline-diffusion~\cite{Hansbo_Szepessy:1990,Johson_Saranen:1986}, continuous interior penalty~\cite{Burman_Fernandez:2007}, local projection stabilization~\cite{Ahmed_Matthies:2021,Arndt_Dallmann_Lube:2015,DeFrutos_etal:2019}, and grad-div~\cite{DeFrutos_etal:2018,Garcia_Bosco:2023}.
Unfortunately, most of these stabilization terms are incompatible with the pressure-robustness property. A possible way to overcome this issue is to stabilize the velocity-pressure formulation with stability terms for the vorticity equation~\cite{Ahmed_etal-vorticity:2021,Beirao_Dassi_Vacca:2024,Barrenechea_etal:2024,Garcia_Novo:2024}, which is compatible with the pressure robustness.
On the other hand, $H(\div, \Omega)$-conforming discontinuous Galerkin (DG) methods count with a natural upwind stabilization for the convective term that does not destroy their pressure robustness~\cite{schroeder2018divergence, barrenechea_etal:2020, Han_Hou:2021}.

A final aspect we focus on, which is frequently neglected, is that of the time discretization. Analyses of fully discrete schemes often concentrate on low-order time discretizations~\cite{Ahmed_etal:2017,Garcia_Bosco:2023,Guzman_Shu_Sequeira:2017,Han_Hou:2021}.
On the other hand, DG time stepping can be easily formulated in a variational setting for any degree of approximation in time. 
This allows for an analysis based on variational tools that do not rely on Taylor expansions, so
less regularity of the continuous solution is usually required. Fully discrete schemes based on the DG time stepping have been analyzed in~\cite{Hansbo_Szepessy:1990,Chrysafinos_Walkington:2010,Ahmed_Matthies:2021,Beirao_Dassi_Vacca:2024,Gazca_Kaltenbach:2024,Vexler_Wanger:2024}, but present at least one of the following limitations: \emph{i)} only consider low-order approximations~$(\ell = 0, 1)$, \emph{ii)} stability and error constants depend on negative powers of the viscosity~$\nu$, or \emph{iii)} do no estimate the~$L^{\infty}(0, T; L^2(\Omega)^d)$-error for the velocity.
In fact, the main difficulty in the analysis of the DG time stepping is that standard energy arguments provide control only of the~$L^2(\Omega)^d$ norm of the velocity~$\u$ at the discrete times~$\{t_n\}_{n = 1}^N$, which is not enough to control the energy at all times for high-order approximations.
Instead, we use some nonstandard test functions to get stability in the~$L^{\infty}(0, T; L^2(\Omega)^d)$ norm. 

% --------------------------------------------------
%                   MAIN CONTRIBUTIONS
% --------------------------------------------------

\paragraph{Main contributions.} 
We carry out a robust stability and convergence analysis of a fully discrete scheme that combines~$H(\div, \Omega)$-conforming spatial discretizations and the DG time stepping scheme (for a preliminary work developed for the much simpler setting of a scalar linear equation, comprising both finite and virtual elements, see \cite{Beirao_Dassi_Gomez:2025}). In particular, our estimates are valid for any degree of approximation in time. 
In order to simplify the stability analysis, we approximate the nonlinear convective terms by using a Gauss-Radau interpolant in time as in~\cite{Ahmed_Matthies:2021}. This also allows us to avoid any restriction of the time step~$\tau$ in terms of the spatial meshsize~$h$.

Our first main result is the following continuous dependence of the discrete solution~$\uht$ to a linearized (Oseen's-like) problem on the data (see Proposition~\ref{prop:continuous-dependence} below):
\begin{equation*}
\Tnorm{\uht}{\uht}^2 \lesssim \Norm{\f}{L^1(0, T; L^2(\Omega)^d)}^2 + \Norm{\u_0}{L^2(\Omega)^d}^2, 
\end{equation*}
where the energy norm~$\Tnorm{\cdot}{\uht}$ defined in~\eqref{def:energy-norm} includes the term~$\Norm{\uht}{L^{\infty}(0, T; L^2(\Omega)^d)}$, and the hidden constant depends only on the degree of approximation in time~$\ell$ (so it is independent of the mesh parameters~$\tau$ and~$h$, the final time~$T$, the viscosity constant~$\nu$, and the coefficient of the convective term). 
Such a stability result,  which strongly relies on the choice of nonstandard test functions, is then used in Theorem~\ref{thm:existence-discrete-solutions} to prove the well-posedness of the nonlinear problem.

Our second main result concerns~\emph{a priori} error estimates for the fully discrete scheme. Assuming a sufficiently regular velocity solution $\u$, we obtain a result of the kind (see Theorem \ref{prop:final} and Corollary~\ref{corol:final} below):
\begin{equation*}
\Tnorm{\u - \uht}{\uht}^2 \lesssim \, {\tau^{2 \ell+2}} + (\nu + h) h^{2k} \, ,
\end{equation*}
with $k$ denoting the polynomial degree in space, and where the hidden constant is independent of the pressure variable $p$ and inverse powers of $\nu$ (thus evidencing the pressure- and Reynolds-semi-robustness of the scheme). 
Furthermore, %note that 
for $\nu \lesssim h$, the method enjoys a quicker pre-asymptotic error reduction rate in the spatial parameter $h$.

\red{As an important further development of this contribution, we propose a \emph{novel high-order DG time stepping scheme}, which treats the nonlinear coefficient explicitly (thus significantly improving the efficiency of the method). More precisely, we extend the approximation from the previous time slab to approximate the nonlinear coefficient in the convective term. 
Differently from finite-difference-based time-stepping schemes, this extension is achieved naturally in the time DG setting simply exploiting that the discrete functions in each time slab are polynomials in time (thus being defined on the whole time interval [0,T]). 
Finally, by a suitable modification of the above mentioned results for the original method, we prove that also this semi-implicit version is well posed, pressure robust, and Reynolds semi-robust.}

% --------------------------------------------------
%                   NOTATION
% --------------------------------------------------
\paragraph{Notation.} 
We shall denote by~$\nabla$ the gradient for scalar functions, and by~$\Nabla$ the matrix whose rows contain the componentwise gradients of a vector field. Similarly, $\Delta$ and~$\div$ denote, respectively, the componentwise Laplacian operator and the divergence for vector fields.
Given~$m \in \IZ$, we also denote by~$\dpt^{(m)}$ the~$m$th time derivative. 

Moreover, we shall use standard notation for Sobolev spaces, as well as for their seminorms and norms~\cite{McLean:2000}. Given an open, bounded domain~$\Upsilon \subset \IR^d$ ($d \in \{1, 2, 3\})$, and scalars~$p \in [1, \infty]$ and~$s \in \IR$, we denote by~$W_p^s(\Upsilon)$ the corresponding Sobolev space, and its associated seminorm and norm by~$|\cdot|_{W_p^s(\Upsilon)}$ and~$\|\cdot\|_{W_p^s(\Upsilon)}$, respectively.
When~$p = 2$, we use the notation~$H^s(\Upsilon) := W_2^s(\Upsilon)$, whereas, for~$s = 0$, we denote by~$L^p(\Upsilon) := W^0_p(\Upsilon)$.
The closure of the space~$C_0^{\infty}(\Upsilon)$ in the~$H^1(\Upsilon)$ norm is denoted by~$H_0^1(\Upsilon)$.

Vector fields are identified with boldface, and the notation~$X^d$ is used for the space of functions with~$d$-components in the space~$X$.

Given an interval~$(a, b)$ and a Banach space~$(Z, \Norm{\cdot}{Z})$, we denote by~$L^p(a, b; Z)$ the corresponding Bochner-Sobolev space, whose norm is given by
\begin{equation*}
\Norm{\varphi}{L^p(a, b; Z)} := \Big(\int_a^b \Norm{\varphi(\cdot, t)}{Z}^p \dt \Big)^{1/p} \ , \quad p \in [1,\infty) \, ,
\end{equation*}
with the usual extension for~$p=\infty$.

% --------------------------------------------------
%                   STRUCTURE OF THE PAPER
% --------------------------------------------------
\paragraph{Structure of the paper.}
In Section \ref{sec:method}, we present the variational formulation of the continuous problem and the proposed discrete scheme. In Section \ref{sec:well-posedness}, we prove existence of discrete solutions, together with continuous dependence 
on the data. Afterwards, in Section \ref{sec:conv}, we prove convergence estimates for the method. 
\red{Section~\ref{sec:semi-implicit} is devoted to the description and analysis of the semi-implicit version of the method.}
Finally, in Section~\ref{sec:num}, 
we develop some tests to provide the numerical evidence of our theory described in 
the previous sections. 

% -------------------------------------------------
%           DESCRIPTION OF THE METHOD
% -------------------------------------------------
\section{Variational formulation and description of the method}\label{sec:method}
We first present the continuous weak formulation of model~\eqref{eq:model-problem}. Then, in Section~\ref{subsec:mesh-notation}, we introduce some standard DG notation and our assumptions on the space--time meshes. Finally, in Section~\ref{subsec:discrete-spaces}, we define the discrete spaces involved in the fully discrete formulation presented in Section~\ref{subsec:fully-discrete-formulation}.

In order to introduce the continuous weak formulation of model~\eqref{eq:model-problem}, we define the spaces
\begin{alignat*}{3}
\V & := H_0^1(\Omega)^d, \\
\Q & := L_0^2(\Omega) = \Big\{q \in L^2(\Omega) \, :\, \int_{\Omega} q \dx = 0 \Big\}, \\
H(\Div, \Omega) & := \big\{\v \in L^2(\Omega)^d \, : \, \Div \v \in L^2(\Omega) \big\}, \\
\Z & := \big\{ \v \in H(\div, \Omega) \, :\, \Div \v = 0 \text{ in } \Omega, \text{ and } \v \cdot \bnOmega = 0 \text{ on } \partial \Omega \big\},
\end{alignat*}
and the following continuous forms:
\begin{alignat*}{3}
a(\u, \v) & := (\Nabla \u, \Nabla \v)_{\Omega} & & \qquad \forall (\u, \v) \in \V \times \V, \\
b(\u, q) & := - (q, \Div \u)_{\Omega} & & \qquad \forall (\u, q) \in \V \times \Q, \\
c(\w; \u, \v) &:= ( (\Nabla \u)\w, \v)_{\Omega} & & \qquad \forall (\w, \u, \v) \in \V \times \V \times \V,
\end{alignat*}
where~$(\cdot, \cdot)_{\Omega}$ denotes the~$L^2(\Omega)$ inner product.

For~$\f \in L^2(\QT)^d$ and~$\uo \in \Z$, the continuous weak formulation of model~\eqref{eq:model-problem} reads: 
find~$\u \in L^2(0, T; \V)$ with~$\dpt \u \in L^2(0, T; \V^*)$, and~$p \in L^2(0, T; \Q)$, such that
\begin{subequations}
\label{eq:weak-formulation}
\begin{alignat}{3}
\int_0^T \big( \langle\dpt \u, \v \rangle + \nu a(\u, \v) + c(\u; \u, \v) + b(\v, p) \big) \dt  & = \int_0^T (\f, \v)_{\Omega} \dt  & & \qquad \forall \v \in L^2(0, T; \V), \\
\displaystyle \int_0^T b(\u, q) \dt & = 0 & & \qquad \forall q \in L^2(0, T; \Q),  \\
\u(\cdot, 0) & = \uo & & \qquad \text{ in } \Z,
\end{alignat}
\end{subequations}
where~$\V^* := H^{-1}(\Omega)^d$, and~$\langle \cdot,\, \cdot \rangle$ denotes the duality pairing between~$\V^*$ and~$\V$.

For~$d = 2$ and the above assumptions on~$\f$ and~$\uo$, if the spatial domain is sufficiently smooth, there exists a unique solution~$(\u, p)$ to~\eqref{eq:weak-formulation} with the following regularity:
\begin{equation}
\label{eq:regularity-weak-solution}
\u \in C^0([0, T]; \Z) \cap L^2(0, T; H^2(\Omega)^d) \cap H^1(0, T; L^2(\Omega)^d) \quad \text{ and } \quad p \in L^2(0, T; H^1(\Omega)),
\end{equation}
whereas, for~$d = 3$, there exists a final time~$T^* > 0$ depending on the data of the problem such that there is a unique solution with the regularity in~\eqref{eq:regularity-weak-solution} but replacing~$T$ by~$T^*$ (see, e.g., \cite[Thm. V.2.1 in Ch. V]{Boyer_Fabrie:2013}).

%\todo{Continuous weak formulation in kernel form is missing...}

% ------------------------------------------------
%                   NOTATION
% ------------------------------------------------
\subsection{Mesh and DG notation\label{subsec:mesh-notation}}
Let~$\{\Th\}_{h > 0}$ be a family of shape-regular conforming simplicial meshes for the spatial domain~$\Omega$. We define the meshsize~$h := \max\{\hK\, :\, K \in \Th\}$, where~$\hK := \diam(K)$,  and denote by $\rho$ the shape regularity parameter of $\{\Th\}_{h > 0}$. We denote the set of facets of~$\Th$ by~$\Fh = \FhI \cup \FhD$, where~$\FhI$ and~$\FhD$ are the sets of interior and boundary facets of~$\Th$, respectively.
For any facet~$F \in \FhI$, we set~$h_F$ as its diameter and~$\bnF$ as one of the two unit normal vectors orthogonal to~$F$ (with the convention that, if~$F \in \FhD$, then~$\bnF$ points outwards of $\Omega$). 
Whenever needed, we will also consider the above diameters $h_K,\, h_F$ as piecewise constant functions living on the set of elements or facets, respectively.

Let also~$\Tt$ be a partition of the time interval~$(0, T)$ of the form~$0:= t_0 < t_1 < \ldots < t_N := T$. For~$n = 1, \ldots, N$, we define the time interval~$\In := (\tnmo, \tn)$, the time step~$\tau_n := \tn - \tnmo$, the partial cylinder~$\Qn := \Omega \times \In$, and the surface~$\Sn := \Omega \times \{\tn\}$. Moreover, for~$n = 1, \ldots, N - 1$ and piecewise smooth (scalar or vector) functions, we define the time jumps
\begin{equation*}
\jump{z}_n(\bx) := z(\bx, \tn^-) - z(\bx, \tn^+) \qquad \forall \bx \in \Omega,
\end{equation*}
where
\begin{equation*}
z(\cdot, \tn^-) := \lim_{\varepsilon \to 0^+} z(\cdot, \tn - \varepsilon) \quad \text{ and } \quad z(\cdot, \tn^+) := \lim_{\varepsilon \to 0^+} z(\cdot, \tn + \varepsilon).
\end{equation*}

For convenience, we introduce the following notation for the \emph{time-like} facets:
\begin{alignat*}{3}
\FhItimen & := \bigcup \Big\{F \times \In \, :\, F \in \FhI\Big\},  & & 
\qquad \FhItime & := \bigcup_{n = 1}^N \FhItimen, \\
\FhDtimen & := \bigcup \Big\{F \times \In \, :\, F \in \FhD \Big\}, & & 
\qquad \FhDtime & := \bigcup_{n = 1}^N \FhDtimen, \\
\Fhtimen &:= \FhItimen \cup \FhDtimen,  & & 
\qquad \Fhtime &:= \FhItime \cup \FhDtime.
\end{alignat*}

Standard notation is used for the spatial average~$\mvl{\cdot}$ and spatial jumps~$\jump{\cdot}$ of piecewise scalar-valued~$(q)$ or vector-valued~$(\v)$ functions: for any interior facet~$F \in \FhI$ shared by two elements~$\Kp$ and~$\Km$ in~$\Th$ with~$\bnF$ pointing outwards of~$\Kp$, we have
\begin{alignat*}{3}
\mvl{q}_F & := \frac12 (q_{|_{\Kp}} + q_{|_{\Km}} ) \quad  & & \text{ and } \quad 
& & \jump{q}_{F} := q_{|_{\Kp}} - q_{|_{\Km}}, \\
%%%
\mvl{\v}_F & := \frac12 \big(\v_{|_{\Kp}} + \v_{|_{\Km}} \big) \quad & & \text{ and } \qquad & & \jump{\v}_{F} := \v_{|_{\Kp}} - \v_{|_{\Km}},
\end{alignat*}
and, for any boundary facet~$F \in \FhD$, we set
\begin{equation*}
\mvl{q}_F = \jump{q}_{F} = q \qquad \text{ and } \qquad \mvl{\v}_F = \jump{\v}_{F} = \v.
\end{equation*}

% ------------------------------------------------
%       DISCRETE SPACES
% ------------------------------------------------
\subsection{Discrete spaces\label{subsec:discrete-spaces}}
Given a degree of approximation in space~$k \in \IN$ with~$k \geq 1$, we denote by~$\RTk(\Th)$ and~$\BDMk(\Th)$ the Raviart--Thomas and the Brezzi--Douglas--Marini elements of order~$k$, respectively. Moreover, we denote by~$\Ppo{k}{\Th}$ the space of piecewise polynomials of degree at most~$k$ defined on~$\Th$ with zero mean over~$\Omega$. 

We set the discrete spaces~$\Vh \subset H(\div, \Omega)$ and~$\Qh \subset \Q$ as either (see, for instance, \cite{Boffi_Brezzi_Fortin:2013, Ern_Guermond-I:2021}):
\begin{equation}\label{space-choice}
\Vh = \RTk(\Th) \text{ and } \Qh = \Ppo{k}{\Th} \qquad \text{ or } \qquad \Vh = \BDMk(\Th) \text{ and } \Qh = \Ppo{k-1}{\Th}.
\end{equation}
Moreover, given a degree of approximation in time~$\ell \in \IN$, we define the following space--time discrete spaces:
\begin{alignat*}{3}
\Vht & := \big\{ \v \in L^2(0, T; H(\div, \Omega)) \, : \, \v_{|_{\SD}}{\cdot {\bf n}_\Omega} = 0, \text{ and } \v_{|_{\Qn}} \in \Pp{\ell}{\In} \otimes \Vh \text{ for }   n = 1, \ldots, N \big\}, \\
\Zht & := \big\{ \vht \in \Vht \, : \, \div \vht {}_{|_{\Qn}} = 0 \text{ for } n = 1, \ldots, N \big\}, \\
\Qht & := \big\{ q \in L^2(0, T; \Q) \, : \, q_{|_{\Qn}} \in \Pp{\ell}{\In} \otimes \Qh \text{ for } n = 1, \ldots, N \big\},
\end{alignat*}
where ${\bf n}_\Omega$ denotes the unitary outward normal to $\partial \Omega$. 

Let~$H^1(\Tt)$ and~$\Pp{\ell}{\Tt}$ be, respectively, the spaces of piecewise~$H^1$ and~$\mathbb{P}_{\ell}$ functions defined on~$\Tt$. For~$n = 1, \ldots, n$, we denote by~$\{(\omega_i^{(n)}, s_i^{(n)})\}_{i = 1}^{\ell + 1}$ the left-sided Gauss-Radau quadrature rule in the interval~$\In$ with nodes~$\tnmo =: s_1^{(n)} < s_2^{(n)} < \ldots < s_{\ell + 1}^{(n)} < \tn$, and positive weights~$\{\omega_i^{(n)}\}_{i = 1}^{\ell + 1}$.
This quadrature rule is exact for polynomials of degree less than or equal to~$2\ell$.

We denote by~$\ItR : H^1(\Tt) \to \Pp{\ell}{\Tt}$ the broken Lagrange interpolant at the Gauss-Radau nodes, and introduce the following notation for the use of the Gauss-Radau quadrature rule on the time interval~$\In$:
\begin{equation}
\label{eq:Gauss-Radau-integration}
\QtRn(u) := \sum_{i = 1}^{\ell + 1} \omega_i^{(n)} u(s_i^{(n)}) \approx \int_{\In} u(t) \dt.
\end{equation}
The following properties can be easily deduced from the exactness of this quadrature rule and the definition of the interpolant~$\ItR$:
\begin{subequations}
\begin{alignat}{3}
\label{eq:property-interpolant-time-1}
\big(\ItR u, w_{\tau} \big)_{\In} & = \QtRn(u w_{\tau}) & & \qquad \text{ for all } u \in H^1(\In) \text{ and } w_{\tau} \in \Pp{\ell}{\In}, \\
\label{eq:property-interpolant-time-2}
\big(\ItR u_{\tau}, w_{\tau} \big)_{\In} & = \QtRn(u_{\tau} w_{\tau}) = (u_{\tau}, w_{\tau})_{\In} & & \qquad \text{ for all } u_{\tau} , w_{\tau} \text{ such that } u_{\tau} w_{\tau} \in \Pp{2 \ell}{\In}, \\
\label{eq:property-interpolant-time-3}
\big(\ItR u, \ItR w \big)_{\In} & = \QtRn(u w)  & & \qquad \text{ for all } u, w \in H^1(\In).
\end{alignat}
\end{subequations}

For the sake of simplicity, in what follows, we may omit the dependence on the time interval when no confusion arises.
Moreover, for any Banach space~$(Z, \Norm{\cdot}{Z})$ and $p \in (1,\infty]$, we still denote by~$\ItR : W_p^1(\Tt; Z) \to P_{\ell} \otimes Z$ the trivial extension of \eqref{eq:Gauss-Radau-integration} with $u(\cdot, s_i^{(n)}) \in Z$. Note that, if~$Z \subseteq L^1(\Omega)$, properties~\eqref{eq:property-interpolant-time-1}--\eqref{eq:property-interpolant-time-3} will hold pointwise in space almost everywhere in~$\Omega$.
%as applied pointwise in space.

% ------------------------------------------------
%       METHOD
% ------------------------------------------------
\subsection{Fully discrete space--time formulation
\label{subsec:fully-discrete-formulation}}
Let~$\sigma$ be a sufficiently large stability parameter %positive constant 
as in standard %IP
interior penalty-DG schemes, and~$\gamma$ be 
the piecewise-constant (in space) function %$\gamma$ 
living in the set of facets defined as %is defined by
\begin{equation}
\label{def:gamma}
\gamma(\wht)_{|_F} := \gamma_F(\wht) := \max{ \{c_S, \| \wht \cdot \bnF \|_{L^\infty(F)} \}} \qquad \forall F \in \Fh \, ,
\end{equation}
for any discrete function~$\wht \in \Vht$, and with $c_S$ a small positive ``safeguard'' constant.
We introduce the following space--time forms:
\begin{subequations}
\label{def:discrete-space-time-forms}
\begin{alignat}{3}
\label{def:mht}
\mht(\uht, \vht) & := \sum_{n = 1}^N  (\dpt \uht, \vht)_{\Qn}  - \sum_{n = 1}^{N - 1} \big(\jump{\uht}_n, \vht(\cdot, \tn^+) \big)_{\Omega} + (\uht, \vht)_{\SO}, \\
%%%
\nonumber
\aht(\uht, \vht) & := \sum_{n = 1}^N \aht^{(n)}(\uht, \vht) \\
\nonumber
& := \sum_{n = 1}^N \Big[ (\Nablah \uht, \Nablah \vht)_{\Qn} - (\mvl{\Nablah \uht} \bnF, \jump{\vht})_{\Fhtimen} \\
\label{def:aht}
& \qquad - (\jump{\uht}, \mvl{\Nabla \vht} \bnF)_{\Fhtimen} + (\sigma h_F^{-1} \jump{\uht}, \jump{\vht}\big)_{\Fhtimen} \Big], \\
%%%
\nonumber
\cht(\wht; \uht, \vht) & := \sum_{n = 1}^N \cht^{(n)}(\wht; \uht, \vht) \\
\nonumber
& := \sum_{n = 1}^N \Big[ \Big(\ItR \big((\Nabla \uht) \wht \big), \vht \Big)_{\Qn} - \Big(\ItR \big((\wht \cdot \bnF) \jump{\uht} \big), \mvl{\vht} \Big)_{\FhItimen} \\
\label{def:cht}
& \quad + \frac12 \Big(\ItR\big( \gamma(\wht) \jump{\uht} \big), \jump{\vht} \Big)_{\FhItimen}\Big],  \\
%%%
\label{def:bht}
\bht(\uht, \qht) &:= -(\Div \uht, \qht)_{\QT},
\end{alignat}
\end{subequations}
for all~$\wht, \uht, \vht \in \Vht$ and~$\qht \in \Qht$.

% \medskip
% \centerline{VERSIONE DA TESTARE:}
% Calcolo di $\cht(\wht; \uht, \vht)$.
%
%$$
%- \Big(\big((\Nabla \vht) \wht \big), \uht \Big)_{\textrm{elementi}} + \Big(\big((\wht \cdot \bnF) \jump{\vht} \big), \mvl{\uht} \Big)_{\textrm{facce interne}}
%$$

The proposed space--time DG formulation reads: find~$\uht \in \Vht$ and~$\qht \in \Qht$, such that
\begin{subequations}
\label{eq:space-time-formulation}
\begin{alignat}{3}
\nonumber
\mht(\uht, \vht) & + \nu \aht(\uht, \vht) + \cht(\uht; \uht, \vht)   \\
\label{eq:space-time-formulation-1}
   & + \bht(\vht, \pht) = (\f, \vht)_{\QT} + (\uo, \vht(\cdot, 0))_{\Omega} & \qquad \forall \vht \in \Vht, \\
 %%%
 \label{eq:space-time-formulation-2}
& \quad \bht(\uht, \qht) = 0 & \qquad \forall \qht \in \Qht.
\end{alignat}
\end{subequations}

Since the space~$\Vht$ is a subspace of~$L^2(0, T; H(\div, \Omega))$ and~$\Div \Vht \subset \Qht$, we can take~$\qht = \Div \uht$ in~\eqref{eq:space-time-formulation-2} and deduce that~$\Div \uht = 0$. 
Therefore, the space--time formulation~\eqref{eq:space-time-formulation} can be written in the following kernel formulation: find~$\uht \in \Zht$ such that
\begin{equation}
\label{eq:kernel-space-time-problem}
\begin{split}
\Bht(\uht; \uht, \zht) := \mht(\uht, \zht) & + \nu \aht(\uht, \zht) 
 + \cht(\uht; \uht, \zht)  \\
 & = (\f, \zht)_{\QT} + (\uo, \zht(\cdot, 0))_{\Omega} \qquad \forall \zht \in \Zht.
\end{split}
\end{equation}

\begin{remark}[Use of the interpolant~$\ItR$]\label{rem:integr}
Recalling the simple identities \eqref{eq:property-interpolant-time-1} and~\eqref{eq:property-interpolant-time-2} it is immediate to check that (i) the particular interpolation~$\ItR$ appearing in the definition of~$\cht(\cdot;\cdot,\cdot)$ corresponds to using a Gauss-Radau quadrature rule in time, and (ii) all the other terms on the left-hand side in~\eqref{eq:space-time-formulation} can be exactly calculated by the same rule. Therefore, the above interpolation~$\ItR$ also allows for an actual simplification of the code.
\eremk
\end{remark}

% -----------------------------------------------------
%                   STABILITY
% -----------------------------------------------------
\section{Existence of discrete solutions}\label{sec:well-posedness}
This section is devoted to proving the existence of discrete solutions to the space--time formulation~\eqref{eq:space-time-formulation}.
In order to do so, we consider the following linearized problem: given~$\wht \in \Zht$, find~$\uht \in \Zht$, 
 %and~$\pht \in \Qht$, 
such that
\begin{equation}
\label{eq:kernel-linearized-problem}
\begin{split}
\Bht(\wht; \uht, \zht) = \mht(\uht, \zht) & + \nu \aht(\uht, \zht) 
 + \cht(\wht; \uht, \zht)  \\
 & = (\f, \zht)_{\QT} + (\uo, \zht(\cdot, 0))_{\Omega} \qquad \forall \zht \in \Zht.
\end{split}
\end{equation}

In Section~\ref{subsec:useful-results}, we deploy some theoretical tools for the stability analysis. Then, in Section~\ref{subsec:weak-partial-bound}, we present the partial bound for~$\uht$ obtained using standard energy arguments, which is later used in Section~\ref{sect:continuous-dependence} to derive a continuous dependence of the discrete solution to~\eqref{eq:kernel-linearized-problem} on the data.
Such a stability estimate is combined with a fixed-point argument to prove the existence of discrete solutions to~\eqref{eq:space-time-formulation} in Section~\ref{subsec:fixed-point}.

% -----------------------------------------------------
%           PRELIMINARY
% -----------------------------------------------------
\subsection{Some useful results for the stability analysis\label{subsec:useful-results}}
\paragraph{Inverse estimates.} 
We now extend a standard polynomial inverse estimate in one dimension (see, e.g., \cite[Lemma 4.5.3]{Brenner-Scott:book}) to tensor-product spaces. This result will be used in the forthcoming stability and convergence analysis.
\begin{lemma}
\label{lemma:L2-Linfty}
For any Banach space~$(Z, \Norm{\cdot}{Z})$ with~$Z \subseteq L^1(\Omega)$, there exists a positive constant~$\Cinv$ %independent of~$\tau$ 
depending only on~$\ell$ such that, for~$n = 1, \ldots, N$, it holds
\begin{equation}
\label{eq:inverse-estimate-Linf-L2}
\Norm{w_{\tau}}{L^{\infty}(\In; Z)}^2 \le \Cinv \tau_n^{-1} \Norm{w_{\tau}}{L^2(\In; Z)}^2 \qquad \forall w_{\tau} \in \Pp{\ell}{\Tt} \otimes Z.
\end{equation}
\end{lemma}
\begin{proof}
Let~$w_{\tau} \in \Pp{\ell}{\Tt} \otimes Z$ for some Banach space~$(Z, \Norm{\cdot}{Z})$, and~$n \in \{1, \ldots, N\}$. Denoting by~$\{L_i\}_{i = 0}^{\ell}$ the Legendre polynomials in~$\In$, the function $w_{\tau}$ can be written in its Legendre expansion as follows
\begin{equation*}
w_{\tau}(\cdot, t) = \sum_{i = 0}^{\ell} \frac{L_i(t)}{\Norm{L_i}{L^2(\In)}^2}  \int_{\In} w_{\tau}(\cdot, s) L_i(s) \ds \qquad \forall t \in \In.
\end{equation*}
Since all norms are convex, the triangle and the Jensen inequalities, the uniform bound~$\Norm{L_i}{L^{\infty}(\In)} = 1$, and the identity~$\Norm{L_i}{L^2(\In)}^2 = 2\tau_n/(2i + 1)$ give
\begin{equation*}
\Norm{w_{\tau}(\cdot, t)}{Z} \le \sum_{i = 0}^{\ell} \frac{2i+1}{2\tau_n} \int_{\In}  \Norm{w_{\tau}(\cdot, s)}{Z} |L_i(s)| \ds \qquad \forall t \in \In,
\end{equation*}
and, using the Cauchy--Schwarz inequality, we obtain
\begin{equation}
\label{eq:bound-wt-Z}
\begin{split}
\Norm{w_{\tau}(\cdot, t)}{Z} & \le \Norm{w_{\tau}}{L^2(\In; Z)}  \sum_{i = 0}^{\ell} \frac{2i+1}{2\tau_n} \Norm{L_i}{L^2(\In)} \\
& \le \Big(\sum_{i = 0}^{\ell}  \sqrt{\frac{2i + 1}{2}}\Big) \tau_n^{-\frac12} \Norm{w_{\tau}}{L^2(\In; Z)} 
\le \frac{(\ell + 1)^{3/2}}{\sqrt{2\tau_n}} \Norm{w_{\tau}}{L^2(\In; Z)} \qquad \forall t \in \In.
\end{split}
\end{equation}
The inverse estimate~\eqref{eq:inverse-estimate-Linf-L2} then follows by choosing~$t \in [\tnmo, \tn]$ as the value where the left-hand side of~\eqref{eq:bound-wt-Z} takes its maximum value.
\end{proof}

\paragraph{Properties of the discrete space--time forms.} We now discuss some identities resulting from the definition of the space--time forms in~\eqref{def:discrete-space-time-forms}, and the properties of the spaces~$\Vht$, $\Zht$, and~$\Qht$.

We introduce the time-jump functional
\begin{equation}
\label{def:jump-term}
\SemiNorm{\vht}{\sf J}^2 := \frac12 \Big(\Norm{\vht}{L^2(\ST)^d}^2 + \sum_{n = 1}^{N - 1} \Norm{\jump{\vht}_{n}}{L^2(\Omega)^d}^2 + \Norm{\vht}{L^2(\SO)^d}^2\Big) \qquad \forall \vht \in \Vht.
\end{equation}
Integration by parts in time and the jump identity
\begin{equation*}
%\label{eq:jump-identity}
\frac12 \jump{w^2}_n - w(\cdot, \tn^+) \jump{w}_n = \frac12 \jump{w}_n^2, \qquad n = 1, \ldots, N - 1,
\end{equation*}
lead to
\begin{equation}
\label{eq:identity-mht}
\mht(\vht, \vht) = \SemiNorm{\vht}{\sf J}^2 \qquad \forall \vht \in \Vht.
\end{equation}

In addition, the following norm in the space~$\Vht$ is induced by the bilinear form~$\aht(\cdot, \cdot)$:
\begin{equation}
\label{def:A-norm}
\Norm{\vht}{\calA}^2 := \sum_{n = 1}^N \int_{\In} \Norm{\vht(\cdot, t)}{\calA, h}^2 \dt \qquad \forall \vht \in \Vht,
\end{equation}
where
\begin{equation*}
\Norm{\vh}{\calA, h}^2 := \Norm{\Nablah \vh}{L^2(\Omega)^{d\times d}}^2 + \sum_{F \in \Fh} \int_{F}  \sigma h_F^{-1} |\jump{\vh}|^2 \dS \qquad \forall \vh \in \Vh.
\end{equation*}
More precisely, if the penalty parameter~$\sigma$ is large enough, there exists a strictly positive constant~$\cA$ independent of~$h$ such that (see e.g., \cite[\S6.1.2.1]{DiPietro-Ern:2012})
\begin{equation}
\label{eq:coercivity-aht}
\aht(\vht, \vht) \geq \cA \Norm{\vht}{\calA}^2 \qquad \forall \vht \in \Vht.
\end{equation}

Defining also the auxiliary norm
\begin{equation*}
% \label{eq:A-star-norm}
\Norm{\v}{\calA_*}^2 := \sum_{n = 1}^N \int_{\In} \Norm{\v(\cdot, t)}{\calA_*, h}^2 \dt \quad \forall \v \in \Vstar + \Vht,
\end{equation*}
where~$\Vstar := L^2(0, T; H_0^1(\Omega)^d \cap H^2(\Omega)^d)$, and
\begin{equation}
\label{def:A*-norm}
\Norm{\v}{\calA_{*}, h}^2 := \Norm{\v}{\calA, h}^2 + \sum_{K \in \Th} \hK \Norm{\Nabla \v_{|_K} \bn_K}{L^2(\partial K)^d}^2,
\end{equation}
there exists a positive constant~$C_{\calA_*}$ independent of~$h$ such that (see~\cite[Lemma 4.16 in~\S4.2.3.2]{DiPietro-Ern:2012})
\begin{equation}
\label{eq:continuity-aht}
\aht(\v, \wht) \le C_{\calA_*} \Norm{\v}{\calA_*} \Norm{\wht}{\calA} \qquad \forall \v \in (\Vstar + \Vht), \wht \in \Vht.
\end{equation}

Given~$\wht \in \Zht$, we define the following upwind-in-space functional
\begin{equation}
\label{def:gamma-norm}
\SemiNorm{\vht}{\gamma, \wht}^2 := \frac12 \sum_{F \in \FhI} \QtR\Big( \gamma_F(\wht) \|\jump{\vht}\|_{L^2(F)^d}^2 \Big)
%\frac12 \sum_{F \in \FhI} \QtR\Big( \int_F  \gamma_F(\wht) \, |\jump{\vht}|^2 \dS \Big)  
\qquad \forall \vht \in \Vht.
\end{equation}

Using property~\eqref{eq:property-interpolant-time-1} of the interpolant~$\ItR$, integration by parts in space, the fact that functions in~$\Vht$ have single-valued normal components, and the average-jump identity
\begin{equation*}
%\label{eq:average-jump}
\jump{\u \cdot \v} = \jump{\u} \cdot \mvl{\v} + \jump{\v} \cdot \mvl {\u},
\end{equation*}
for all~$(\wht, \uht) \in \Zht \times \Vht$, we have
\begin{alignat}{3}
\nonumber
\cht(\wht; \uht, \uht) & = \sum_{n = 1}^N \Big(\ItR \big((\Nabla \uht) \wht \big), \uht \Big)_{\Qn} - \Big(\ItR \big((\wht \cdot \bnF) \jump{\uht} \big), \mvl{\uht} \Big)_{\FhItime} \\
\nonumber
& \quad + \frac12 \Big(\ItR\big( \gamma(\wht) \jump{\uht} \big), \jump{\uht} \Big)_{\FhItime} \\
%%%
\nonumber
& = \sum_{n = 1}^N \QtR\Big( \big((\Nabla \uht) \wht,  \uht \big)_{\Omega} \Big) - \QtR\Big( \sum_{F \in \FhI} \int_F (\wht \cdot \bnF) \jump{\uht} \cdot \mvl{\uht} \dS \Big) \\
\nonumber
& \quad + \frac12 \QtR \Big(\sum_{F \in \FhI} \int_F  \gamma_F(\wht) |\jump{\uht}|^2 \dS \Big)\\
%%%
\label{eq:identity-cht}
& = \SemiNorm{\uht}{\gamma, \wht}^2.
\end{alignat}

\paragraph{Auxiliary weight functions.} 
We introduce some linear weight functions, which have been used in~\cite{Walkington:2014,Dong-Mascotto-Wang:2024,Gomez-Nikolic:2024} to prove continuous dependence on the data in~$L^{\infty}(0, T; X)$-type norms for linear and quasilinear wave problems. 

For~$n = 1, \ldots, N$, we define
\begin{equation}\label{lambdadef}
\varphi_n(t) := 1 - \lambda_n (t - \tnmo) \quad \text{ with } \lambda_n := \frac{1}{2\tau_n}.
\end{equation}
These functions satisfy the following uniform bounds:
\begin{subequations}
\label{eq:bounds-varphi}
\begin{alignat}{3}
\label{eq:bounds-varphi-1}
\frac12 & \le \varphi_n(t)  \le 1 & & \qquad \forall t \in [\tnmo, \tn], \\
\label{eq:bounds-varphi-2}
\varphi_n'(t) & = - \lambda_n & & \qquad \forall t \in [\tnmo, \tn].
\end{alignat}
\end{subequations}

\begin{lemma}
\label{lemma:bilinear-form-varphi}
Let~$\alpha_h(\cdot, \cdot)$ be a symmetric and coercive bilinear form on~$\Vh$. For~$n = 1, \ldots, N$, it holds
\begin{alignat*}{3}
\int_{\In} \alpha_h\big(\vht(\cdot, t), \ItR(\varphi_n \vht)(\cdot, t) \big)\dt \geq \frac12 \int_{\In} \alpha_h\big(\vht(\cdot, t), \vht(\cdot, t)\big) \dt \qquad \forall \vht \in \Vht.
\end{alignat*}
\end{lemma}
\begin{proof}
Let~$\vht \in \Vht$ and~$n \in \{1, \ldots, N\}$. We use the exactness of the Gauss-Radau quadrature rule, the positivity of the weights~$\{\omega_i^{(n)}\}_{i = 1}^{\ell + 1}$, and the uniform bound in~\eqref{eq:bounds-varphi-1} for~$\varphi_n$ to obtain
\begin{alignat*}{3}
\int_{\In} \alpha_h\big(\vht(\cdot, t), \ItR(\varphi_n \vht)(\cdot, t) \big)\dt & = \sum_{i = 1}^{\ell + 1} \omega_{i}^{(n)} \varphi_{n}(s_i^{(n)}) \alpha_h\big(\vht(\cdot, s_i^{(n)}),  \vht(\cdot, s_i^{(n)}) \big) \\
& \geq \frac12 \sum_{i = 1}^{\ell + 1} \omega_{i}^{(n)} \alpha_h\big(\vht(\cdot, s_i^{(n)}),  \vht(\cdot, s_i^{(n)}) \big) \\
& = \frac12 \int_{\In} \alpha_h\big(\vht(\cdot, t), \vht(\cdot, t) \big) \dt,
\end{alignat*}
which completes the proof.
\end{proof}

\begin{lemma}[Stability of~$\ItR$]
\label{lemma:stab-ItR}
For any Banach space~$(Z, \Norm{\cdot}{Z})$ with~$Z \subseteq L^1(\Omega)$, there exists a positive constant~$\CSI$ depending only on~$\ell$ such that, for~$n = 1, \ldots, N$, it holds
\begin{equation*}
\Norm{\ItR w%(\varphi_n w_{\tau})
}{L^{\infty}(\In; Z)} \le \CSI \Norm{w}{L^{\infty}(\In; Z)} \qquad \forall w_{\tau} \in %\Pp{\ell}{\Tt} 
H^1(\In; Z).
\end{equation*}
\end{lemma}
\begin{proof}
Let~$(Z, \Norm{\cdot}{Z})$ be a Banach space, and~$n \in \{1, \ldots, N\}$. Denoting by~$\{\mathcal{L}_i\}_{i=1}^{\ell + 1}$ the Lagrange polynomials associated with the left-sided Gauss-Radau nodes~$\{s_i^{(n)}\}_{i = 1}^{\ell + 1}$, we have
\begin{equation*}
{\ItR w} (\cdot, t) = \sum_{i = 1}^{\ell + 1} % \varphi(s_i) 
w(\cdot, s_i) \mathcal{L}_i(t) \qquad \forall t \in \In.
\end{equation*}
By the triangle inequality, we have
\begin{equation}
\label{eq:aux-bound-ItR}
\Norm{\ItR w(\cdot, t)}{Z} \le \Norm{w}{L^{\infty}(\In; Z)} \sum_{i = 1}^{\ell + 1} |\mathcal{L}_i(t)| \qquad \forall t \in \In,
\end{equation}
and taking~$t$ as the value where the left-hand side of~\eqref{eq:aux-bound-ItR} achieves its maximum value, we get
\begin{equation*}
\Norm{\ItR w}{L^{\infty}(\In; Z)} \le \Lambda_{\ell + 1} \Norm{w}{L^{\infty}(\In; Z)},
\end{equation*}
where~$\Lambda_{\ell + 1}$ is the Lebesgue constant for the nodes~$\{s_i\}_{i = 1}^{\ell + 1}$, which grows as~$\mathcal{O}(\sqrt{\ell + 1})$ (see~\cite[Thm.~5.1]{Hager_Lebesgue:2017}).
%from the stability of the interpolant~$\ItR$ in the norm~$L^{\infty}(\In)$ (see, e.g., \cite[Thm.~4.4.20]{Brenner-Scott:book}) and the uniform bound in~\eqref{eq:bounds-varphi-1} for~$\varphi_n$.
\end{proof}
% -----------------------------------------------------
%           PARTIAL BOUND
% -----------------------------------------------------
\subsection{Weak partial bound of the discrete solution\label{subsec:weak-partial-bound}}
Standard techniques provide $L^2(\Omega)$-bounds of discrete solutions to~\eqref{eq:kernel-linearized-problem} only at the discrete times~$\{\tn\}_{n = 0}^N$, which is not enough to bound the solution at all times for high-order approximations. 
In next lemma, we show the \emph{weak partial bound} obtained using standard techniques. In Section~\ref{sect:continuous-dependence} below, we use Lemma~\ref{lemma:weak-bound} to prove continuous dependence of any discrete solution to~\eqref{eq:kernel-linearized-problem} on the data.
%, in a norm including the term~$L^{\infty}(0, T; L^2(\Omega)^d)$.}

\begin{lemma}[Weak partial bound of the discrete solution\label{lemma:weak-bound}]
For~$n = 1, \ldots, N$ and a sufficiently large penalty parameter~$\sigma$, any solution~$\uht \in \Zht$ to the linearized problem~\eqref{eq:kernel-linearized-problem} satisfies
\begin{alignat*}{3}
\frac12 \Norm{\uht(\cdot, \tn^{-})}{L^2(\Omega)^d}^2 & + \nu \sum_{m = 1}^n \int_{I_m} \Norm{\uht(\cdot, t) }{\calA, h}^2 \dt  + \frac12 \sum_{m = 1}^{n - 1} \Norm{\jump{\uht}_{m}}{L^2(\Omega)^d}^2 + \frac14 \Norm{\uht(\cdot, 0)}{L^2(\Omega)^d}^2  \\
& + \frac12 \sum_{m = 1}^{n} \sum_{F \in \FhI} \QtRm\Big( \int_F  \gamma_F(\wht) \, |\jump{\uht}|^2 \dS \Big) \\
& \le \Norm{\u_0}{L^2(\Omega)^d}^2 + \Norm{\f}{L^1(0, \tn; L^2(\Omega)^d)} \Norm{\uht}{L^{\infty}(0, \tn; L^2(\Omega)^d)}.
\end{alignat*}
\end{lemma}
\begin{proof}
Without loss of generality, we show the result for the case~$n = N$.

Taking~$\zht = \uht$ in~\eqref{eq:kernel-linearized-problem}, we obtain
\begin{equation}
\label{eq:aux-identity-weak-bound}
\begin{split}
\mht(\uht, \uht) + \nu \aht(\uht, \uht) 
 & + \cht(\wht; \uht, \uht)  \\
 & = (\f, \uht)_{\QT} + (\uo, \uht(\cdot, 0))_{\Omega} \qquad \forall \zht \in \Zht.
\end{split}
\end{equation}

Using properties~\eqref{eq:identity-mht}, \eqref{eq:coercivity-aht}, and~\eqref{eq:identity-cht} of the forms~$\mht(\cdot, \cdot)$, $\aht(\cdot, \cdot)$, and~$\cht(\cdot; \cdot, \cdot)$, the left-hand side of~\eqref{eq:aux-identity-weak-bound} can be bounded from below as follows:
\begin{equation}
\label{eq:left-side-weak-bound}
    \begin{split}
    \mht(\uht, \uht) + \nu \aht(\uht, \uht) & + \cht(\wht; \uht, \uht) \\
    & \geq \SemiNorm{\uht}{\sf J}^2 + \nu \cA \Norm{\uht}{\calA}^2 + \SemiNorm{\uht}{\gamma, \wht}^2.
    \end{split}
\end{equation}

The following bound of the right-hand side of~\eqref{eq:aux-identity-weak-bound} follows by using the H\"older and the Young inequalities:
\begin{equation}
\label{eq:right-side-weak-bound}
\begin{split}
(\f, \uht)_{\QT} & + (\u_0, \uht(\cdot, 0) )_{\Omega} \\
& \le \Norm{\f}{L^1(0, T; L^2(\Omega)^d)} \Norm{\uht}{L^{\infty}(0, T; L^2(\Omega)^d)} + \Norm{\u_0}{L^2(\Omega)^d}^2 + \frac14 \Norm{\uht}{L^2(\SO)^d}^2\, .
\end{split}
\end{equation}

Inserting~\eqref{eq:left-side-weak-bound} and~\eqref{eq:right-side-weak-bound} in~\eqref{eq:aux-identity-weak-bound}, we obtain the desired result.
\end{proof}

% ----------------------------------------------------
%               CONTINUOUS DEPENDENCE
% ----------------------------------------------------
\subsection{Uniform continuous dependence on the data of the discrete solution \label{sect:continuous-dependence}}
Henceforth, we use~$a \lesssim b$ to indicate the existence of a positive constant~$C$ independent of the meshsize~$h$, the maximum time step~$\tau$, and the viscosity constant~$\nu$ such that~$a \le C b$. 

We define the following energy norm in the space~$\Vht$:
\begin{equation} 
\label{def:energy-norm}
 \Tnorm{\uht}{\wht}^2 := \Norm{\uht}{L^{\infty}(0, T; L^2(\Omega)^d)}^2 
 + \nu \Norm{\uht}{\cal A}^2 + \SemiNorm{\uht}{\sf J}^2 + \SemiNorm{\uht}{\gamma, \wht}^2,
\end{equation}
where the last three terms are defined in~\eqref{def:A-norm}, \eqref{def:jump-term}, and~\eqref{def:gamma-norm}, respectively.
%%%%%%%%%%%%%%%%%%%%%%%%%%%%%%%%%%%%%%%%%%%%
\begin{proposition}[Continuous dependence on the data]
\label{prop:continuous-dependence}
Any solution to the linearized space--time formulation~\eqref{eq:kernel-linearized-problem} satisfies the following bound:
\begin{equation}
\label{eq:continuous-dependence-final}
\begin{split}
 \Tnorm{\uht}{\wht}^2 \lesssim \Norm{\f}{L^1(0, T; L^2(\Omega)^d)}^2 + \Norm{\u_0}{L^2(\Omega)^d}^2,
\end{split}
\end{equation}
where the hidden constant depends only on~$\ell$ (thus being, in particular, independent of the final time~$T$ and the discrete function~$\wht$).
\end{proposition}
\begin{proof}
Let~$n \in \{1, \ldots, n\}$ and~$\uht$ be a solution to the linearized problem~\eqref{eq:kernel-linearized-problem}. We define the following test functions  (c.f. \eqref{lambdadef}):
\begin{equation*}
\zht^{\star, (n)} {}_{|_{Q_m}}:= \begin{cases}
\ItR(\varphi_n \uht) & \text{ if } m = n, \\
0 & \text{ otherwise}.
\end{cases}
\end{equation*}
Taking~$\zht^{\star, (n)}$ as the test function in~\eqref{eq:kernel-linearized-problem}, we have
\begin{equation}
\label{eq:aux-identity-continuous-dependence}
\begin{split}
\big(\dpt \uht, \ItR(\varphi_n \uht) \big)_{\Qn} & - \big(\jump{\uht}_{n-1}, \ItR(\varphi_n \uht)(\cdot, \tnmo^+) \big)_{\Omega} + \nu \aht\big(\uht, \zht^{\star, (n)} \big)\\
& + \cht\big(\wht; \uht, \zht^{\star, (n)}\big) = \big(\f, \ItR(\varphi_n \uht) \big)_{\Qn},
\end{split}
\end{equation}
where, for~$n = 1$, we have used the abuse of notation~$\uht(\cdot, 0^{-}) := \u_0$.

Using property~\eqref{eq:property-interpolant-time-2} of the interpolant~$\ItR$ and the identity~$\ItR w(\cdot, \tnmo^+) = w(\cdot, \tnmo^+)$, the fact that~$\varphi_n(\tnmo) = 1$, $\varphi_n(\tn) = 1/2$, and~$\varphi'(t) = -\lambda_n$, integration by parts in time, and the identity
\begin{equation*}
\frac12 w(\cdot, \tnmo^+)^2 + \jump{w}_{n - 1} w(\cdot, \tnmo^+) = - \frac12 \jump{w}_{n - 1}^2 + \frac12 w(\cdot, \tnmo^-)^2,
\end{equation*}
we obtain the following equation:
\begin{alignat}{3}
\nonumber
\big(\dpt \uht, & \, \ItR(\varphi_n \uht) \big)_{\Qn} - \big(\jump{\uht}_{n-1}, \ItR(\varphi_n \uht)(\cdot, \tnmo^+) \big)_{\Omega} \\
%%%%
\nonumber
& = \big(\dpt \uht, \varphi_n \uht \big)_{\Qn} - \big(\jump{\uht}_{n-1}, \uht(\cdot, \tnmo^+) \big)_{\Omega} \\
%%%
\nonumber
& = \frac12 \int_{\Qn} \dpt (\varphi_n \uht^2) \dV - \frac12 \int_{\Qn} \varphi' \uht^2 \dV - \big(\jump{\uht}_{n-1}, \uht(\cdot, \tnmo^+) \big)_{\Omega} \\
%%%
\nonumber
& = \frac14 \Norm{\uht(\cdot, \tn^-)}{L^2(\Omega)^d}^2 - \Big(\frac12 \uht(\cdot, \tnmo^+) + \jump{\uht}_{n - 1}, \uht(\cdot, \tnmo^+) \Big)_{\Omega} + \frac{\lambda_n}{2} \Norm{\uht}{L^2(\Qn)^d}^2 \\
%%%
\label{eq:mht-ItR}
& = \frac14 \Norm{\uht(\cdot, \tn^-)}{L^2(\Omega)^d}^2 + \frac12 \Norm{\jump{\uht}_{n - 1}}{L^2(\Omega)^d}^2 - \frac12 \Norm{\uht(\cdot, \tnmo^-)}{L^2(\Omega)^d}^2  + \frac{\lambda_n}{2} \Norm{\uht}{L^2(\Qn)^d}^2,
\end{alignat}
where the second and third terms become~$(1/4) \Norm{\uht}{L^2(\SO)^d}^2 - \Norm{u_0}{L^2(\Omega)^d}^2$ in the case~$n = 1$.

Moreover, as an immediate consequence of Lemma~\ref{lemma:bilinear-form-varphi}, we have
\begin{equation}
\label{eq:aht-ItR}
\nu \aht\big(\uht, \zht^{\star, (n)}\big) \geq \frac{\nu}{2} \cA \int_{\In} \Norm{\uht(\cdot, t)}{\calA,h}^2 \dt \geq 0.
\end{equation}

Then, using property~\eqref{eq:property-interpolant-time-3} and proceeding as for~\eqref{eq:identity-cht}, we easily get
\begin{alignat}{3}
\nonumber
\cht( \wht; \uht, \zht^{\star, (n)}) & = \frac12 \sum_{F \in \FhI} \QtRn \Big( \varphi_n \int_F \gamma_F(\wht) \, |\jump{\uht}|^2\dS \Big) \\
%%%
\label{eq:cht-ItR}
& \geq \frac14 \sum_{F \in \FhI} \QtRn \Big( \int_F \gamma_F(\wht) \, |\jump{\uht}|^2\dS \Big) \geq 0.
\end{alignat}

Combining~\eqref{eq:mht-ItR}, \eqref{eq:aht-ItR}, and~\eqref{eq:cht-ItR} with identity~\eqref{eq:aux-identity-continuous-dependence}, using the H\"older inequality, the stability of the interpolant~$\ItR$ in Lemma~\ref{lemma:stab-ItR}, and the weak partial bound in Lemma~\ref{lemma:weak-bound} for~$n - 1$, we get
\begin{alignat}{3}
\nonumber
\frac14 \Norm{\uht(\cdot, \tn^-)}{L^2(\Omega)^d}^2 & + \frac12 \Norm{\jump{\uht}_{n - 1}}{L^2(\Omega)^d}^2 + \frac{\lambda_n}{2} \Norm{\uht}{L^2(\Qn)^d}^2 \\
%%%
\nonumber
& \le \frac12 \Norm{\uht(\cdot, \tnmo^-)}{L^2(\Omega)^d}^2 + (\f, \ItR (\varphi_n \uht) )_{\Qn} \\
%%%
\nonumber
& \le \frac12 \Norm{\uht(\cdot, \tnmo^-)}{L^2(\Omega)^d}^2 + \Norm{\f}{L^1(\In; L^2(\Omega)^d)} \Norm{\ItR(\varphi_n \uht)}{L^{\infty}(\In; L^2(\Omega)^d)} \\
%%%
\nonumber
& \le \frac12 \Norm{\uht(\cdot, \tnmo^-)}{L^2(\Omega)^d}^2 + \CSI \Norm{\f}{L^1(\In; L^2(\Omega)^d)} \Norm{\uht}{L^{\infty}(\In; L^2(\Omega)^d)} \\
%%%
\nonumber
& \le \Norm{\u_0}{L^2(\Omega)^d}^2 + \Norm{\f}{L^1(0, \tnmo; L^2(\Omega)^d)} \Norm{\uht}{L^{\infty}(0, \tn; L^2(\Omega)^d)} \\
\nonumber
& \quad + \CSI \Norm{\f}{L^1(\In; L^2(\Omega)^d)} \Norm{\uht}{L^{\infty}(\In; L^2(\Omega)^d)} \\
%%%
\label{eq:continuous-dependence-bound-tn}
& \lesssim \Norm{\u_0}{L^2(\Omega)^d}^2 + \Norm{\f}{L^1(0, \tn; L^2(\Omega)^d)} \Norm{\uht}{L^{\infty}(0, \tn; L^2(\Omega)^d)},
\end{alignat}
where the hidden constant is independent of the time interval~$\In$ and depends only on the degree of approximation~$\ell$.

From the inverse estimate in Lemma~\ref{lemma:L2-Linfty}, we deduce that
\begin{equation}
\label{eq:L2-Linfty-uht}
\frac{1}{4 \Cinv^2} \Norm{\uht}{L^{\infty}(\In; L^2(\Omega)^d)}^2 \le \frac{\lambda_n}{2} \Norm{\uht}{L^2(\Qn)^d}^2.
\end{equation}

Summing~\eqref{eq:continuous-dependence-bound-tn} and the bound in Lemma~\ref{lemma:weak-bound} for~$n = N$, and using~\eqref{eq:L2-Linfty-uht}, we obtain the following bound valid for all~$n \in \{1,2,\ldots,N\}$:
\begin{equation}
\label{eq:sum-estimates-continuous-dependence}
\begin{split}
\Norm{\uht}{L^{\infty}(\In; L^2(\Omega)^d)}^2 & + \frac12 \Norm{\uht(\cdot, \tn^{-})}{L^2(\Omega)^d}^2 + \nu \Norm{\uht}{\calA}^2 + \SemiNorm{\uht}{\sf J}^2 + \SemiNorm{\uht}{\gamma, \wht}^2 \\
& \lesssim \Norm{\u_0}{L^2(\Omega)^d}^2 + \Norm{\f}{L^1(0, T; L^2(\Omega)^d)} \Norm{\uht}{L^{\infty}(0, T; L^2(\Omega)^d)},
\end{split}
\end{equation}
where the hidden constant is independent of the final time~$T$, the time interval~$\In$, and the discrete function~$\wht$.

Bound~\eqref{eq:continuous-dependence-final} then follows by taking the index~$n \in \{1, \ldots, N\}$ where the left-hand side of~\eqref{eq:sum-estimates-continuous-dependence} takes its maximum value, and using the Young inequality.
\end{proof}

\begin{remark}[Well-posedness of the linearized problem~\eqref{eq:kernel-linearized-problem}]
\label{rem:existence-linearized}
Uniqueness of the solution to the linearized space--time formulation~\eqref{eq:kernel-linearized-problem} is an immediate consequence of Proposition~\ref{prop:continuous-dependence}. Moreover, since~\eqref{eq:kernel-linearized-problem} is equivalent to a square linear system, existence follows from uniqueness.
\eremk
\end{remark}

\begin{remark}[Helmholtz-Hodge decomposition]
Since~$\f \in L^1(0, T; L^2(\Omega)^d)$ and $\Zht \subset \Z$, the linearized space--time formulation~\eqref{eq:kernel-linearized-problem} and the continuous dependence on the data result in Proposition~\ref{prop:continuous-dependence} can be written replacing~$\f$ by its Helmholtz-Hodge projection~$\calH(\f)$, thus reflecting the pressure-robustness of the scheme (see~\cite[Lemma 2.6]{John_etal:2017} and~\cite[Thm.~3.3]{Gauger_Linke_Schroeder:2019}).
\eremk
\end{remark}

% --------------------------------------------------
\subsection{Fixed-point argument\label{subsec:fixed-point}}
We denote by~$\Clin > 0$ the hidden constant in the statement of Proposition~\ref{prop:continuous-dependence}, which we recall is independent of the meshsize~$h$, the maximum time step~$\tau$, the viscosity~$\nu$, the final time~$T$, and the discrete function~$\wht$.

\begin{theorem}[Existence of discrete solutions]
\label{thm:existence-discrete-solutions}
Given~$\f \in L^1(0, T; L^2(\Omega)^d)$ and~$\u_0 \in \Z$, there exists at least a solution~$\uht \in \Zht$ to the space--time formulation~\eqref{eq:space-time-formulation}, which satisfies
\begin{equation}
\label{eq:continuous-dependence-nonlinear}
\Tnorm{\uht}{\uht}^2 \le \Clin \big(\Norm{\f}{L^1(0, T; L^2(\Omega)^d)}^2 + \Norm{\u_0}{L^2(\Omega)^d}^2 \big).
\end{equation}
\end{theorem}
\begin{proof}
We define the ball
\begin{equation*}
\calB := \left\{\wht \in \Zht \, : \, \Tnorm{\wht}{{\bf 0}}^2 \le \Clin \big(\Norm{\f}{L^1(0, T; L^2(\Omega)^d)}^2 + \Norm{\u_0}{L^2(\Omega)^d}^2 \big) \right\},
\end{equation*}
and the map~$\Phi : \Zht \to \Zht$, which assigns~$\wht \mapsto \uht$ as follows: given~$\wht \in \Zht$, $\Phi(\wht) := \uht \in \Zht$, where~$\uht$ is the solution to the linearized space--time formulation~\eqref{eq:kernel-linearized-problem} with discrete coefficient~$\wht$.

As discussed in Remark~\ref{rem:existence-linearized}, Proposition~\ref{prop:continuous-dependence} guarantees that~$\Phi$ is well defined. Moreover, since the constant~$\Clin$ does not depend on~$\wht$  and noting that $\Tnorm{\wht}{{\bf 0}} \le \Tnorm{\wht}{{\wht}}$ for any $\wht \in \Zht$, we also have~$\Phi(\Zht) \subset \calB$. Therefore, by the Schauder fixed-point theorem (see, e.g., \cite[Thm.~4.1.1 in Ch.~4]{Smart:1974}), the map~$\Phi$ has at least a fixed point~$\uht \in \calB$, which solves the space--time formulation~\eqref{eq:kernel-space-time-problem} and satisfies the continuous dependence on the data in~\eqref{eq:continuous-dependence-nonlinear}.
\end{proof}

% ----------------------------------------------------
%           CONVERGENCE ANALYSIS
% ----------------------------------------------------
\section{Convergence analysis}\label{sec:conv}
In this section, we derive \emph{a priori} error estimates for the space--time formulation~\eqref{eq:space-time-formulation}.

% ----------------------------------------------------
%           TOOLS CONVERGENCE
% ----------------------------------------------------
\subsection{Some tools for the convergence analysis}
We first recall the definition of the projection operator in~\cite[Eq.~(12.9) in Ch.~12]{Thome_book:1997} (see also~\cite[Def.~3.1]{Schotzau_Schwab:2000}), which is common in the analysis of DG-time discretizations of parabolic problems.

\begin{definition}[Projection~$\Pt$]
\label{def:Pt}
The projection operator~$\Pt: H^1(0, T) \rightarrow \Pp{\ell}{\Tt}$ is defined as follows: for any~$v \in H^1(0, T)$,  the projection~$\Pt v$ satisfies
\begin{subequations}
\begin{alignat}{3}
\label{eq:Pt-1}
\Pt v (\tn^-) - v(\tn) & = 0, \\
\label{eq:Pt-2}
\big( (\Id - \Pt) v, q_{\ell - 1} \big)_{\In} & = 0 & & \qquad \forall q_{\ell - 1} \in \Pp{\ell - 1}{\In},
\end{alignat}
\end{subequations}
for~$n = 1, \ldots, N$. In particular, if~$\ell = 0$, condition~\eqref{eq:Pt-2} is omitted. 

Given a Banach space~$(Z, \Norm{\cdot}{Z})$ with $Z \subseteq L^1(\Omega)$, the definition of~$\Pt$ can be extended to functions in~$H^1(0, T; Z) \hookrightarrow C^0([0, T]; Z)$ by requiring that~$\Pt : H^1(0, T; Z) \rightarrow P_{\ell}(\Tt) \otimes Z$ and the left-hand sides of~\eqref{eq:Pt-1} and~\eqref{eq:Pt-2} are equal to zero almost everywhere in~$\Omega$.
\end{definition}

%%%
\begin{lemma}[Stability of~$\Pt$]
\label{lemma:stab-Pt}
Let~$(Z, \Norm{\cdot}{Z})$ with~$Z \subseteq L^1(\Omega)$ be a given Banach space. For all~$v \in H^1(0, T; Z)$ and~$n \in \{1, \ldots, N\}$, it holds
\begin{equation*}
\Norm{\Pt v}{L^{\infty}(\In; Z)} \lesssim \Norm{v}{L^{\infty}(\In; Z)}.
\end{equation*}
\end{lemma}
\begin{proof}
Given~$r \in \IN$, we denote by~$\Pi_r^t : L^1(0, T; Z) \to P_r(\Tt) \otimes Z$ the~$L^2(0, T)$-orthogonal projection in the space~$P_r(\Tt)$, defined by imposing that the orthogonality condition holds true almost everywhere in~$\Omega$.
Let~$v \in H^1(0, T; Z) \hookrightarrow L^{\infty}(0, T; Z)$. Due to~\eqref{eq:Pt-2}, it can be easily seen that~$\Pi_{\ell - 1}^t \Pt v = \Pi_{\ell - 1}^t v$. 
Therefore, we can write
\begin{equation*}
\begin{split}
\Pt v(\cdot, t) & = \Pi_{\ell - 1}^t v (\cdot, t) + \alpha(\cdot) L_{\ell}(t) = \sum_{i = 0}^{\ell - 1} \frac{L_i(t)}{\Norm{L_i}{L^2(\In)}^2} \int_{\In} v(\cdot, s) L_i(s) \ds  + \alpha(\cdot) L_{\ell}(t) \qquad \forall t \in \In,
\end{split}
\end{equation*}
where~$\alpha(\cdot) = (v - \Pi_{\ell - 1} v)(\cdot, \tn^-)$ due to~\eqref{eq:Pt-1}, and~$L_{\ell}$ is the~$\ell$th Legendre polynomial in the interval~$\In$. 
Since~$L_i(\tn^-) = 1$ for all~$i \in \IN$, $\alpha(\cdot)$ can be written as
\begin{equation*}
\alpha(\cdot) = v(\cdot, \tn^-) - \sum_{i = 0}^{\ell - 1} \frac{1}{\Norm{L_i}{L^2(\In)}^2} \int_{\In} v(\cdot, s) L_i(s) \ds .
\end{equation*}

Therefore, using the convexity of~$\Norm{\cdot}{Z}$, the Jensen and the triangle inequalities, and the uniform bound~$\Norm{L_i}{L^{\infty}(\In)} = 1$, we get
\begin{alignat*}{3}
\Norm{\Pt v(\cdot, t)}{Z} \le \sum_{i = 0}^{\ell - 1} \frac{2}{\Norm{L_i}{L^2(\In)}^2} \int_{\In} \Norm{v(\cdot, s)}{Z} |L_i(s)| \ds  + \Norm{v(\cdot, \tn^-)}{Z} \qquad \forall t \in \In,
\end{alignat*}
which, combined with the H\"older inequality and the fact that~$\Norm{L_i}{L^1(\In)}/\Norm{L_i}{L^2(\In)}^2 \le \sqrt{(2i + 1)/2}$, leads to
\begin{alignat}{3}
\nonumber
\Norm{\Pt v(\cdot, t)}{Z} & \le \Norm{v}{L^{\infty}(\In; Z)} \Big( 2\sum_{i = 0}^{\ell - 1} \frac{\Norm{L_i}{L^1(\In)}}{\Norm{L_i}{L^2(\In)}^2} + 1\Big) 
\le \Norm{v}{L^{\infty}(\In; Z)} \Big( \sqrt{2}\sum_{i = 0}^{\ell - 1} \sqrt{2i + 1} + 1\Big) \\
\label{eq:bound-Ptv-Z}
& \le (\sqrt{2} \ell^{3/2} + 1) \Norm{v}{L^{\infty}(\In; Z)}.
\end{alignat}
The result then follows by taking~$t \in [\tnmo, \tn]$ where the left-hand side of~\eqref{eq:bound-Ptv-Z} achieves its maximum value.

%Using the stability of~$\Pi_{\ell - 1}^t$ in~$L^{\infty}(\In)$ (see, e.g, \cite[Thm.~18.16]{Ern_Guermond-I:2021}), the triangle inequality, and the fact that~$\Norm{L_{\ell}}{L^{\infty}(\In)} = 1$, we get
%\begin{equation*}
%\Norm{\Pt v}{L^{\infty}(\In)} \le \Norm{\Pi_{\ell - 1}^t v}{L^{\infty}(\In)} + |\alpha| \le \Norm{\Pi_{\ell - 1}^t v}{L^{\infty}(\In)} + \Norm{(\Id - \Pi_{\ell - 1}^t) v}{L^{\infty}(\In)} \lesssim \Norm{v}{L^{\infty}(\In)}.
%\end{equation*}
\end{proof}

We now prove some approximation properties that are valid for any projection~$\mathbb{P}_{\tau}$ in the space~$P_{\ell}(\Tt) \otimes Z$ that is stable in~$L^{\infty}(0, T; Z)$. In particular, due to Lemmas \ref{lemma:stab-ItR} and~\ref{lemma:stab-Pt}, it holds true for~$\ItR$ and~$\Pt$.
%%%%
\begin{comment}
\begin{lemma}[Estimates for~$\Pt$]
Let~$(Z, \Norm{\cdot}{Z})$ with~$Z \subseteq L^1(\Omega)$ be a given Banach space. Then, 
\label{lemma:estimates-Pt}
\begin{equation}
\Norm{(\Id - \Pt) v}{L^{\infty}(\In; Z)} \lesssim \tau_n^{m + 1} \Norm{\dpt^{(m + 1)} v}{L^{\infty}(\In; Z)} \quad \forall v \in W^{m + 1}_{\infty}(\In; Z), \ 0 \le m \le \ell.
\end{equation}
\end{lemma}
\end{comment}

\begin{lemma}[Estimates for a generic projector~$\mathbb{P}_{\tau}$] 
%{see~\cite[Thm.~11.13 and~\S11.5.1]{Ern_Guermond-I:2021}}]
\label{lemma:estimates-Ptau}
Given~$p \in [1, \infty]$, a Banach space~$(Z, \Norm{\cdot}{Z})$ with~$Z \subseteq L^1(\Omega)$, and a projection operator~$\mathbb{P}_{\tau}$ in the space~$P_{\ell}(\Tt) \otimes Z$ that is stable in the~$L^{\infty}(0, T; Z)$ norm, %, and~$m \in \{0, \ldots, r\}$, 
the following estimate holds for all~$v \in W^{\ell + 1}_p(0, T; Z)$ and~$n \in \{1, \ldots, N\}$:
\begin{equation}
\label{eq:estimate-ItR}
\Norm{(\Id - \mathbb{P}_{\tau}) v}{L^p(\In; Z)} \lesssim \tau_n^{\ell + 1} \Norm{\dpt^{(\ell + 1)} v}{L^p(\In; Z)}.
\end{equation}
\end{lemma}
\begin{proof}
Let~$p \in [1, \infty]$, and~$v \in W_p^{\ell + 1}(\In; Z)$. Using the triangle and the H\"older inequalities, and the stability %in Lemma~\ref{lemma:stab-ItR} 
of~$\mathbb{P}_{\tau}$, for any~$v_{\tau} \in P_{\ell}(\In) \otimes Z$, we get
\begin{alignat}{3}
\nonumber
\Norm{(\Id - \mathbb{P}_{\tau}) v}{L^p(\In; Z)} 
& \le \Norm{v - v_{\tau}}{L^p(\In; Z)} + \Norm{\mathbb{P}_{\tau} (v - v_{\tau})}{L^p(\In; Z)} \\
%%%
\nonumber
& \lesssim \Norm{v - v_{\tau}}{L^p(\In; Z)} + \tau_n^{1/p} \Norm{v - v_{\tau}}{L^{\infty}(\In; Z)},
\end{alignat}
which, together with the approximation properties of~$P_{\ell}(\In) \otimes Z$ from~\cite[Lemma~A.1]{Diening_Storn_Tscherpel:2023} and the H\"older inequality, yields
\begin{alignat*}{3}
\Norm{(\Id - \mathbb{P}_{\tau}) v}{L^p(\In; Z)} 
& \lesssim \tau_n^{\ell + 1} \Norm{\dpt^{(\ell + 1)} v}{L^p(\In; Z)} + \tau_n^{1/p + \ell} \Norm{\dpt^{(\ell + 1)} v}{L^1(\In; Z)} \\
& \lesssim \tau_n^{\ell + 1} \Norm{\dpt^{(\ell + 1)} v}{L^p(\In; Z)}.
\end{alignat*}
This completes the proof of~\eqref{eq:estimate-ItR}.
\end{proof}

We shall denote by~$\IRT$ the standard RT interpolant operator; see, e.g., \cite[\S6.1.1]{Wang_Ye:2007} and~\cite[\S2.5.1 in Ch.~2]{Boffi_Brezzi_Fortin:2013}. 

\begin{definition}[Interpolant~$\IRT$]
The interpolant~$\IRT : H^1(\Omega)^d \to \RTk(\Th)$ is defined on each element~$K \in \Th$ as follows:
\begin{subequations}
\begin{alignat}{3}
((\Id - \IRT) \v \cdot \bn_{K}, p_k)_{\partial K} & = 0  & &\qquad \forall p_k \in \mathcal{R}_k(\partial K), \\
\label{def:IRT-2}
((\Id - \IRT) \v, \p_{k-1} )_{K} & = 0 & & \qquad \forall \p_{k-1} \in \Pp{k}{K}^d,
\end{alignat}
\end{subequations}
where~$\mathcal{R}_k(\partial K) := \{\phi \in L^2(K) \, :\, \phi_{|_f} \in \Pp{k}{f}$ for all facets~$f$ of~$K$\}.
\end{definition}

Next lemma concerns the standard local approximation properties of~$\IRT$.
%%%
\begin{lemma}[Estimates for~$\IRT$, {see~\cite[Prop.~2.5.1]{Boffi_Brezzi_Fortin:2013}}]
\label{lemma:estimates-IRT}
For all~$\v \in H^{s+1}(K)$ with~$0 \le s \le k$, it holds
\begin{alignat*}{3}
\Norm{(\Id - \IRT) \v}{L^2(K)} & \lesssim h_K^{s + 1} \SemiNorm{\v}{H^{s + 1}(K)^d}, \\
\Norm{\nabla (\Id - \IRT) \v}{L^2(K)^d} & \lesssim h_K^s \SemiNorm{\v}{H^{s + 1}(K)^d}.
\end{alignat*}
\end{lemma}
In what follows, whenever needed in the context, the interpolant~$\IRT$ is to be understood as applied pointwise in time almost everywhere.
\begin{remark}
For any ${\bf u} \in W_p^1(0,T;\Z)$, $p \in [1,\infty]$, it holds \cite[Cor.~2.3.1]{Boffi_Brezzi_Fortin:2013}
$$
\IRT {\bf u}(\cdot,t) \in \BDMk(\Th) \subseteq \RTk(\Th) \qquad \forall t \in [0,T] \, ,
$$
so that $\IRT {\bf u} \in \Zht$ for both choices in \eqref{space-choice}.
\eremk
\end{remark}

% ----------------------------------------------------
%           PARTIAL BOUND DISCRETE ERROR
% ----------------------------------------------------
\subsection{A priori error estimates for the discrete error}

In the present section we derive Reynolds semi-robust and pressure-robust error estimates for the proposed scheme.
As above, the symbol $\u$ will denote the solution to the continuous weak formulation~\eqref{eq:weak-formulation} while~$\uht \in \Zht$ will represent the solution to the space--time formulation~\eqref{eq:kernel-space-time-problem}. Finally, we will make use of the composed projection~$\Piht \u \in \Zht$ given by~$\Pt \IRT \u$, which satisfies the following approximation result.

\begin{lemma}[Estimates for~$\Piht$]
\label{lemma:estimate-Piht}
For any~$n \in \{1,2,\ldots,N\}$, $K \in \Th$, and~$s \in\{0,1\}$, it holds
\begin{equation}\label{approx-estimate-st}
\Norm{(\Id - \Piht) \v}{L^\infty(\In; H^s(K)^d)} 
\lesssim \hK^{k+1-s} \Norm{\v}{L^\infty(\In; H^{k+1}(K)^d)} 
+ \tau_n^{\ell + 1} \Norm{\v}{W^{\ell+1}_{\infty}(\In; H^s(K)^d)}
\end{equation}
for all~$\v \in L^{\infty}(\In; H^{k + 1}(K)^d) \cap W^{\ell + 1}_{\infty}(\In; H^s(K)^d)$.
\end{lemma}
\begin{proof}
By a triangle inequality and the stability in Lemma~\ref{lemma:stab-Pt} of~$\Pt$%the stability in~$L^\infty(I_n)$ of the operator~$\Pt$
$$
\begin{aligned}
\| \u - \Piht \u \|_{L^\infty(I_n;H^s(K)^d)} 
& \lesssim
\| \u - \Pt \u \|_{L^\infty(I_n;H^s(K)^d)}
+ \| \Pt(\u - \IRT \u) \|_{L^\infty(I_n;H^s(K)^d)} \\
& \lesssim \| \u - \Pt \u \|_{L^\infty(I_n;H^s(K)^d)}
+ \| \u - \IRT \u \|_{L^\infty(I_n;H^s(K)^d)} \, .
\end{aligned}
$$
The result then follows by the approximation properties in Lemmas~\ref{lemma:estimates-Ptau} and~\ref{lemma:estimates-IRT} of~$\Pt$ and~$\IRT$, respectively.
%standard approximation properties of polynomials in one dimension and of the RT/BDM interpolation operator 
%{\color{magenta} (add some reference/specification)}. 
\end{proof}

% --------------------------------------------------------
\subsubsection{Discrete error equations}
We introduce the following error functions:
\begin{equation*}
%\label{eq:error-functions}
\eu := \u - \uht = \epi + \Piht \eu, \ \text{ with } \epi := \u - \Piht \u \ \text{ and } \ \Piht \eu := \Piht \u - \uht.
\end{equation*}

Using the definition of the form~$\Bht$ in~\eqref{eq:kernel-linearized-problem}, and the definition of the space--time formulation~\eqref{eq:kernel-space-time-problem}, we deduce that the following error equation is satisfied for all~$\zht \in \Zht$:
\begin{alignat}{3}
\nonumber
\Bht(\uht; \Piht \eu, \zht) & = \Bht(\uht; \Piht \u, \zht) - (\f, \zht)_{\QT} - (\u_0, \zht(\cdot, 0))_{\Omega} \\
%%%
\nonumber
& = -\mht(\epi, \zht) - \nu \aht(\epi, \zht) \\
\label{eq:discrete-error-equation}
&\quad + \cht(\uht; \Piht \u, \zht) - c^T(\u; \u, \zht),
\end{alignat}
where we have conveniently used the notation
$$c^T(\u; \u, \zht) := \int_0^T c(\u; \u, \zht) \dt.$$

The first term on the right-hand side of~\eqref{eq:discrete-error-equation} can be simplified using the properties of the projection operator~$\Pt$. Next lemma is a key ingredient to avoid any restriction of the form~$\tau \lesssim h$ in the error analysis.
\begin{lemma}
\label{lemma:identity-mht-error}
The following identity holds for all~$\zht \in \Zht$:
\begin{equation*}
% \begin{aligned}
\mht(\epi, \zht) 
%%%
% & =
% - \sum_{n = 1}^N \big((\Id - \IRT) \u, \dpt \zht \big)_{\Qn} + \big((\Id - \IRT) \u, \zht\big)_{\ST} \\
% & \quad + \sum_{n = 1}^{N - 1} \big((\Id - \IRT) \u(\cdot, \tn^-), \jump{\zht}_n \big)_{\Omega} \\
%%%
= \mht( (\Id - \IRT) \u, \zht) = \big( (\Id - \IRT) \dpt \u, \zht\big)_{\QT} + ((\Id - \IRT) \eu, \zht)_{\SO}.
%.
% \end{aligned}
\end{equation*}
\end{lemma}
\begin{proof}
We use the definition in~\eqref{def:discrete-space-time-forms} of~$\mht(\cdot, \cdot)$, integration by parts in time, and the flux-jump identity
\begin{equation*}
%\label{eq:flux-jump-identity}
\jump{v}_n w(\cdot, \tn^+) + \jump{w}_n v(\cdot, \tn^-) = \jump{vw}_n \quad \text{ for }n = 1, \ldots, N,
\end{equation*}
to obtain
\begin{alignat*}{3}
\mht(\epi, \zht) & = \sum_{n = 1}^N (\dpt \epi, \zht)_{\Qn} - \sum_{n = 1}^{N - 1} \big(\jump{\epi}_n, \zht(\cdot, \tn^+) \big)_{\Omega} + (\epi, \zht)_{\SO} \\
%%%
& = - \sum_{n = 1}^N \big((\Id - \Pt \IRT) \u, \dpt \zht \big)_{\Qn} + \big((\Id - \Pt \IRT) \u, \zht \big)_{\ST} \\
& \quad + \sum_{n = 1}^{N - 1} \int_{\Omega} \Big(\jump{\zht \cdot (\Id - \Pt \IRT) \u}_n - \jump{(\Id - \Pt \IRT) \u}_n \cdot \zht(\cdot, \tn^+) \Big) \dx \\
%%%
& = - \sum_{n = 1}^N \big((\Id - \Pt \IRT) \u, \dpt \zht \big)_{\Qn} + \big((\Id - \Pt \IRT) \u, \zht\big)_{\ST} \\
& \quad + \sum_{n = 1}^{N - 1} \big((\Id - \Pt \IRT) \u(\cdot, \tn^-), \jump{\zht}_n \big)_{\Omega}.
\end{alignat*}
Consequently, due to Definition~\ref{def:Pt}, we can remove the projection~$\Pt$ in the last expression. Integrating by parts back in time and noting that all time jumps of $\IRT\u$ vanish, we obtain the desired identity.
\end{proof}

Therefore, using the identity in Lemma~\ref{lemma:identity-mht-error}, the error equation~\eqref{eq:discrete-error-equation} reduces to
\begin{alignat}{3}
\nonumber
\Bht(\uht; \Piht \eu, \zht) & =
- \big( (\Id - \IRT) \dpt \u, \zht\big)_{\QT} - ((\Id - \IRT) \eu, \zht)_{\SO} - \nu \aht(\epi, \zht) \\
\label{eq:discrete-error-equation-simplified}
&\quad + \cht(\uht; \Piht \u, \zht) - c^T(\u; \u, \zht).
\end{alignat}

% --------------------------------------------------
\subsubsection{Bounds for the consistency terms}
% --------------------------------------------------
We here develop bounds for the consistency terms in the error equation~\eqref{eq:discrete-error-equation-simplified}.
In the following, in order to lighten the notation in the proofs, and unless otherwise needed, we will collect into generic constants~${\cal R}_{\u}$ certain regularity terms of the solution~$\u$, all with summability~$p=+\infty$ in time. Such a summability order in time could be relaxed in many instances below without any loss in the final convergence order but at the price of a more cumbersome presentation.

\begin{lemma}[Bound for the time derivative term]\label{lemma:cons-time}
Let~$\u$ denote the velocity solution to the continuous weak formulation~\eqref{eq:weak-formulation} and~$\uht \in \Zht$ be the solution to the space--time formulation~\eqref{eq:kernel-space-time-problem}.
Let also~$\u \in W^1_\infty(0,T; H^{k+1}(\Omega)^d)$. Then, for any function~$\vht \in \Zht$ satisfying
\begin{equation}
\label{eq:scaling-test}
\Norm{\vht}{L^{\infty}(\In; L^2(\Omega)^d)} \lesssim \Norm{\Piht \eu}{L^{\infty}(\In; L^2(\Omega)^d)},
\end{equation}
for~$n = 1, \ldots, N$, the following estimate holds for any positive real $\varepsilon$:
\begin{alignat*}{3}
\big( (\Id - \IRT) \dpt \u, \vht\big)_{\Qn} 
%%%
& \lesssim \varepsilon^{-1} {\cal R}_{\u} \tau_n h^{2k + 2} +  
\varepsilon \tau_n \Norm{\Piht \eu}{L^{\infty}(\In; L^2(\Omega)^d)}^2.
\end{alignat*}
where~${\cal R}_{\u} = \Norm{\dpt \u}{L^\infty(\In; H^{k + 1}(\Omega)^d)}^{2}$, and the hidden constant only depends on the degrees of approximation $k$ and~$\ell$, and the shape-regularity parameter~$\rho$.
\end{lemma}
\begin{proof}
The result is a simple consequence of the approximation properties in Lemma~\ref{lemma:estimates-IRT} of~$\IRT$ and~\eqref{eq:scaling-test},  combined with the H\"older and the Young inequalities.
\end{proof}

\begin{lemma}[Bound for the diffusion term]\label{lemma:cons-diff}
Let~$\u$ denote the velocity solution to the continuous weak formulation~\eqref{eq:weak-formulation} and~$\uht\in \Zht$ be the solution to space--time formulation~\eqref{eq:kernel-space-time-problem}. Let also~$\u \in W^{\ell + 1}_\infty(0, T; H^2(\Omega)^d) \cap L^{\infty}(0, T; H^{k + 1}(\Omega)^d)$. Then, for any function~$\vht \in \Zht$ satisfying
\begin{equation}
\label{eq:scaling-vht-aht}
\int_{\In} \Norm{\vht(\cdot, t)}{\calA, h}^2 \dt \lesssim \int_{\In} \Norm{\Piht \eu (\cdot, t)}{\calA, h}^2 \dt,
\end{equation}
for~$n = 1, \ldots, N$, the following estimate holds for any positive real $\varepsilon$:
\begin{equation*}
\aht^{(n)}((\Id - \Piht) \u, \vht) \lesssim  \varepsilon^{-1} {\cal R}_{\u} \Big( {\tau_n^{2\ell + 3}} + \tau_n h^{2k} \Big)
+ \varepsilon \int_{\In} \Norm{\Piht \eu}{\calA, h}^2 \dt.
\end{equation*}
where the hidden constant only depends on $k,\ell$, and~$\rho$. The positive real~${\cal R}_{\u}$ depends on~$\u$ evaluated in the norms of the regularity assumption in the statement of this lemma.
\end{lemma}
%%%
\begin{proof}
Using the continuity in~\eqref{eq:continuity-aht} of the bilinear form~$\aht^{(n)}(\cdot, \cdot)$ and~\eqref{eq:scaling-vht-aht}, we get
\begin{alignat}{3}
\nonumber
\aht^{(n)}((\Id - \Piht) \u, \vht) 
& \le C_{\calA_*} \Big(\int_{\In} \Norm{(\Id - \Piht) \u(\cdot, t)}{\calA_*, h}^2 \dt \Big)^{\frac12} \Big(\int_{\In} \Norm{\vht}{\calA, h}^2 \dt \Big)^{\frac12} \\
%%% 
\label{eq:aux-estimate-aht}
& \lesssim C_{\calA_*} \Big(\int_{\In} \Norm{(\Id - \Piht) \u(\cdot, t)}{\calA_*, h}^2 \dt \Big)^{\frac12} \Big(\int_{\In} \Norm{\Piht \eu}{\calA, h}^2 \dt \Big)^{\frac12}.
\end{alignat}
We now estimate the term involving~$(\Id - \Piht) \u$.
Due to the definition in~\eqref{def:A*-norm} of the norm~$\Norm{\cdot}{\calA_*, h}^2$, we have
\begin{alignat}{3}
\nonumber
\int_{\In} \Norm{(\Id - \Piht) \u(\cdot, t)}{\calA_*, h}^2 \dt & = \Norm{\Nablah (\Id - \Piht) \u}{L^2(\In; L^2(\Omega)^{d \times d})}^2 \\
\nonumber
& \quad + \sum_{F \in \Fh} \int_F \sigma h_F^{-1} \Norm{\jump{(\Id - \Piht) \u}}{L^2(\In; L^2(F)^d)}^2 \\
\nonumber
& \quad + \sum_{K \in \Th} \hK \Norm{\Nabla (\Id - \Piht) \u {}_{|_K} \bn_K}{L^2(\In; L^2(\partial K)^d)}^2 \\
%%% 
\nonumber
& =: \Theta_1 + \Theta_2 + \Theta_3.
\end{alignat}

The triangle inequality, the stability property in Lemma~\ref{lemma:stab-Pt} of~$\Pt$, the approximation properties in Lemmas~\ref{lemma:estimates-IRT} and~\ref{lemma:estimate-Piht} of~$\IRT$ and~$\Pt$, respectively, {and the H\"older inequality} yield
\begin{alignat}{3}
\nonumber
\Theta_1 & = \Norm{\Nablah (\Id - \Piht) \u}{L^2(\In; L^2(\Omega)^{d\times d})}^2 \\
%%%
\nonumber
& \lesssim  \Norm{(\Id - \Pt) \Nabla \u}{L^2(\In; L^2(\Omega)^{d\times d})}^2 + \tau_n \Norm{\Pt \Nablah (\Id - \IRT) \u}{L^{\infty}(\In; L^2(\Omega)^{d \times d})}^2 \\
%%%
\nonumber
& \lesssim \tau_n^{2\ell + 2} \Norm{\Nabla \dpt^{(\ell + 1)} \u}{L^2(\In; L^2(\Omega)^{d \times d})}^2 + \tau_n h^{2k} \Norm{\u}{L^{\infty}(\In; H^{k + 1}(\Omega)^d)}^2\\
%%%
\nonumber
& {\lesssim \tau_n^{2\ell + 3} \Norm{\Nabla \dpt^{(\ell + 1)} \u}{L^{\infty}(\In; L^2(\Omega)^{d\times d})}^2 + \tau_n h^{2k} \Norm{\u}{L^{\infty}(\In; H^{k + 1}(\Omega)^d)}^2.} 
\end{alignat}

As for the term~$\Theta_2$, we use the continuity in space of~$\u$ and~$\Pt \u$, the stability property in Lemma~\ref{lemma:stab-Pt} of~$\Pt$, standard trace inequalities in space, and the approximation properties in Lemma~\ref{lemma:estimates-IRT} of~$\IRT$ to obtain
\begin{alignat*}{3}
\nonumber
\Theta_2 & = \sum_{F \in \Fh} \int_F \sigma h_F^{-1} \Norm{\jump{(\Id - \Piht) \u}}{L^2(\In; L^2(F)^d)}^2  \\
%%%
\nonumber
& \lesssim  \sum_{F \in \Fh} \int_F \sigma h_F^{-1} \tau_n \Norm{\Pt \jump{(\Id - \IRT) \u}}{L^{\infty}(\In; L^2(F)^d)}^2  \\
%%%
\nonumber
& \lesssim \tau_n \sum_{K \in \Th} \sigma h_F^{-1} \Big(\hK^{-1} \Norm{(\Id - \IRT) \u}{L^{\infty}(\In; L^2(K)^d)}^2 + \hK \Norm{\Nabla (\Id - \IRT) \u}{L^{\infty}(\In; L^2(K)^{d \times d})}^2 \Big)  \\
%%%
& \lesssim \tau_n h^{2k} \Norm{\u}{L^{\infty}(\In; H^{k + 1}(\Omega)^d)}^2.
\end{alignat*}

Similar steps can be used to estimate the term~$\Theta_3$ as follows:
\begin{alignat}{3}
\nonumber
\Theta_3 & = \sum_{K \in \Th} \hK \Norm{\Nabla (\Id - \Piht) \u_{|_{K}} \bn_K}{L^2(\In; L^2(\partial K)^d)}^2 \\
%%%
\nonumber
& = \sum_{K \in \Th} \Big(\Norm{(\Id - \Pt) \Nabla \u}{L^2(\In; L^2(K)^{d \times d})}^2 + \tau_n \Norm{\Pt \Nabla (\Id - \IRT) \u}{L^{\infty}(\In; L^2(K)^{d \times d})}^2 \\
%%%
\nonumber
& \quad \quad + \hK^2 \Norm{(\Id - \Pt) \mathbf{H} \u}{L^2(\In; L^2(K)^{d \times d \times d})}^2 + \hK^2 \tau_n \Norm{\Pt \mathbf{H} (\Id - \IRT) \u}{L^{\infty}(\In; L^2(K)^{d \times d \times d})}^2 \Big) \\
%%%
\nonumber
& \lesssim \tau_n^{2 \ell + 2} \Norm{\Nabla \dpt^{(\ell + 1)} \u}{L^2(\In; L^2(\Omega)^{d \times d})}^2 + \tau_n h^{2k} \Norm{\u}{L^{\infty}(\In; H^{k + 1}(\Omega)^{d})}^2 \\
%%%
\nonumber
& \quad + h_K^2 \tau_n^{2 \ell + 2} \Norm{\dpt^{(\ell + 1)} \u}{L^2(\In; H^2(\Omega)^{d})}^2 \\
%%%
\nonumber
& {\lesssim \tau_n^{2 \ell + 3} \Norm{\Nabla \dpt^{(\ell + 1)} \u}{L^{\infty}(\In; L^2(\Omega)^{d \times d})}^2 + \tau_n h^{2k} \Norm{\u}{L^{\infty}(\In; H^{k + 1}(\Omega)^{d})}^2 }\\
%%%
\nonumber
& {\quad + h_K^2 \tau_n^{2 \ell + 3} \Norm{\dpt^{(\ell + 1)} \u}{L^{\infty}(\In; H^2(\Omega)^{d})}^2,
}
\end{alignat}
where we have denoted by~$\mathbf{H}$ the Hessian operator. The desired estimate then follows by combining the above estimates for~$\{\Theta_i\}_{i = 1}^3$ with~\eqref{eq:aux-estimate-aht}, noting that~$h_K \lesssim 1$ and applying a Young inequality.
\end{proof}

%%%
\begin{lemma}[Bound for the convection term]\label{lemma:cons-conv}
Let $\u$ denote the velocity solution to the continuous weak formulation~\eqref{eq:weak-formulation} and let~$\uht \in \Zht$ be the discrete solution to the 
space--time formulation~\eqref{eq:kernel-space-time-problem}. Let the following regularity assumptions hold: $\u \in L^\infty(0,T;W^{k+1}_4(\Omega)^d)$, 
$\u \in W^{\ell+1}_{\infty}(0,T;L^2(\Omega)^d)$,  $(\Nabla \u) \u \in W^{\ell+1}_{\infty}(0,T;L^2(\Omega)^d)$. Let furthermore $\vht$ be any function in $\Zht$ satisfying (for all $n \in \{1,2,\ldots,N\}$)
\begin{equation}\label{eq:vh:equiv}
\esssup_{t \in I_n} | \vht(\cdot,t) |_S \lesssim  \esssup_{t \in I_n} | \Piht \eu (\cdot,t) |_S \, ,
\end{equation}
where $| \cdot |_S$ denotes any seminorm on $\Vh$ and the hidden constant depends only on $\ell$. Then, for any positive real $\varepsilon$ and $n \in \{1,2,\ldots,N\}$, it holds
$$
\begin{aligned}
& {\big|} \cht^{(n)}(\uht; \Piht \u, \vht) - \int_{I_n} \! c(\u; \u, \vht)\dt {\big|}
\le C \, {\cal R}_{\u} \tau_n \varepsilon^{-1}  \Big( h^{2k+1} {+ \tau_n^{2\ell+2}}  \Big) \\
& \ + \varepsilon C \Big( \tau_n \,  \| \Piht \eu \|_{L^\infty(I_n;L^2(\Omega)^d)}^2 +  
\sum_{F \in \FhI} \QtRn\Big( \gamma_F(\uht) \, \|\jump{\vht}\|_{L^2(F)^d}^2 \Big)
+ \tau_n C^\star_{\u} \| \Piht \eu \|_{L^\infty(I_n;L^2(\Omega)^d)}^2  \, .
\end{aligned}
$$
where $C^\star_{\u} = \widetilde{C} \| \u \|_{L^\infty(0, T;W^{1}_{\infty}(\Omega)^d)}$ and both constants $C,\widetilde{C}$ only depend on $k,\ell$, and~$\rho$. The term ${\cal R}_{\u}$ depends on~$\u$ evaluated in the norms of the regularity assumption in the statement of this lemma. 
\end{lemma}
\begin{proof}
%%
% We define the following norm on the~$n$th time interval:
% \begin{equation}\label{koala}
% \Tnormn{\uht}^2 := 
% \nu \int_{I_n} \Norm{\uht(\cdot, t) }{\calA, h}^2 \dt
% \, + \sum_{F \in \FhI} \QtRn\Big( \gamma_F(\wht) \, \|\jump{\vht}\|_{L^2(F)^d}^2 \Big)
% \end{equation}
We preliminarily note that, from the definition of $\gamma$ in~\eqref{def:gamma}, for any $\wht \in \Vh$, it holds
\begin{equation}\label{eq:equiv:jmp}
\QtRn\Big( \| \jump{\vht} \|_{L^2(\FhI)^d}^2 \Big)
\lesssim \sum_{F \in \FhI} \QtRn\Big( \gamma_F(\wht) \, \|\jump{\vht}\|_{L^2(F)^d}^2 \Big).
\end{equation}

We now split the (spatial) discrete convective form into three parts. For all $\wht, \uht, \vht$ in $\Vh$
\begin{equation}
\label{eq:def-convec-forms}
\begin{split}
& \chtb(\wht; \uht, \vht) :=  \big((\Nabla_h \uht) \wht, \vht \big)_{\Omega} \, , \\
& \chth(\wht; \uht, \vht) := \big((\wht \cdot \bnF) \jump{\uht}, \mvl{\vht} \big)_{\FhI} \, , \\
& \chtt(\wht; \uht, \vht) :=
\frac12 \big(\gamma(\wht) \jump{\uht}, \jump{\vht} \big)_{\FhI} \, ,
\end{split}
\end{equation}
which we can be trivially extended to all sufficiently regular functions. We will denote by 
$\chtbh(\cdot;\cdot,\cdot)$ the sum of the first two forms above; note that $\chtbh(\cdot;\cdot,\cdot)$ is antisymmetric with respect to the last two entries, whenever the first one is solenoidal.

Let %now 
$\vht \in \Zht$. First recalling Remark \ref{rem:integr}, then by some simple algebra and noting that the spatial jumps of~$\u$ vanish, one can get
\begin{equation}\label{eq:conv:start}
\begin{aligned}
& \big| \cht^{(n)}(\uht; \Piht \u, \vht) - \int_{I_n} \! c(\u; \u, \vht) \dt \big|
\le \Big| \int_{I_n} c(\u; \u, \vht) \dt - \QtRn \big( c(\u; \u, \vht) \big) \Big| \\
& \quad +  \big| \QtRn \big( \chtbh (\u; (\Id - \Piht) \u , \vht) \big) \big|
+ \big| \QtRn \big( \chtbh (\Piht \eu; \Piht \u, \vht) \big) \big| \\
& \quad + \big| \QtRn \big( \chtbh ((\Id - \Piht) \u; \Piht \u, \vht) \big) \big|
+ \big| \QtRn \big( \chtt(\uht; \Piht \u, \vht) \big) \big| \\
& \quad =: \sum_{i=1}^{5} J^{(n)}_i \, .
\end{aligned}
\end{equation}
We now bound all the terms separately for any fixed~$n \in \{1,2,\ldots,N\}$. In the following derivations, the positive real~$\varepsilon$ will denote the parameter associated with the Young inequality. 

For the first term, using \eqref{eq:property-interpolant-time-1}, the H\"older inequality, and the approximation properties in Lemma~\ref{lemma:estimates-Ptau} of~$\ItR$ yield
%%
\begin{comment}
{\color{blue}
\begin{equation}
\begin{aligned}
J_1^{(n)} & = \Big| %\int_{I_n} 
\big((\Id - \ItR) \big((\Nabla \u) \u \big), \vht \big)_{\Qn} \Big|
\lesssim \tau_n^{\ell + 1} \Norm{\dpt^{(\ell + 1)} ((\nabla \u) \u)}{L^1(\In; L^2(\Omega)^d)} \Norm{\vht}{L^{\infty}(\In; L^2(\Omega)^d)} \, . 
\end{aligned}
\end{equation}
}
\end{comment}

\begin{alignat*}{3}
J_1^{(n)} & = \Big| %\int_{I_n} 
\big((\Id - \ItR) \big((\Nabla \u) \u \big), \vht \big)_{\Qn} \Big|
\le \Norm{(\Id - \ItR)((\nabla \u) \u)}{L^1(\In; L^2(\Omega)^d)} \Norm{\vht}{L^{\infty}(\In; L^2(\Omega)^d)} \\
%%%
& \lesssim \tau_n^{\ell+3/2}  \Norm{\dpt^{(\ell+1)} ((\nabla \u) \u)}{L^{2}(\In; L^2(\Omega)^d)} \Norm{\vht}{L^{\infty}(\In; L^2(\Omega)^d)} \\
%%%
& {\lesssim \tau_n^{\ell + 2}  \Norm{\dpt^{(\ell+1)} ((\nabla \u) \u)}{L^{\infty}(\In; L^2(\Omega)^d)} \Norm{\vht}{L^{\infty}(\In; L^2(\Omega)^d)} }\\
%%%
& \lesssim \varepsilon^{-1} {\cal R}_{\u} {\tau_n^{2\ell+3}} 
+ \varepsilon \tau_n \Norm{\vht}{L^{\infty}(\In; L^2(\Omega)^d)}^2 \, .
\end{alignat*}
Adding and subtracting suitable terms, using the antisymmetry of the form~$\chtbh(\cdot; \cdot, \cdot)$, and the fact 
that~$\chth(\u; (\Id - \Pt) \u, \vht) = 0$ since also $\Pt\u$ is continuous in space, we get
\begin{alignat*}{3}
\nonumber
J_2^{(n)} & = |\QtRn \big(\chtbh(\u; \epi, \vht) \big) | \\
%%%
\nonumber
& \le |\QtRn \big(\chtbh(\u; (\Id - \Pt) \u, \vht) \big)| + |\QtRn \big( \chtbh(\u; \Pt (\Id - \IRT) \u, \vht) \big)| \\
%%%
\nonumber
& = |\QtRn \big(\chtb(\u; (\Id - \Pt) \u, \vht) \big)| + |\QtRn \big( \chtbh(\u; \vht, \Pt (\Id - \IRT) \u) \big)| \\
%%%
\nonumber
& \le |\QtRn \big(  [(\Id - \Pt)\Nabla \u] \u, \vht)_{\Omega} \big)|  + |\QtRn \big( ( (\nabla \vht) \u, \Pt (\Id - \IRT) \u)_{\Omega} \big)| \\
\nonumber
& \quad + |\QtRn\big( ((\u \cdot \bnF ) \jump{\vht}, \mvl{\Pt (\Id - \IRT) \u})_{\Omega} \big) | \\
%%%
& =: J_{2, 1}^{(n)} + J_{2, 2}^{(n)} + J_{2, 3}^{(n)}.
\end{alignat*}

The term~$J_{2, 1}^{(n)}$ can be easily estimated using the definition of the quadrature rule, the H\"older inequality, and the approximation properties in Lemma~\ref{lemma:estimates-Ptau} of~$\Pt$, as follows:
\begin{alignat*}{3}
% |\QtRn \big(((\Id - \Pt) (\Nabla \u ) \u, \vht)_{\Omega} \big)|
J_{2, 1}^{(n)}
& \lesssim \tau_n \Norm{(\Id - \Pt) \Nabla \u}{L^{\infty}(\In; L^2(\Omega)^{d \times d})}
\Norm{\u}{L^{\infty}(\In; L^{\infty}(\Omega)^d)} \Norm{\vht}{L^{\infty}(\In; L^2(\Omega)^d)} \\
& \lesssim \varepsilon^{-1} \tau_n^{2\ell + 3} \Ru + \varepsilon \tau_n \Norm{\vht}{L^{\infty}(\In; L^2(\Omega)^d)}^2.
\end{alignat*}

\begin{comment}
For the second term, by exploiting the antisymmetry of~$\cht^{(1)+(2)}(\cdot; \cdot, \cdot)$, we get
\begin{equation}
\begin{aligned}
J_2^{(n)} & = \big| \QtRn \big( \cht^{(1)+(2)} (\u; \vht, \epi) \big) \big| \\
& \leq \big| \QtRn \big( \big( (\Nabla_h \vht) \u, \epi \big)_{\Omega} \big) \big| 
+ \big| \QtRn \big( \big((\u \cdot \bnF) \jump{\vht}, \mvl{\epi} \big)_{\FhI} \big) \big| \\
& =: J_{2,1}^{(n)} + J_{2,2}^{(n)} \, .
\end{aligned}
\end{equation}
\end{comment}

Let now $\overline{\u}$ denote the $L^2(\Omega)^d$-orthogonal projection of~$u$, at each time instant, on~$\Pp{0}{\Th}$.
Let furthermore $\widehat{\epi} := \Pt (\Id - \IRT) \u$ and note that, combining Lemma \ref{lemma:estimates-IRT} and
Lemma \ref{lemma:stab-Pt}, one easily obtains
\begin{equation}\label{eq:wepi-bound}
\| \widehat{\epi} \|_{L^\infty(\In; H^m(K)^d)} 
\lesssim h_K^{k+1-m} \| \u \|_{L^\infty(\In; H^{k+1}(K)^d)} \, , 
\quad m \in \{ 0,1 \} \, , K \in \Th \, .
\end{equation}

We first make use of the orthogonality properties of the RT/BDM interpolant in~\eqref{def:IRT-2}, then recall the definition of $\QtRn$ and apply suitable H\"older inequalities in space, finally use standard approximation properties of piecewise constant polynomials in~$\Th$:
$$
\begin{aligned}
J_{2,2}^{(n)} & = \big| \QtRn \big( \big( (\Nabla_h \vht) (\u-\overline{\u}), \widehat{\epi} \big)_{\Omega} \big) \big| \\
& \lesssim \tau_n \sum_{K \in \Th} \Norm{\Nabla_h \vht }{L^\infty(\In; L^2(K)^{d\times d})} 
\Norm{ \u - \overline{\u}}{L^\infty(\In; L^\infty(K)^d)}
\Norm{\widehat{\epi}}{L^\infty(\In; L^2(K)^d)} \\
& \lesssim \tau_n \sum_{K \in \Th} \Norm{\Nabla_h \vht}{L^\infty(\In; L^2(K)^{d\times d})} 
\hK \Norm{\u}{L^\infty(\In; W^{1}_{\infty}(K)^d)}
\Norm{\widehat{\epi}}{L^\infty(\In; L^2(K)^d)} \, .
\end{aligned}
$$
An inverse estimate for polynomials in space, combined with bound \eqref{eq:wepi-bound}, yields
$$
\begin{aligned}
J_{2,2}^{(n)} & \lesssim 
\tau_n \sum_{K \in \Th} \Norm{\vht}{L^\infty(I_n;L^2(K)^d)} 
\Norm{\u}{L^\infty(\In; W^{1}_{\infty}(K)^d)}
\Norm{\widehat{\epi}}{L^\infty(\In; L^2(K)^d)} \\
%%%
& \lesssim \tau_n \Norm{\u}{L^\infty(\In; W^{1}_{\infty}(\Omega)^d)}
\sum_{K \in \Th} h_K^{k+1} \Norm{\u}{L^\infty(\In; H^{k+1}(K)^d)} 
\Norm{\vht}{L^\infty(\In; L^2(K)^d)} \\
%%%
& \lesssim {\cal R}_{\u} \tau_n h^{k+1} \,  \| \vht \|_{L^\infty(I_n;L^2(\Omega)^d)} \\
& \lesssim \ \varepsilon^{-1}  \tau_n h^{2k+2} 
{\cal R}_{\u}^2
+  \varepsilon \, \tau_n \| \vht \|_{L^\infty(I_n,L^2(\Omega)^d)}^2 \, ,
\end{aligned}
$$
where we have used the Young inequality in the last step.

We now make use of the definition of the Gauss-Radau integration rule, apply scaled trace inequalities combined with \eqref{eq:wepi-bound}, recall that all integration weights are bounded by $\tau_n$ up to a constant depending only on $\ell$. Including all the regularity terms of $\u$ in ${\cal R}_{\u}$ as usual, which may change at each occurrence, we finally obtain
\begin{alignat*}{3}
J_{2,3}^{(n)} & = \big| \QtRn \big( \big((\u \cdot \bnF) \jump{\vht}, \mvl{\widehat{\epi}} \big)_{\FhI} \big) \big| \\
& \leq \| \u \|_{L^\infty(I_n;L^\infty(\Omega)^d)} 
\Big[ \QtRn\big( \|\jump{\vht}\|_{L^2(\FhI)^d}^2 \big) \Big]^{1/2}
\Big[ \QtRn\big( \|\mvl{\widehat{\epi}} \|_{L^2(\FhI)^d}^2 \big) \Big]^{1/2} \\
& \lesssim {\cal R}_{\u} \tau_n^{1/2} 
\Big[ \QtRn\big( \|\jump{\vht}\|_{L^2(\FhI)^d}^2 \big) \Big]^{1/2} \\
& \quad \times \Big[ \tau_n \sum_{K \in \Th} \Big( h_K^{-1/2} \| \widehat{\epi} \|_{L^\infty(I_n;L^2(K)^d)}^2 
+ h_K^{1/2} | \widehat{\epi} |_{L^\infty(I_n;H^1(K)^d)}^2\Big) \Big]^{1/2} \\
& \lesssim {\cal R}_{\u} \tau_n^{1/2} h^{k+1/2}  
\Big[ \QtRn\Big( \|\jump{\vht}\|_{L^2(\FhI)^d}^2 \Big) \Big]^{1/2} \\
& \lesssim \varepsilon^{-1} {\cal R}_{\u} \tau_n \, h^{2k+1}
+ \varepsilon \sum_{F \in \FhI} \QtRn\Big( \gamma_F(\uht) \, \|\jump{\vht}\|_{L^2(F)^d}^2 \Big) \, ,
\end{alignat*}
where the last bound follows from \eqref{eq:equiv:jmp} and the Young inequality.

By the triangle inequality,
\begin{equation*} %\label{eq:set:2}
J_3^{(n)} \leq \big| \QtRn \big( \chtb (\Piht \eu; \Piht \u, \vht) \big) \big| + \big| \QtRn \big( \chth (\Piht \eu; \Piht \u, \vht) \big) \big|
=: J_{3,1}^{(n)} + J_{3,2}^{(n)} \, .
\end{equation*}
Again by the definition of the quadrature rule and suitable H\"older inequalities, also recalling the stability properties in Lemma~\ref{lemma:stab-Pt} of~$\Pt$, it follows %the $\Piht$ operator, it follows \cite{Boffi_Brezzi_Fortin:2013, Ern_Guermond-I:2021}
$$
\begin{aligned}
J_{3,1}^{(n)} & = \big| \QtRn \big( (\Nabla_h \Piht \u) \Piht \eu, \vht \big)_{\Omega} \big| \\
& \le \tau_n \| \Nabla_h \Piht \u \|_{L^\infty(I_n;L^\infty(\Omega)^{d \times d})}
\| \Piht \eu \|_{L^\infty(I_n;L^2(\Omega)^d)}
\| \vht \|_{L^\infty(I_n;L^2(\Omega)^d)} \\
& \le \tau_n  \widehat{C} \| \u \|_{L^\infty(I_n;W^{1}_{\infty}(\Omega)^d)}
\| \Piht \eu \|_{L^\infty(I_n;L^2(\Omega)^d)} 
\| \vht \|_{L^\infty(I_n;L^2(\Omega)^d)},
\end{aligned}
$$
where the constant $\widehat{C}$ depends only on $k,\ell, \rho$, and we have used the stability 
property~$\Norm{\Nablah \IRT \w}{L^{\infty}(\Omega)^{d \times d}} \lesssim \Norm{\Nabla \w}{L^{\infty}(\Omega)^{d \times d}}$ 
(see, for instance, \cite[Thm.~16.4]{Ern_Guermond-I:2021}).

Using that $\Pt \u$ has vanishing spatial jumps and applying suitable H\"older inequalities and the definition of the Gauss-Radau quadrature, we get
\begin{equation}\label{eq:set:1}
\begin{aligned}
J_{3,2}^{(n)} & = 
\big| \QtRn \big( (\Piht \eu \cdot \bnF) \jump{\Piht \u}, \mvl{\vht} \big)_{\FhI} \big| \\
& =
\big| \QtRn \big( (\Piht \eu \cdot \bnF) \jump{\Piht \u - \Pt\u}, \mvl{\vht} \big)_{\FhI} \big| \\
& \lesssim \tau_n \sum_{F \in \FhI} \| \Piht \eu \|_{L^\infty(I_n;L^2(F)^d)}
\| \jump{\Pt ({\Id} - \IRT) \u} \|_{L^\infty(I_n;L^\infty(F)^d)}
\| \mvl{\vht} \|_{L^\infty(I_n;L^2(F)^d)} \, .
\end{aligned}
\end{equation}
We now use the boundedness in $L^\infty(I_n; Z)$ of the $\Pt$ operator (see Lemma~\ref{lemma:stab-Pt}), then trivial $L^\infty$ trace inequalities combined with approximation properties of $\IRT$, to conclude (for any $F \in \FhI$)
$$
\| \jump{\Pt ({\Id} - \IRT) \u} \|_{L^\infty(I_n;L^\infty(F)^d)} 
\le \widehat{C} h_F \| \u \|_{L^\infty(I_n;W^{1}_{\infty}(\omega_F)^d)} \, ,
$$
where $\omega_F$ denotes the union of the two triangles sharing the facet~$F$ and, as above, the constant $\widehat{C}$ depends only on $k,\ell,\rho$.
We now apply the latter bound in \eqref{eq:set:1}, make use of scaled trace inequalities and inverse estimates for polynomials, thus obtaining
$$
\begin{aligned}
J_{3,2}^{(n)} & \lesssim \tau_n 
\| \u \|_{L^\infty(I_n;W^{1}_{\infty}(\Omega)^d)}
\sum_{F \in \FhI} \| \Piht \eu \|_{L^\infty(I_n;L^2(\omega_F)^d)}
\| \vht \|_{L^\infty(I_n;L^2(\omega_F)^d)} \\
& \lesssim \tau_n \widehat{C}  \| \u \|_{L^\infty(I_n;W^{1}_{\infty}(\textcolor{magenta}{\Omega})^d)}
\| \Piht \eu \|_{L^\infty(I_n;L^2(\Omega)^d)}
\| \vht \|_{L^\infty(I_n;L^2(\Omega)^d)} \, .
\end{aligned}
$$

Addendum $J_{4}^{(n)}$ is bounded by similar arguments as for $J_{3}^{(n)}$, with the advantage that the $(\u - \Piht \u)$ term in the first entry is easier to handle using \eqref{approx-estimate-st}. 
Without showing the details, we obtain
$$
\begin{aligned}
J_{4}^{(n)} & \lesssim 
\tau_n \| \u \|_{L^\infty(I_n;W^{1}_{\infty}(\Omega)^d)} 
\Big( h^{k+1} \| \u \|_{L^\infty(I_n;H^{k+1}(\Omega)^d)} 
+ \tau_n^{\ell+1} \| \dpt^{(\ell + 1)} \u \|_{L^{\infty}(I_n;L^2(\Omega)^d)} \Big)
\| \vht \|_{L^\infty(I_n;L^2(\Omega)^d)} \\
& \lesssim  \varepsilon^{-1} \, \tau_n \, (h^{2k+2} + \tau_n^{2\ell+2}) {\cal R}_{\u}
+ \varepsilon \, \tau_n \| \vht \|_{L^\infty(I_n;L^2(\Omega)^d)}^2 \, .
\end{aligned}
$$

Since $\Pt \u$ has vanishing spatial jumps, we obtain
\begin{alignat}{3}
\nonumber
J_{5}^{(n)} & = 
\frac12 \big| \QtRn \big(\gamma(\uht) \jump{\Piht \u - \Pt \u}, \jump{\vht} \big)_{\FhI} \big| \\
%%%
\nonumber
& \lesssim  
\Big[\QtRn\Big(
\sum_{F \in \FhI} \gamma_F(\uht) \| \jump{\Pt ({\Id} - \IRT) \u} \|_{L^2(F)}^2 
\Big)\Big]^{1/2} 
\Big[\QtRn\Big( 
\sum_{F \in \FhI} \gamma_F(\uht) \, |\jump{\vht}|_{L^2(F)^d}^2 
\Big)\Big]^{1/2} 
\\
%%%
\nonumber
& \lesssim \tau_n^{1/2} \Big[ \sum_{F \in \FhI}  
\max{ \{c_S, \| \uht \|_{L^\infty(I_n;L^\infty(F)^d)} \} }
\| \jump{\Pt ({\Id} - \IRT) \u} \|_{L^\infty(I_n;L^2(F)^d)}^2
\Big]^{1/2} \\
\label{eq:new:1}
& \quad \times \Big[ \sum_{F \in \FhI} \QtRn\Big( \gamma_F(\uht) \, \|\jump{\vht}\|_{L^2(F)^d}^2 \Big) \Big]^{1/2} \, .
\end{alignat}
We now write
\begin{equation}\label{eq:Tsplit}
\sum_{F \in \FhI} \max{ \{c_S, \| \uht \|_{L^\infty(I_n;L^\infty(F)^d)} \} }
\| \jump{\Pt ({\Id} - \IRT) \u} \|_{L^\infty(I_n;L^2(F)^d)}^2 
\le J_1' + J_2' \, ,
\end{equation}
where 
\begin{equation}
\label{eq:J1p-J2p}
\begin{aligned}
& J_1' : = c_S \sum_{F \in \FhI} \| \jump{\Pt ({\Id} - \IRT) \u} \|_{L^\infty(I_n;L^2(F)^d)}^2 \, , \\
& J_2' := \sum_{F \in \FhI} \| \uht \|_{L^\infty(I_n;L^\infty(F)^d)} \| \jump{\Pt ({\Id} - \IRT) \u} \|_{L^\infty(I_n;L^2(F)^d)}^2  \, .
\end{aligned}
\end{equation}
The term $J_1'$ can be easily bounded by the stability of $\Pt$ in $L^\infty(I_n; Z)$,
classical trace inequalities and approximation estimates for $\IRT$, leading to
\begin{equation}\label{eq:case-T1}
J_1' \lesssim h^{2k+1} {\cal R}_{\u} \, .
\end{equation}
We need to handle $J_2'$. We recall the boundedness in $L^\infty(I_n; Z)$ of the $\Pt$ operator (see Lemma~\ref{lemma:stab-Pt}), apply scaled trace inequalities and approximation properties of $\IRT$, finally make use of a classical H\"older inequality. We obtain (for any $F \in \FhI$)
$$
\| \jump{\Pt ({\Id} - \IRT) \u} \|_{L^\infty(I_n;L^2(F)^d)}^2 
\lesssim h_F^{2k+1} \| \u \|_{L^\infty(I_n;H^{k+1}(\omega_F)^d)}^2
\lesssim h_F^{2k+1} |\omega_F|^{1/2} 
\| \u \|_{L^\infty(I_n;W^{k+1}_{4}(\omega_F)^d)}^2 \, .
$$
Inserting the above bound into the definition of $J_2'$, by an inverse inequality in space and a discrete H\"older inequality (with $p=p'=2$) it follows
\begin{alignat}{3}
\nonumber
J_2' & \lesssim 
\sum_{F \in \FhI} \| \uht \|_{L^\infty(I_n;L^\infty(\omega_F)^d)}  
h_F^{2k+1} |\omega_F|^{1/2} \| \u \|_{L^\infty(I_n;W^{k+1}_{4}(\omega_F)^d)}^2 \\
%%%
\nonumber
& \lesssim \sum_{F \in \FhI} \| \uht \|_{L^\infty(I_n;L^2(\omega_F)^d)}  
h_F^{2k+1} \| \u \|_{L^\infty(I_n;W^{k+1}_{4}(\omega_F)^d)}^2 \\
%%%
\label{eq:case-T2}
& \lesssim h^{2k+1} \| \uht \|_{L^\infty(I_n; L^2(\Omega)^d)}
\| \u \|_{L^\infty(I_n;W^{k+1}_{4}(\Omega)^d)}^2
\lesssim h^{2k+1} {\cal R}_{\u} \, ,
\end{alignat}
where, in the last step, we also used that 
$\| \uht \|_{L^\infty(I_n;L^2(\Omega)^d)}$ is bounded uniformly in~$h$ and~$\tau$, see Proposition \ref{prop:continuous-dependence}.
Combining \eqref{eq:Tsplit}, \eqref{eq:case-T1} and \eqref{eq:case-T2} into \eqref{eq:new:1}, a Young inequality immediately yields
$$
J_{5}^{(n)} \lesssim 
\varepsilon^{-1} \tau_n h^{2k+1} {\cal R}_{\u}
+ \varepsilon \, \sum_{F \in \FhI} \QtRn\Big( \gamma_F(\uht) \, \|\jump{\vht}\|_{L^2(F)^d}^2 \Big) \, .
$$
The proof of the lemma now follows easily by combining \eqref{eq:conv:start} with the bounds for the $J_i^{(n)}$ terms, and using \eqref{eq:vh:equiv} to substitute $\vht$ with $\Piht \eu$ inside all the norms and seminorms.
\end{proof}

% --------------------------------------------------------
\subsection{Error estimates} 
% --------------------------------------------------------
\begin{comment}
We first recall the following discrete Gr\"onwall lemma from~\cite[Lemma~5.1]{Heywood_Rannacher:1990}.
\begin{lemma}[Discrete Gr\"onwall]
Let~$\{a_n\}$, $\{b_n\}$, $\{c_n\}$, $\{\gamma_n\}$, and~$\{\tau_n\}$ be sequences of nonnegative numbers, and let~$B \geq 0$ be a constant such that, for each~$n \in \IN$, it holds
\begin{equation*}
a_n + \sum_{m = 0}^n \tau_m b_m \le \sum_{m = 0}^n \tau_m \gamma_m a_m + \sum_{m = 0}^n k_m c_m + B,
\end{equation*}
and~$\tau_m \gamma_m < 1$ for all~$m \in \{0, \ldots, n\}$. Then, with~$\sigma_m := (1 - \tau_m \gamma_m)^{-1}$, it holds
\begin{equation*}
a_n + \sum_{m = 0}^n \tau_m b_m \le \exp \Big(\sum_{m = 0}^n \tau_m \sigma_m \gamma_m \Big) \Big(\sum_{m = 0}^n \tau_m c_m + B \Big).
\end{equation*}
\end{lemma}
\end{comment}

%We have the following result.
We are now in a position to prove the main result in this section. In the following $\tau := \max_{\{n=1,\ldots,N\}}{\tau_n}$ will denote the maximum time mesh size.

\begin{theorem}[Estimate for the discrete error]\label{prop:final}
Let $\u$ denote the velocity solution to the continuous weak formulation~\eqref{eq:weak-formulation} and let~$\uht \in \Zht$ be the discrete solution to the space--time formulation~\eqref{eq:kernel-space-time-problem}. Let $\u\in W^{1}_{\infty}(0,T; H^{k+1}(\Omega)^d) \cap L^{\infty}(0, T; W^{k+1}_{4}(\Omega)^d) \cap W^{\ell + 1}_{\infty}(0, T; H^2(\Omega)^d)$, and~$(\Nabla \u) \u \in W^{\ell + 1}_\infty (0, T; L^2(\Omega)^d)$.
Then, there exists a positive constant~$C_{\ell}$ depending only on~$\ell$ such that, if~${\tau_n}\le C_\ell/(2 C_{\u}^{\star})$ for all~$n \in \{1,2,\ldots,N\}$ with~$C_{\u}^{\star}$ as in Lemma~\ref{lemma:cons-conv}, it holds
%%
%Furthermore, we assume~$\tau := \max_{\{n=1,..,N\}}{\tau_n} \le {C_\ell/(2 C_{\u}^{\star})}$ for all~$n \in \{1,2,..,N\}$, {where~$C_{\u}^{\star}$ is as in Lemma~\ref{lemma:cons-conv}, and~$C_{\ell}$ is a constant depending only on~$\ell$}. %where the positive constant~$c$ only depends on~$k,\ell$, and~$\rho$.
% Then, it holds
%%
\begin{equation}\label{bound:final}    
\begin{aligned}
\| \Piht \eu \|_{L^\infty(0,T;L^2(\Omega)^d)}^2 & + \nu \sum_{n = 1}^N \int_{I_n} \Norm{\Piht \eu(\cdot, t) }{\calA, h}^2 \dt + \sum_{n = 1}^{N-1} \Norm{\jump{\Piht \eu}_{n}}{L^2(\Omega)^d}^2 \\ 
&  + \sum_{n = 1}^{N} \sum_{F \in \FhI} \QtRn\Big( \int_F  \gamma_F(\uht) \, |\jump{\Piht \eu}|^2 \dS \Big) \\
& \lesssim \, \Big(1 + \exp\Big(\frac{C_{\u}^{\star}}{C_{\ell}} T\Big)\Big) \Bigg[\sum_{n=1}^N \Big( {\tau_n^{2 \ell + 3}} + \tau_n (\nu+h) h^{2k} \Big) + h^{2k+2} \Bigg] \\
& \lesssim \, \Big(1 + \exp\Big(\frac{C_{\u}^{\star}}{C_{\ell}} T\Big)\Big)
 \big( {\tau^{2 \ell+2}} + (\nu + h) h^{2k}\big) \, ,
\end{aligned}
\end{equation}
with hidden constant that depends only on $k,\ell$, $\rho$, and on the solution~$\u$ (evaluated in the norms of the regularity assumption above).
%and exponentially on $T$. 
\end{theorem}
\begin{proof}
We start considering, for any $n \in \{1,\ldots,N\}$, the test function $\vht \in \Zht$ such that %$\vht(\cdot,t) = \Piht \eu (\cdot, t)$ for all $t \in [0,t_n]$ and vanishing for all $t > t_n$
\begin{equation}
\label{eq:test-function-error}
\vht{}_{|_{\Qm}} = \begin{cases}
\Piht \eu {}_{|_{\Qm}} & \text{ if } 1 \le m \le n, \\
0 & \text{otherwise}.
\end{cases}.
\end{equation}

By the same argument {as} in the proof of Lemma \ref{lemma:weak-bound}, we obtain
\begin{equation}\label{finals:1}
\Bht(\uht; \Piht \eu,\vht)  \geq E_n  \, ,
\end{equation}
where the quantity
\begin{equation}\label{finals:E1}
\begin{aligned}
E_n & :=  
{\frac12} \Norm{\Piht \eu(\cdot, \tn^{-})}{L^2(\Omega)^d}^2  + \nu \sum_{m = 1}^n \int_{I_m} \Norm{\Piht \eu(\cdot, t) }{\calA, h}^2 \dt  + {\frac12} \sum_{m = 1}^{n - 1} \Norm{\jump{\Piht \eu}_{m}}{L^2(\Omega)^d}^2 \\
& \quad + \frac14 \Norm{\Piht\eu(\cdot, 0)}{L^2(\Omega)^d}^2
 + {\frac12} \sum_{m = 1}^{n} \sum_{F \in \FhI} \QtRm\Big( \int_F  \gamma_F(\uht) \, |\jump{\Piht \eu}|^2 \dS \Big) \, .
\end{aligned}
\end{equation}
We now apply identity \eqref{eq:discrete-error-equation-simplified} and the Lemmas \ref{lemma:cons-time}, \ref{lemma:cons-diff} and \ref{lemma:cons-conv} to bound the consistency terms in $\Bht(\uht; \Piht \eu,\vht)$ for time derivative, diffusion, and convection, respectively. The initial data error term is handled trivially by the Cauchy--Schwarz inequality and Lemma \ref{lemma:estimates-IRT}. We obtain, for any positive real $\varepsilon$,
\begin{alignat}{3}\label{finals:2}
\nonumber
& \Bht(\uht; \Piht\eu,\vht) \le 
C \varepsilon^{-1} {\cal R}_{\u} \sum_{m=1}^n \Big( {\tau_m^{2 \ell + 3}} + \tau_m (\nu+h) h^{2k} \Big) + C \varepsilon^{-1} {\cal R}_{\u} h^{2k+2} \\
\nonumber
& \ 
+ C\varepsilon \sum_{m=1}^n \Big[ \tau_m \| \Piht \eu \|_{L^\infty(I_m;L^2(\Omega)^d)}^2  \ + \nu \int_{I_m} \!\! \Norm{\Piht \eu(\cdot, t) }{\calA, h}^2
+ \!\!\!\! \sum_{F \in \FhI} \!\!\! \QtRm\Big( \int_F  \gamma_F(\uht) \, |\jump{\Piht \eu}|^2 \Big)
\Big] \\
& \ + C \varepsilon \Norm{\Piht \eu(\cdot, 0)}{L^2(\Omega)^d}^2 {+ C_{\u}^{\star}} \sum_{m=1}^n \tau_m \| \Piht \eu \|_{L^\infty(I_m;L^2(\Omega)^d)}^2\, ,
\end{alignat}
%with the usual meaning for~${\cal R}_{\u}$ and where the hidden constant only depends on $k,\ell$ and the mesh shape-regularity parameter $\rho$. 
{where~${\cal R}_{\u}$ has the usual meaning, the constant~$C$ depends only on~$k,\ell$, and the shape-regularity parameter~$\rho$, and~$C_{\u}^{\star}$ is as in Lemma~\ref{lemma:cons-conv}.}
Here above, for conciseness of exposition and without loss of generality, we assumed $\nu \lesssim 1$.
Combining \eqref{finals:1} and \eqref{finals:2}, for $\varepsilon$ sufficiently small (but depending only on $k$, $\ell$, and~$\rho$), we obtain
\begin{equation}\label{finals:3}
% \begin{aligned}
E_n \le 
C {\cal R}_{\u} \sum_{m=1}^n \Big( {\tau_m^{2 \ell + 3}} + \tau_m (\nu+h) h^{2k} \Big) + C {\cal R}_{\u} h^{2k+2}
+ {C_{\u}^{\star}} \sum_{m=1}^n \tau_m \| \Piht \eu \|_{L^\infty(I_m;L^2(\Omega)^d)}^2 \, .
% \end{aligned}
\end{equation}

We now consider, for any $n \in \{1,2,\ldots,N\}$, the following test function $\vht^{(n)}$ (c.f. \eqref{lambdadef}):
\begin{equation}
\label{eq:super-test-function-error}
\vht^{(n)} {}_{|_{Q_m}}:= \begin{cases}
\ItR(\varphi_n \Piht\eu) & \text{ if } m = n, \\
0 & \text{ otherwise}.
\end{cases}
\end{equation}
Following the same steps as in Proposition \ref{prop:continuous-dependence} (including the inverse estimate in time \eqref{eq:L2-Linfty-uht}, now applied to $\Piht\eu$), we obtain
\begin{equation} \label{finals:4}
\Bht(\uht; \Piht\eu,\vht^{(n)}) \ge {C_{\ell}} \, \widehat{E}_n - \frac{1}{4} \Norm{\Piht \eu(\cdot, t_{n-1}^{-})}{L^2(\Omega)^d}^2 \, ,
\end{equation}
{where the positive constant~$C_{\ell}$ only depends on~$\ell$,} %$k,\ell$, and~$\rho$, 
and the quantity
$$
\begin{aligned}
\widehat{E}_n := &  
\| \Piht \eu \|_{L^\infty(I_n;L^2(\Omega)^d)}^2
+ \nu \int_{I_n} \Norm{\Piht \eu(\cdot, t) }{\calA, h}^2 \dt  \\
& + \Norm{\jump{\Piht \eu}_{n-1}}{L^2(\Omega)^d}^2 
+ \sum_{F \in \FhI} \QtRn\Big( \int_F  \gamma_F(\uht) \, |\jump{\Piht \eu}|^2 \dS \Big) \, .
\end{aligned}
$$
Similarly as above, we now apply identity \eqref{eq:discrete-error-equation-simplified} and Lemmas \ref{lemma:cons-time}, \ref{lemma:cons-diff}, and \ref{lemma:cons-conv} to bound the consistency terms in $\Bht(\uht; \Piht \eu,\vht^{(n)})$. We obtain
\begin{equation}\label{finals:5}
\begin{aligned}
\Bht(\uht; \Piht\eu,\vht^{(n)}) &  \le
C \varepsilon^{-1} {\cal R}_{\u} \Big( {\tau_n^{2 \ell + 3}} + \tau_n (\nu+h) h^{2k} \Big) \\
& \quad 
+ \varepsilon \Big[ \tau_n \| \Piht \eu \|_{L^\infty(I_n;L^2(\Omega)^d)}^2  + \nu \int_{I_n} \!\! \Norm{\Piht \eu(\cdot, t) }{\calA, h}^2 \\
& \quad 
+ \!\!\!\! \sum_{F \in \FhI} \!\!\! \QtRn\Big( \int_F  \gamma_F(\uht) \, |\jump{\Piht \eu}|^2 \Big)
\Big] + {C_{\u}^{\star} }\tau_n \| \Piht \eu \|_{L^\infty(I_n;L^2(\Omega)^d)}^2 \, ,
% + \varepsilon \Norm{\Piht \eu(\cdot, 0)}{L^2(\Omega)^d}^2 
\end{aligned}
\end{equation}
where, for simplicity of exposition, we are considering the case $n > 1$. In the specific case $n=1$ also an additional term of order $h^{2k+2}$, associated with the approximation of the initial data on $\Sigma_0$, is present; since such a term is handled very easily (in an analogous way as shown for the previous test function), we omit the related details.  
Combining \eqref{finals:4} and \eqref{finals:5}, for $\varepsilon$ sufficiently small (but depending only on $k,\ell$, and~$\rho$) yields
\begin{equation*} %\label{finals:6}
\begin{aligned}
\| \Piht \eu \|_{L^\infty(I_n;L^2(\Omega)^d)}^2  \le \widehat{E}_n & \le 
C {\cal R}_{\u} \Big( {\tau_m^{2 \ell + 3}} + \tau_m (\nu+h) h^{2k} \Big) + \frac{1}{4{C_{\ell}}} \Norm{\Piht \eu(\cdot, t_{n-1}^{-})}{L^2(\Omega)^d}^2 \\
& \quad + \tau_n {\frac{C_{\u}^{\star}}{C_{\ell}}} \| \Piht \eu \|_{L^\infty(I_n;L^2(\Omega)^d)}^2
\, .
\end{aligned}
\end{equation*}
{For $\tau_n < C_\ell/(2C_{\u}^{\star})$}, %sufficiently small (but only depending on $k,\ell,\rho$) {\color{red} [Possiamo essere piu espliciti? controllare la costante nella dipendenza]} 
the above bound becomes
\begin{equation}\label{finals:7}
\| \Piht \eu \|_{L^\infty(I_n;L^2(\Omega)^d)}^2 \le
{2 C} {\cal R}_{\u} \Big( {\tau_m^{2 \ell + 3}} + \tau_m (\nu+h) h^{2k} \Big) 
{+ \frac{1}{2C_\ell}} \Norm{\Piht \eu(\cdot, t_{n-1}^{-})}{L^2(\Omega)^d}^2 \, .
\end{equation}
The above bound is instrumental to derive the desired estimate.
% , and can be easily obtained for the lowest order cases~$(\ell = 0, 1)$.}
%%
We now apply \eqref{finals:7} for variable index $m$ in the last sum of \eqref{finals:3}, easily obtaining
\begin{equation*} %\label{finals:8}
% \begin{aligned}
E_n \le
C {\cal R}_{\u} \sum_{m=1}^n \Big( {\tau_m^{2 \ell + 3}} + \tau_m (\nu+h) h^{2k} \Big) + C {\cal R}_{\u} h^{2k+2}
{+ \frac{C_{\u}^{\star}}{2C_\ell}}\sum_{m=1}^n \tau_m \Norm{\Piht \eu(\cdot, t_{m-1}^{-})}{L^2(\Omega)^d}^2 \, ,
% \end{aligned}
\end{equation*}
which holds for all $n \in \{1,2,\ldots, N\}$. Recalling definition \eqref{finals:E1}, which implies $\Norm{\Piht \eu(\cdot, t_{m - 1}^-)}{L^2(\Omega)^d}^2 \le 2 E_{m-1}$, an application of the discrete Gr\"onwall inequality yields the following  bound for the error norm $E_n$: 
\begin{equation*}
E_n \le \exp\Big(\frac{C_{\u}^{\star} }{C_\ell} T \Big) \Big(C {\cal R}_{\u} \sum_{m=1}^n \Big( {\tau_m^{2 \ell + 3}} + \tau_m (\nu+h) h^{2k} \Big) + C {\cal R}_{\u} h^{2k+2} \Big).
\end{equation*}
Finally, the bound %in the 
{for the term}~$\| \Piht \eu \|_{L^\infty(I_n;L^2(\Omega)^d)}^2$, %norm, 
$n \in \{1,2,\ldots, N\}$, is recovered applying again \eqref{finals:7}.
\end{proof}

\begin{corollary}[\emph{A priori} error estimate]\label{corol:final}
Under the same assumptions as in Theorem \ref{prop:final}, it holds 
\begin{equation}
\begin{aligned}
\| \eu \|_{L^\infty(0,T;L^2(\Omega)^d)}^2 & + \nu \sum_{n = 1}^N \int_{I_n} \Norm{\eu(\cdot, t) }{\calA, h}^2 \dt 
+ \sum_{n = 1}^{N} \sum_{F \in \FhI} \QtRn\Big( \int_F  \gamma_F(\uht) \, |\jump{\eu}|^2 \dS \Big) \\
& \lesssim \, \sum_{n=1}^N \Big( {\tau_n^{2 \ell + 3}} + \tau_n (\nu+h) h^{2k} \Big) + h^{2k+2} 
\, \lesssim \, {\tau^{2 \ell+2}} + (\nu + h) h^{2k} \, .
\label{eqn:finalEst}
\end{aligned}
\end{equation}
\end{corollary}
\begin{proof}
Since the proof follows easily by a triangle inequality and approximation estimates for the $\Piht$ operator, we omit the details. We underline that the time jump term in~\eqref{bound:final} is purposefully not included in the estimate. Adding such a term would lead to an additional addendum on the right-hand side of the type 
$\sum_{n=1}^N (\tau_n^{-1}h^{2k+2})$.
\end{proof}

%%%
\begin{remark}[Pressure- and Reynolds-semi-robustness]\label{remark:final}
The a priori error estimate in Corollary~\ref{corol:final} 
enjoys the following notable properties: 
{\bf (1)} the estimate is optimal in both $h$ and $\tau$ in the diffusive regime, and acquires an additional $h^{1/2}$ pre-asymptotic error reduction rate in convection-dominated cases; 
{\bf (2)} the estimate is Reynolds semi-robust in the sense that the involved constant does not grow when $\nu \rightarrow 0$;
{\bf (3)} the error bound reflects the pressure-robustness of the scheme since it is independent of the pressure~$p$.
\eremk
\end{remark}

% ---------------------------------------------------------
%               SEMI-IMPLICIT
% ---------------------------------------------------------
\red{
\section{Semi-implicit DG time discretization\label{sec:semi-implicit}}
In this section, we introduce a novel high-order DG time discretization for the incompressible Navier--Stokes equations~\eqref{eq:model-problem}. 
The proposed scheme treats the nonlinear convective term semi-implicitly, which results in a significant reduction of the computational cost while maintaining the stability, robustness, and convergence properties of the fully implicit method in Section~\ref{subsec:fully-discrete-formulation}. 
We start by defining the space
$$
\Sht : = \Vh \otimes P_\ell(0,T) = \big\{ \vht :\Omega \times [0,T] \rightarrow \IR^3 \ : \ 
\vht(\cdot,t) = \sum_{i=0}^\ell {\bf v}_{h,i} \, t^i \ \textrm{ with } \ {\bf v}_{h,i} \in \Vh \big\} \, .
$$
Given any function~$\vht^{(n)} = \vht{}_{|_{\Qn}}$ with $\vht \in \Vht$ and $n \in \{1,2,\ldots,N\}$, we can consider its natural extension (exploiting that the function $\vht^{(n)}$ is polynomial in time) to the whole space--time cylinder~$\QT$. For convenience, and with a small abuse of notation, in this section, we use the same symbol~$\vht^{(n)}$ also to denote such an extension.
Given any function $\vht \in \Vht$ we denote by $\vhtt$ the following function in $\Vht$
\begin{equation}\label{eq:tilded:def} % +++
\vhtt{}_{|_{\Qn}} := \begin{cases}
\vht^{(n)} {}_{|_{\Qn}} & \text{ if } n = 1, \\[1em]
\vht^{(n - 1)} {}_{|_{\Qn}} & \text{ if } n > 1. 
\end{cases} \qquad n \in \{1,2,\ldots,N\} \, .
\end{equation}
Note that if~$\vht \in \Zht$, then also~$\vhtt \in \Zht$.
Recalling the notation and definitions in Sections~\ref{subsec:mesh-notation}, \ref{subsec:discrete-spaces}, and~\ref{subsec:fully-discrete-formulation}, the proposed semi-implicit space--time formulation reads: find~$\uht \in \Zht$ such that
\begin{equation}
\label{eq:semi-implicit-formulation}
\begin{split}
\Bht(\uhtt; \uht, \zht) = \mht(\uht, \zht) & + \nu \aht(\uht, \zh) + \cht(\uhtt; \uht, \zh) \\
& = (\f, \zht)_{\QT} + (\uo, \zht(\cdot, 0))_{\Omega} \qquad \forall \zht \in \Zht,
\end{split}
\end{equation}
% where the discrete coefficient~$\uhtt \in \Zht$ is defined as follows:
% \begin{equation}
% \uhtt{}_{|_{\Qn}} := \begin{cases}
% \uht^{(n)} {}_{|_{\Qn}} & \text{ if } n = 1, \\[1em]
% \uht^{(n - 1)} {}_{|_{\Qn}} & \text{ if } n > 1. 
% \end{cases}
% \end{equation}

Upon careful perusal of \eqref{eq:semi-implicit-formulation} and \eqref{eq:tilded:def}, it becomes clear that, unlike~\eqref{eq:space-time-formulation}, the system in~\eqref{eq:semi-implicit-formulation} is linear in all time steps (excluding the first one).
We now present the following simple but critical lemma, stating some stability and approximation properties of the tilde operator in \eqref{eq:tilded:def}. 

\begin{lemma}\label{lemma:ext:tilded}
Let~$\eta$ denote $\max_{n=2,\ldots,N}\{ \tau_n/\tau_{n-1} \}$, and let~$(Z, \Norm{\cdot}{Z})$ be any Banach space such that~$Z \subseteq L^1(\Omega)$.
There exists a constant~$\Breve{C} = \Breve{C}(\ell,\eta)$ such that, for all~$\vht \in \Vht$ and~$\v \in L^\infty(0,T;Z)$, the following hold:
\begin{eqnarray}
\label{eq:X:1}
&& \| \vhtt \|_{L^\infty(I_n;Z)} \le \Breve{C} \| \vht \|_{L^\infty(I_{n-1};Z)} \qquad \forall n \in \{2,\ldots,N\} \, , \\
\label{eq:X:2}
&& \| \vhtt \|_{L^\infty(0,T;Z)} \le \Breve{C} \| \vht \|_{L^\infty(0,T;Z)} \, , \\
\label{eq:X:3}
&& \| \v - \vhtt \|_{L^\infty(I_n;Z)} 
\le \Breve{C} \Big( \| \v - \vht \|_{L^\infty(I_{n-1};Z)} + \| \v - \wht \|_{L^\infty(I_{n-1} \cup I_n;Z)} \Big) \, ,
% \quad \forall \wht \in \Vht \, .
\end{eqnarray}
the latter equation holding for any~$\wht$ in $\Sht$ and all~$n \in \{2,\ldots, N\}$.
\end{lemma}
\begin{proof}
Given $n \in \{2,\ldots, N\}$, let $\{ \hat{t}_i \}_{i=0}^{\ell}$ denote a set of distinct points in~$I_{n-1}$ such as, for example, the Gauss-Radau nodes. Let $\{ \phi_j \}_{j=0}^\ell$ denote the associated %canonical 
Lagrangian polynomial basis functions, that is, $\phi_j \in \Pp{\ell}{0, T}$ %(0, T)
and $\phi_j(\hat{t}_i)= \delta_{ij}$ for all~$i,j=0,1, \ldots, \ell$. %Let 
For any~$\vht \in \Vht$, we can write
$$
\vht(\cdot,t)_{|_{I_{n-1}}} = \sum_{j=0}^{\ell} {\bf v}_{h,j} \phi_j(t) \qquad 
\textrm{for all } t \in I_{n-1} \, ,
$$
with ${\bf v}_{h,j} = \vht(\cdot,\hat{t}_j) \in \Vh$ for all $j \in \{0,1,\ldots, \ell\}$. 
By definition~\eqref{eq:tilded:def}, 
$\vhtt(\cdot,t)_{|_{I_n}} = \sum_{j=0}^{\ell} {\bf v}_{h,j} \phi_j(t)$ for all $t \in \In$. 
By the triangle inequality and noting that all~$\| \phi_j(t) \|_{L^\infty(I_n)}$ are bounded by a uniform constant~$C$ only depending on~$\ell$ and $\eta$, %~$\nu$, 
we derive
$$
\| \vhtt \|_{L^\infty(I_n;Z)} \le \sum_{j=0}^\ell \| {\bf v}_{h,j} \|_Z \| \phi_j(t) \|_{L^\infty(I_n)}
\lesssim \max_{j=0,1,\ldots,\ell} \| {\bf v}_{h,j} \|_Z \le \| \vht \|_{L^\infty(I_{n-1};Z)} \, ,
$$
where the hidden constant depends only on~$\ell$ and~$\eta$. This proves \eqref{eq:X:1}, and bound~\eqref{eq:X:2} follows immediately from definition \eqref{eq:tilded:def} and~\eqref{eq:X:1}.

We now prove bound~\eqref{eq:X:3}. Let~$\wht$ in~$\Sht$ and~$n \in \{2,\ldots,N\}$. 
Since $\whtt=\wht$, using the triangle inequality and applying~\eqref{eq:X:1}, we easily obtain 
$$
\begin{aligned}
\| \v - \vhtt \|_{L^\infty(I_n;Z)} & \le \| \v - \wht \|_{L^\infty(I_n;Z)} + \| \vhtt - \whtt \|_{L^\infty(I_n;Z)} \\
& \lesssim \| \v - \wht \|_{L^\infty(I_n;Z)} + \| \vht - \wht \|_{L^\infty(I_{n-1};Z)} \\
& \le \| \v - \wht \|_{L^\infty(I_{n-1} \cup I_n;Z)} + \| \v - \vht \|_{L^\infty(I_{n-1};Z)} \, .
\end{aligned}
$$
\end{proof}
Combining a suitable choice of~$\wht$ in~\eqref{eq:X:3} with approximation properties (in the spirit of Lemma \ref{lemma:estimate-Piht}), we can easily obtain, whenever $\u$ is sufficiently regular for the right-hand side to make sense, 
\begin{equation}
\label{eq:aux-estimate-tilde}
\begin{aligned}
\| \v - \vhtt \|_{L^\infty(I_n;L^2(\Omega)^d)} & \lesssim
\| \v - \vht \|_{L^\infty(I_{n-1};L^2(\Omega)^d)} + 
h^{k+1} \Norm{\v}{L^\infty(\In\cup I_{n-1}; H^{k+1}(\Omega)^d)} \\
& \quad + \tau_n^{\ell + 1} \Norm{\v}{W^{\ell+1}_{\infty}(\In\cup I_{n-1};L^2(\Omega)^d)}
\qquad \forall n \in \{2,\ldots, N\} \, .
\end{aligned}
\end{equation} 
Note that the hidden constant here above depends, in addition to the usual quantities, also on the parameter $\eta$ introduced above.

The following result, which is analogous to Theorem~\ref{thm:existence-discrete-solutions} above, states the existence of discrete solutions to the space--time formulation~\eqref{eq:semi-implicit-formulation}, as well as its unconditional stability.
%%%
\begin{theorem}[Existence of discrete solutions]
\label{thm:existence-discrete-solutions-semi-implicit}
Given~$\f \in L^1(0, T; L^2(\Omega)^d)$ and~$\u_0 \in \Z$, there exists at least a solution~$\uht \in \Zht$ to the semi-implicit space--time formulation~\eqref{eq:semi-implicit-formulation}, which satisfies
\begin{equation}
\label{eq:continuous-dependence-nonlinear-semi-implicit}
\Tnorm{\uht}{\uhtt}^2 \le \Clin \big(\Norm{\f}{L^1(0, T; L^2(\Omega)^d)}^2 + \Norm{\u_0}{L^2(\Omega)^d}^2 \big).
\end{equation}
\end{theorem}
\begin{proof}
Similar to as in the proof of Theorem~\ref{thm:existence-discrete-solutions}, we define the ball
\begin{equation*}
\calB := \left\{\wht \in \Zht \, : \, \Tnorm{\wht}{{\bf 0}}^2 \le \Clin \big(\Norm{\f}{L^1(0, T; L^2(\Omega)^d)}^2 + \Norm{\u_0}{L^2(\Omega)^d}^2 \big) \right\},
\end{equation*}
and the map~$\Phi : \Zht \to \Zht$, which assigns to each~$\wht \in \Zht$ the element~$\uht \in \Zht$, where~$\uht$ is the solution to the linearized space--time formulation~\eqref{eq:kernel-linearized-problem} with discrete coefficient~$\whtt$.

Since~$\wht \in \Zht$, then~$\Div \whtt = 0$ in~$\QT$, which implies that~$\whtt \in \Zht$. Moreover, the following stability result holds (cf. Proposition~\ref{prop:continuous-dependence}): 
\begin{equation*}
\Tnorm{\uht}{\whtt}^2 \le \Clin \Norm{\f}{L^1(0, T; L^2(\Omega)^d)}^2 + \Norm{\u_0}{L^2(\Omega)^d}^2,
\end{equation*}
with the same constant~$\Clin$ as in Proposition~\ref{prop:continuous-dependence}.
Thus, the map~$\Phi$ is well defined, and since~$\Tnorm{\wht}{{\bf 0}} \le \Tnorm{\wht}{{\whtt}}$ for any $\wht \in \Zht$, we also have~$\Phi(\Zht) \subset \calB$. 
%Note that 
%The continuity of the map~$\Phi$ follows by standard arguments combined with the critical fact that the ``tilde'' operator \textcolor{blue}{defined in~\eqref{eq:tilded:def}} is continuous due to~\eqref{eq:X:2} \textcolor{blue}{in Lemma~\ref{lemma:ext:tilded}}. \textcolor{blue}

Since~$\Phi(\Zht) \subset \calB$, any sequence~$\{\wht^{(s)}\}_{s \in \IN} \subset \Zht$ converging to some~$\wht \in \Zht$ is such that the corresponding sequence~$\{\uht^{(s)}\}_{s \in \IN}$ is uniformly bounded in the norm~$\Tnorm{\cdot}{\bf 0}$. 
As a consequence, up to a subsequence, $\{\uht^{(s)}\}_{s \in \IN}$ converges to some~$\zht \in \Zht$. Since the~``tilde" operator is continuous due to~\eqref{eq:X:2} in Lemma~\ref{lemma:ext:tilded}, it follows that~$\whtt^{(s)} \to \whtt$. This, combined with the fact that~$\uht^{(s)} \to \zht$, implies that~$\zht$ satisfies the linearized problem~\eqref{eq:semi-implicit-formulation} with coefficient~$\whtt$. The uniqueness of the solution to~\eqref{eq:kernel-linearized-problem} then ensures that~$\zht = \Phi(\wht)$. This shows the continuity of~$\Phi$. 

The result then follows by the Schauder fixed-point theorem (see, e.g., \cite[Thm.~4.1.1 in Ch.~4]{Smart:1974}). 
\end{proof}

\begin{remark}[Insights on the semi-implicit DG time discretization]
The scheme introduced in this section can be seen as an extension of the semi-implicit Euler scheme used in~\cite[\S3.1]{Han_Hou:2021} to high-order approximations. In fact, for~$\ell = 0$, method~\eqref{eq:semi-implicit-formulation} reduces to the one in~\cite[Eq.~(3.3)]{Han_Hou:2021} with the only difference that we solve a nonlinear system of equations for the first time slab. 
The reason is that, for~$\ell > 0$, lifting the initial condition to the partial cylinder~$Q_1$ as a constant function in time is not enough to get high-order convergence. 
Nonetheless, for~$n > 1$, as already observed, the system of equations arising from~\eqref{eq:semi-implicit-formulation} is linear. 
\eremk
\end{remark}

\begin{lemma}[Bound for the convection term]\label{lemma:cons-conv-new}
Let $\u$ denote the velocity solution to the continuous weak formulation~\eqref{eq:weak-formulation} and let~$\uht \in \Zht$ be the discrete solution to the 
semi-implicit space--time formulation~\eqref{eq:semi-implicit-formulation}. Let the following regularity assumptions hold: $\u \in L^\infty(0,T;W^{k+1}_4(\Omega)^d)$, 
$\u \in W^{\ell+1}_{\infty}(0,T;L^2(\Omega)^d)$,  $(\Nabla \u) \u \in W^{\ell+1}_{\infty}(0,T;L^2(\Omega)^d)$. Let furthermore $\vht$ be any function in $\Zht$ satisfying (for all $n \in \{1,2,\ldots,N\}$)
\begin{equation}\label{eq:vh:equiv-new}
\esssup_{t \in I_n} | \vht(\cdot,t) |_S \lesssim  \esssup_{t \in I_n} | \Piht \eu (\cdot,t) |_S \, ,
\end{equation}
where $| \cdot |_S$ denotes any seminorm on $\Vh$ and the hidden constant depends only on $\ell$. Then, for any positive real $\varepsilon$ and $n \in \{1,2,\ldots,N\}$, it holds
$$
\begin{aligned}
& \big| \cht^{(n)}(\uhtt; \Piht \u, \vht) - \int_{I_n} \! c(\u; \u, \vht)\dt \big|
\le   C \, {\cal R}_{\u} \tau_n \varepsilon^{-1}  \Big( h^{2k+1} {+ \tau_n^{2\ell+2}}  \Big) \\
& \ + \varepsilon C \Big( \tau_n \,  \| \Piht \eu \|_{L^\infty(I_n;L^2(\Omega)^d)}^2 +  
\sum_{F \in \FhI} \QtRn\Big( \gamma_F(\uhtt) \, \|\jump{\vht}\|_{L^2(F)^d}^2 \Big) \\
& \ 
+ \tau_n C^\star_{\u} \| \Piht \eu \|_{L^\infty(\Inmo \cup \In; L^2(\Omega)^d)}^2  \, .
\end{aligned}
$$
where $C^\star_{\u} = \widetilde{C} \| \u \|_{L^\infty(0, T;W^{1}_{\infty}(\Omega)^d)}$ and both constants $C,\widetilde{C}$ only depend on $k$, $\ell$, $\rho$, and~$\eta$. The term ${\cal R}_{\u}$ depends on~$\u$ evaluated in the norms of the regularity assumption in the statement of this lemma. 
\end{lemma}
\begin{proof}
For~$n = 1$, the estimate can be obtained exactly as in Lemma~\ref{lemma:cons-conv}. Let~$n \in \{2, \ldots, N\}$, and recall Remark~\ref{rem:integr} and the forms defined in~\eqref{eq:def-convec-forms}. Adding and subtracting suitable terms, we get
\begin{alignat}{3}
\label{eq:split-new-conv}
\big|\cht^{(n)}(\uhtt; \Piht \u, \vht) - \int_{\In} c(\u; \u, \vht) \dt \big| \le J_1 + J_2 + \widetilde{J}_3^{(n)} + \widetilde{J}_5^{(n)},
\end{alignat}
where~$J_1^{(n)}$ and~$J_2^{(n)}$ are the same terms as in~\eqref{eq:conv:start}, and
\begin{equation*}
\begin{split}
\widetilde{J}_3^{(n)} &:= \big|\QtRn\big(\chtbh(\u - \uhtt; \Piht \u, \vht) \big) \big|, \\
\widetilde{J}_5^{(n)} & := \big|\QtRn \big( \chtt(\uhtt; \Piht \u, \vht) \big) \big|.
\end{split}
\end{equation*}

We bound the term~$\widetilde{J}_3^{(n)}$ as follows:
\begin{equation}
\begin{split}
\widetilde{J}_3^{(n)} & \le \big|\QtRn \big( \chtb(\u - \uhtt; \Piht \u, \vht) \big) \big| + \big|\QtRn \big( \chth(\u - \uhtt; \Piht \u, \vht) \big) \big| \\
& =: \widetilde{J}_{3,1}^{(n)} + \widetilde{J}_{3,2}^{(n)}.
\end{split}
\end{equation}

Using the definition of the quadrature rule, suitable H\"older inequalities, the stability properties of the projection operator~$\Piht$, bound~\eqref{eq:aux-estimate-tilde} and the triangle inequality, we obtain
\begin{alignat}{3}
\nonumber
\widetilde{J}_{3, 1}^{(n)} & = \big| \QtRn \big((\jump{\Nablah \Piht \u} (\u - \uhtt), \vht)_{\Omega} \big) \big| \\
%%%
\nonumber
& \le \tau_n \Norm{\Nablah \Piht \u}{L^{\infty}(\In; L^{\infty}(\Omega)^{d\times d})} \Norm{\u - \uhtt}{L^{\infty}(\In; L^2(\Omega)^d)} \Norm{\vht}{L^{\infty}(\In; L^2(\Omega)^d)} \\
%%%
\nonumber
& \le \widehat{C} \tau_n \Norm{\u}{L^{\infty}(\In; W^{1}_{\infty}(\Omega)^{d})} \Norm{\u - \uhtt}{L^{\infty}(\In; L^2(\Omega)^d)} \Norm{\vht}{L^{\infty}(\In; L^2(\Omega)^d)} \\
%%%
\nonumber
& \le \widehat{C} \Breve{C} \tau_n \Norm{\u}{L^{\infty}(\In; W^1_{\infty}(\Omega)^d)} \big(\Norm{\u - \uht}{L^{\infty}(\Inmo; L^2(\Omega)^d)} + 
(h^{k+1} + \tau_n^{\ell + 1}) \Ru \big) \Norm{\vht}{L^{\infty}(\In; L^2(\Omega)^d)} \\
%%%
\nonumber
& \le \widehat{C} \Breve{C} \tau_n \Norm{\u}{L^{\infty}(\In; W^1_{\infty}(\Omega)^d)} \big(\Norm{\u - \Piht \u}{L^{\infty}(\Inmo; L^2(\Omega)^d)} + (h^{k+1} + \tau_n^{\ell + 1}) \Ru \\
\label{eq:J31-new}
& \quad + \Norm{\Piht \eu}{L^{\infty}(\Inmo; L^2(\Omega)^d)}\big) \Norm{\vht}{L^{\infty}(\In; L^2(\Omega)^d)}.
\end{alignat}
The above inequality, together with the approximation properties in Lemma~\ref{lemma:estimate-Piht} for~$\Piht$, the assumption in~\eqref{eq:vh:equiv-new}, and the Young inequality with parameter~$\varepsilon > 0$, yields (recall that~$\Ru$ may change at each occurrence)
\begin{alignat*}{3}
\widetilde{J}_{3,1} \le C \varepsilon^{-1} \Ru ( \tau_n^{2\ell + 3} + \tau_n h^{2k + 2}) +  C \varepsilon \tau_n \Norm{\Piht \eu}{L^{\infty}(\In; L^2(\Omega)^d)}^2 + C_{\u}^{\star} \tau_n \Norm{\Piht \eu}{L^{\infty}(\Inmo \cup \In; L^2(\Omega)^d)}^2,
\end{alignat*}
where~$C_{\u}^{\star} = \widetilde{C} \Norm{\u}{L^{\infty}(\In; W^1_{\infty}(\Omega)^d)}$, and the constants~$C$ and~$\widetilde{C}$ depend only on~$k$, $\ell$, $\rho$, and~$\eta$. 

Proceeding as for the term~$J_{3,2}^{(n)}$ in~\eqref{eq:set:1}, we get
\begin{alignat}{3}
\nonumber
\widetilde{J}_{3, 2}^{(n)} & = \big| \QtRn \big(( (\u - \uhtt)\cdot \bnF \jump{\Piht \u}, \mvl{\vht})_{\Omega} \big) \big| \\
& \lesssim \widehat{C}  \tau_n \Norm{\u}{L^{\infty}(\In; W^1_{\infty}(\Omega)^d)} \Norm{\u - \uhtt}{L^{\infty}(\In; L^2(\Omega)^d)} \Norm{\vht}{L^{\infty}(\In; L^2(\Omega)^d)},
\end{alignat}
where we note that the right-hand side exactly corresponds to the third line in~\eqref{eq:J31-new}. Therefore,
% and estimating the term~$\Norm{\u - \uhtt}{L^{\infty}(\In; L^2(\Omega)^d)}$ as for~$J_{3,1}^{(n)}$, it follows:
\begin{equation*}
\widetilde{J}_{3,2}^{(n)} \le C \varepsilon^{-1} \Ru ( \tau_n^{2\ell + 3} + \tau_n h^{2k + 2}) +  C \varepsilon \tau_n \Norm{\Piht \eu}{L^{\infty}(\In; L^2(\Omega)^d)}^2 + C_{\u}^{\star} \tau_n \Norm{\Piht \eu}{L^{\infty}(\Inmo \cup \In; L^2(\Omega)^d)}^2.
\end{equation*}

We now focus on the term~$\widetilde{J}_5$. Following the same steps as in~\eqref{eq:new:1} and~\eqref{eq:Tsplit} for~$J_5^{(n)}$, we obtain
\begin{alignat}{3}
\nonumber
\widetilde{J}_5^{(n)} & = \frac12 \big|\QtRn\big((\gamma (\uhtt) \jump{\Piht \u}, \vht)_{\FhI} \big) \big| \\
& \lesssim \tau_n^{1/2}\big(J_1' + \widetilde{J}_2'\big)^{\frac12} \Big(\sum_{F \in \FhI} \QtRn\big(\gamma_F (\uhtt) \Norm{\vht}{L^2(F)^d}^2 \big) \Big)^{\frac12},
\end{alignat}
with~$J_1'$ as in~\eqref{eq:J1p-J2p}, and
\begin{equation*}
\widetilde{J}_2' := \sum_{F \in \FhI} \Norm{\uhtt}{L^{\infty}(\In; L^{\infty}(F)^d)} \Norm{\jump{\Pt (\Id - \IRT) \u}}{L^{\infty}(\In; L^2(F)^d)}^2.
\end{equation*}
The same steps used in~\eqref{eq:case-T2} to treat the term~$J_2'$, combined with~\eqref{eq:X:1} in Lemma~\ref{lemma:ext:tilded} and the fact that~$\Norm{\uht}{L^{\infty}(\In; L^2(\Omega)^d)}$ is bounded uniformly in~$h$ and~$\tau$, easily lead to
\begin{equation*}
\widetilde{J}_2' \lesssim h^{2k + 1} \Ru.
\end{equation*}
The following bound for the term~$\widetilde{J}_4^{(n)}$ can then be obtained applying~\eqref{eq:case-T1} and a Young inequality:
\begin{equation}
\widetilde{J}_5 \lesssim \varepsilon^{-1} \tau_n h^{2k + 1} \Ru + \varepsilon \sum_{F \in \FhI} \Big(\gamma_F(\uhtt) \Norm{\jump{\vht}}{L^2(F)^d}^2 \Big).
\end{equation}
The desired result then follows by combining the estimates for~$J_1^{(n)}$, $J_2^{(n)}$, $\widetilde{J}_3^{(n)}$, and~$\widetilde{J}_5^{(n)}$ into~\eqref{eq:split-new-conv}.
\end{proof}

We recall the definition of the error function~$\eu = \u - \uht$ and of the composed projection~$\Piht = \Pt \circ \IRT$. The next result shows that the semi-implicit scheme converges with the same convergence rates as the fully discrete scheme~\eqref{eq:kernel-space-time-problem}, although with a slightly more stringent step length restriction.
%%%
\begin{theorem}[Estimate for the discrete 
error]\label{prop:final-semi-implicit}
Let~$\eta = \max_{n=2,\ldots,N}\{ \tau_n/\tau_{n-1} \}$. Let also~$\u$ denote the velocity solution to the continuous weak formulation~\eqref{eq:weak-formulation}, and let~$\uht \in \Zht$ be the discrete solution to the space--time formulation~\eqref{eq:semi-implicit-formulation}. Let $\u\in W^{1}_{\infty}(0,T; H^{k+1}(\Omega)^d) \cap L^{\infty}(0, T; W^{k+1}_{4}(\Omega)^d) \cap W^{\ell + 1}_{\infty}(0, T; H^2(\Omega)^d)$, and~$(\Nabla \u) \u \in W^{\ell + 1}_\infty (0, T; L^2(\Omega)^d)$.
Then, there exists a positive constant~$C_{\ell}$ depending only on~$\ell$ such that, if~$\tau_n \le C_\ell/(4 \eta C_{\u}^{\star})$ for all~$n \in \{1,2,\ldots,N\}$ with~$C_{\u}^{\star}$ as in Lemma~\ref{lemma:cons-conv-new}, it holds
\begin{equation}\label{bound:final-semi-implicit}    
\begin{aligned}
\| \Piht \eu \|_{L^\infty(0,T;L^2(\Omega)^d)}^2 & + \nu \sum_{n = 1}^N \int_{I_n} \Norm{\Piht \eu(\cdot, t) }{\calA, h}^2 \dt + \sum_{n = 1}^{N-1} \Norm{\jump{\Piht \eu}_{n}}{L^2(\Omega)^d}^2 \\ 
&  + \sum_{n = 1}^{N} \sum_{F \in \FhI} \QtRn\Big( \int_F  \gamma_F(\uhtt) \, |\jump{\Piht \eu}|^2 \dS \Big) \\
& \lesssim \, \Big(1 + \exp\Big(\frac{\eta C_{\u}^{\star}}{C_{\ell}} T\Big)\Big)
 \big( {\tau^{2 \ell+2}} + (\nu + h) h^{2k}\big) \, ,
\end{aligned}
\end{equation}
with hidden constant that depends only on $k$, $\ell$, $\rho$, $\eta$, and on the solution~$\u$ (evaluated in the norms of the regularity assumption above).
\end{theorem}
%%%
\begin{proof}
The proof follows closely the one for Theorem~\ref{prop:final}, so we focus on the main differences due to the semi-implicit treatment of the nonlinear term. 
In particular, the new convection term (which is the only difference between the two schemes) is still antisymmetric with respect to the last two entries, so the same coercivity arguments still apply.
In the first time slab~($n = 1$), the two methods coincide; therefore, we omit the details for such a case and concentrate on the subsequent time slabs~($n > 1$).

The test function~$\vht \in \Zht$ defined in~\eqref{eq:test-function-error} can be used to show that
\begin{equation}
\label{eq:new-S1}
\Bht(\uhtt; \Piht \eu, \vht) \geq \widetilde{E}_n,
\end{equation}
with
\begin{equation*}
\begin{aligned}
\widetilde{E}_n & :=  
{\frac12} \Norm{\Piht \eu(\cdot, \tn^{-})}{L^2(\Omega)^d}^2  + \nu \sum_{m = 1}^n \int_{I_m} \Norm{\Piht \eu(\cdot, t) }{\calA, h}^2 \dt  + {\frac12} \sum_{m = 1}^{n - 1} \Norm{\jump{\Piht \eu}_{m}}{L^2(\Omega)^d}^2 \\
& \quad + \frac14 \Norm{\Piht\eu(\cdot, 0)}{L^2(\Omega)^d}^2
 + {\frac12} \sum_{m = 1}^{n} \sum_{F \in \FhI} \QtRm\Big( \int_F  \gamma_F(\uhtt) \, |\jump{\Piht \eu}|^2 \dS \Big) \, .
\end{aligned}
\end{equation*}
Moreover, for all~$\varepsilon > 0$, the following bound can be obtained as in~\eqref{finals:2}, by estimating the convective term with Lemma~\ref{lemma:cons-conv-new} instead of Lemma~\ref{lemma:cons-conv}:
\begin{alignat}{3}
\nonumber
& \Bht(\uhtt; \Piht\eu,\vht) \le 
C \varepsilon^{-1} {\cal R}_{\u} \sum_{m=1}^n \Big( {\tau_m^{2 \ell + 3}} + \tau_m (\nu+h) h^{2k} \Big) + C \varepsilon^{-1} {\cal R}_{\u} h^{2k+2} \\
\nonumber
& \ 
+ C\varepsilon \sum_{m=1}^n \Big[ \tau_m \| \Piht \eu \|_{L^\infty(I_m;L^2(\Omega)^d)}^2  \ + \nu \int_{I_m} \!\! \Norm{\Piht \eu(\cdot, t) }{\calA, h}^2
+ \!\!\!\! \sum_{F \in \FhI} \!\!\! \QtRm\Big( \int_F  \gamma_F(\uhtt) \, |\jump{\Piht \eu}|^2 \Big)
\Big] \\
\label{eq:new-S2}
& \ + C \varepsilon \Norm{\Piht \eu(\cdot, 0)}{L^2(\Omega)^d}^2 {+ C_{\u}^{\star}} \sum_{m=1}^n \tau_m \| \Piht \eu \|_{L^\infty(I_{m-1} \cup I_m; L^2(\Omega)^d)}^2\, ,
\end{alignat}
where~$C_{\u}^{\star}$ is as in Lemma~\ref{lemma:cons-conv-new}, and the generic constant~$C$ depends only on the degrees of approximation~$k$ and~$\ell$, the shape-regularity parameter~$\rho$, and~$\eta$.

The last term in~\eqref{eq:new-S2} can be further bounded as
\begin{equation}\label{eq:boundIm}
C_{\u}^{\star} \sum_{m = 1}^n \tau_m \Norm{\Piht \eu}{L^{\infty}(I_{m-1} \cup I_m; L^2(\Omega)^d)}^2 \le 2 \eta C_{\u}^{\star} \sum_{m = 1}^n \tau_m \Norm{\Piht \eu}{L^{\infty}(I_m; L^2(\Omega)^d)}^2,
\end{equation}
which, combined with~\eqref{eq:new-S1} and~\eqref{eq:new-S2}, leads to
\begin{equation}
\label{eq:new-S3}
\widetilde{E}_n \le C \Ru \sum_{m = 1}^n \big(\tau_m^{2 \ell + 3} + \tau_m (\nu + h) h^{2k} \big) + C \mathcal{R}_u h^{2k + 2} + 2 \eta C_{\u}^{\star} \sum_{m = 1}^n \tau_m \Norm{\Piht \eu}{L^{\infty}(I_m; L^2(\Omega)^d)}^2,
\end{equation}
for a sufficiently small~$\varepsilon$ (but depending only on~$k$, $\ell$, $\rho$, and~$\eta$).

Using now the test function~$\vht^{(n)}$ in~\eqref{eq:super-test-function-error} and proceeding as for~\eqref{finals:4} and~\eqref{finals:5}, for~$m = 2, \ldots, N$ and all~$\varepsilon > 0$, we get
\begin{alignat}{3}
\nonumber
C_{\ell} \widehat{\widetilde{E}}_m & - \frac14 \Norm{\Piht \eu(\cdot, t_{m-1}^-)}{L^2(\Omega)^d}^2 \le 
\Bht(\uhtt; \Piht\eu,\vht^{(m)}) \\
%%%
\nonumber
&  \le
C \varepsilon^{-1} {\cal R}_{\u} \Big( {\tau_m^{2 \ell + 3}} + \tau_m (\nu + h) h^{2k} \Big) \\
\nonumber
& \quad 
+ \varepsilon \Big[ \tau_m \| \Piht \eu \|_{L^\infty(I_m; L^2(\Omega)^d)}^2  + \nu \int_{I_m} \!\! \Norm{\Piht \eu(\cdot, t) }{\calA, h}^2 
+ \!\!\!\! \sum_{F \in \FhI} \!\!\! \QtRm\Big( \int_F  \gamma_F(\uhtt) \, |\jump{\Piht \eu}|^2 \Big)
\Big] \\
\label{eq:new-S4}
& \quad + {C_{\u}^{\star} }\tau_m \| \Piht \eu \|_{L^\infty(I_m \cup I_{m-1}; L^2(\Omega)^d)}^2,
\end{alignat}
where
\begin{alignat*}{3}
\widehat{\widetilde{E}}_m := &  
\| \Piht \eu \|_{L^\infty(I_m;L^2(\Omega)^d)}^2
+ \nu \int_{I_m} \Norm{\Piht \eu(\cdot, t) }{\calA, h}^2 \dt  \\
& + \Norm{\jump{\Piht \eu}_{m-1}}{L^2(\Omega)^d}^2 
+ \sum_{F \in \FhI} \QtRm\Big( \int_F  \gamma_F(\uhtt) \, |\jump{\Piht \eu}|^2 \dS \Big) \, .
\end{alignat*}
Multiplying~\eqref{eq:new-S4} by~$\tau_m$ and summing up from~$m = 1$ to~$n$, for a sufficiently small~$\varepsilon$ and by the same observation as in~\eqref{eq:boundIm}, we obtain
\begin{alignat*}{3}
& \sum_{m = 1}^n \tau_m \Norm{\Piht \eu}{L^{\infty}(I_m; L^2(\Omega)^d}^2 \le \sum_{m = 1}^n \tau_m \widehat{\widetilde{E}}_m \le C \Ru \sum_{m = 1}^n \tau_m \big(\tau_m^{2 \ell + 3} + \tau_m (\nu + h) h^{2k} \big) \\
& \qquad + \frac{1}{4C_{\ell}} \sum_{m = 1}^{n} \tau_m \Norm{\Piht \eu(\cdot, t_{m-1}^-)}{L^2(\Omega)^d)}^2 
+ 2 \eta \frac{C_{\u}^{\star}}{C_{\ell}}\sum_{m = 1}^n \tau_m^2 \Norm{\Piht \eu}{L^{\infty}(I_m; L^2(\Omega)^d)}^2.
\end{alignat*}
Since~$2\tau \eta C_{\u}^{\star}/ C_{\ell} \le 1/2$, we have
\begin{equation*}
\sum_{m = 1}^n \tau_m \Norm{\Piht \eu}{L^{\infty}(I_m; L^2(\Omega)^d)}^2 \le 2 C \Ru \sum_{m = 1}^n \tau_m \big(\tau_m^{2\ell + 3} + \tau_m (\nu + h) h^{2k} \big) + \frac{1}{2 C_{\ell}} \sum_{m = 1}^n \tau_m \Norm{\Piht \eu(\cdot, t_{m - 1}^-)}{L^2(\Omega)^d}^2.
\end{equation*}
Inserting the above bound in~\eqref{eq:new-S3}, it follows:
\begin{equation*}
\widetilde{E}_n \le C \Ru \sum_{m = 1}^n \big(\tau_m^{2\ell + 3} + \tau_m (\nu + h) h^{2k} \big) + C \Ru h^{2k + 2} + \frac{\eta C_{\u}^{\star} }{C_{\ell}} \sum_{m = 1}^n \tau_m \Norm{\Piht \eu(\cdot, t_{m -1}^-)}{L^2(\Omega)^d}^2
\end{equation*}
Since~$\Norm{\Piht \eu (\cdot, t_{m - 1}^-)}{L^2(\Omega)^d}^2 \le 2 \widetilde{E}_{m-1}$, estimate~\eqref{bound:final-semi-implicit} then follows by using the discrete Gr\"onwall inequality. 
\end{proof}

From Theorem \ref{prop:final-semi-implicit}, an identical counterpart of Corollary \ref{corol:final} can be immediately derived.
\begin{corollary}[\emph{A priori} error estimate]\label{corol:final:2}
Under the same assumptions as in Theorem \ref{prop:final-semi-implicit}, it holds 
\begin{equation}
\begin{aligned}
\| \eu \|_{L^\infty(0,T;L^2(\Omega)^d)}^2 & + \nu \sum_{n = 1}^N \int_{I_n} \Norm{\eu(\cdot, t) }{\calA, h}^2 \dt 
+ \sum_{n = 1}^{N} \sum_{F \in \FhI} \QtRn\Big( \int_F  \gamma_F(\uhtt) \, |\jump{\eu}|^2 \dS \Big) \\
& \lesssim \, \sum_{n=1}^N \Big( {\tau_n^{2 \ell + 3}} + \tau_n (\nu+h) h^{2k} \Big) + h^{2k+2} 
\, \lesssim \, {\tau^{2 \ell+2}} + (\nu + h) h^{2k} \, .
\label{eqn:finalEst}
\end{aligned}
\end{equation}
\end{corollary}
Analogously, the observations in Remark \ref{remark:final} apply also to this scheme.
}

% ---------------------------------------------------------
%               NUMERICAL EXPERIMENTS
% ---------------------------------------------------------
\section{Numerical experiments\label{sec:num}}

\red{In this section, we will provide numerical evidence of 
our theoretical results for both the fully implicit (\texttt{full-impl}, Section~\ref{subsec:fully-discrete-formulation}) and 
the semi-implicit (\texttt{semi-impl},  Section~\ref{sec:semi-implicit}) schemes.}
Specifically, in Section~\ref{sec:numConv}, we verify the convergence properties of the proposed schemes in both
the diffusion- and the convection-dominated regimes.
Then, in Section~\ref{sec:convT}, 
we focus on the convergence in time, and  
numerically assess the independence of the velocity error 
from negative powers of the viscosity~$\nu$ (Reynolds-semi-robustness) and from the pressure~$p$ (pressure-robustness).

Since we will mostly refer to the error estimate~\eqref{eqn:finalEst}, 
recalling that~$\eu = \u - \uht$, we define the squared error
$$
\errU^2:=\| \eu \|_{L^\infty(0,T;L^2(\Omega)^d)}^2 + \nu \sum_{n = 1}^N \int_{I_n} \Norm{\eu(\cdot, t) }{\calA, h}^2 \dt 
+ \sum_{n = 1}^{N} \sum_{F \in \FhI} \QtRn\Big( \int_F  \gamma_F(\uht) \, |\jump{\eu}|^2 \dS \Big)\,.
$$

The time step~$\tau$ and the spatial meshes used 
are indicated in each section.
Specifically, for the space discretization,
we construct four Delaunay triangular meshes of the unit square~$\Omega=(0,\,1)^2$
with decreasing mesh size~$h$ using the package~\texttt{triangle}~\cite{Shewchuk:1996:TEA}.
We denote such meshes as \texttt{mesh1}, \texttt{mesh2}, \texttt{mesh3}, and \texttt{mesh4}
from the coarsest to the finest.
For the time discretization, we always consider the time interval~$[0,\,1]$
and use uniform time steps. 
More precisely, we construct four time partitions 
\texttt{inter1}, \texttt{inter2}, \texttt{inter3}, and \texttt{inter4},
corresponding to time steps~$\tau$ equal to~$1/3$, $1/6$, $1/12$, and~$1/24$, respectively.
We recall that no constraints relating~$\tau$ and~$h$ are needed in our analysis.

We set the same approximation degree for both space and time, denoted by~$k$.
If not explicitly specified in the section,
we compute the errors 
using the following space--time discretizations:
$$
(\texttt{mesh1},\texttt{inter1}),\qquad(\texttt{mesh2},\texttt{inter2}),\qquad (\texttt{mesh3},\texttt{inter3}),\qquad\text{and}\qquad(\texttt{mesh4},\texttt{inter4}).
$$

Moreover, we set the penalty parameter~$\sigma = 10k^2$ in~\eqref{def:aht}. 
The nonlinear systems stemming from~\eqref{eq:space-time-formulation} are solved using a simple fixed-point iteration with tolerance~$10^{-8}$, which uses the approximation from the previous fixed-point iteration in the convective term.

The proposed 
space--time method has been implemented taking the steps from the \texttt{C++} library \texttt{Vem++}~\cite{Dassi:2023:VAC}.

% --------------------------------------------------
%               CONVERGENCE (h, tau)
% --------------------------------------------------
\subsection{Convergence test}\label{sec:numConv}

In this section, we validate the 
properties of~$\errU$ in the diffusion- and convection-dominated regimes, according to Corollaries \ref{corol:final} \red{ and \ref{corol:final:2},
as well as Remark~\ref{remark:final} and the analogous of Section~\ref{sec:semi-implicit}}.
We consider the IBVP defined in~\eqref{eq:model-problem}, 
and choose the right-hand side and the initial condition 
so that the exact solution is given by
\begin{equation}
\label{sol1}
\begin{split}
\mathbf{u}(x, y, t)&=\begin{bmatrix}
          -0.5\cos(t)\cos^2(\pi(x-0.5))\cos(\pi(y-0.5))\sin(\pi(y-0.5))\\
\phantom{-}0.5\cos(t)\cos^2(\pi(y-0.5))\cos(\pi(x-0.5))\sin(\pi(x-0.5))
\end{bmatrix}\,,\\[0.5em]
p(x,y,t) & = \cos(t)(\sin(\pi(x-0.5))-\sin(\pi(y-0.5)))\,.
\end{split}
% \label{eqn:sol1}
\end{equation}
We will 
choose~$\nu$ depending on which aspect of~$\errU$ we aim to verify numerically.

First, we will show that the estimate is quasi-optimal with respect to both~$\tau$ and~$h$ in a diffusion-dominated regime, and 
that it exhibits an additional~$h^{1/2}$ pre-asymptotic error reduction rate in convection-dominated cases.
To simulate a diffusion-dominated regime, 
we fix $\nu=1$, while~$\nu=10^{-5}$ is used to obtain a convection-dominated problem.
\red{In Figure~\ref{fig:diffConvDom}(left panel), 
we report the values of~$\errU$ for different space--time meshes, degrees~$k$, and values of~$\nu$ for both
the fully implicit and the semi-implicit schemes.}

\begin{figure}[!htb]
\centering
\begin{tabular}{cc}
\includegraphics[width=0.45\textwidth]{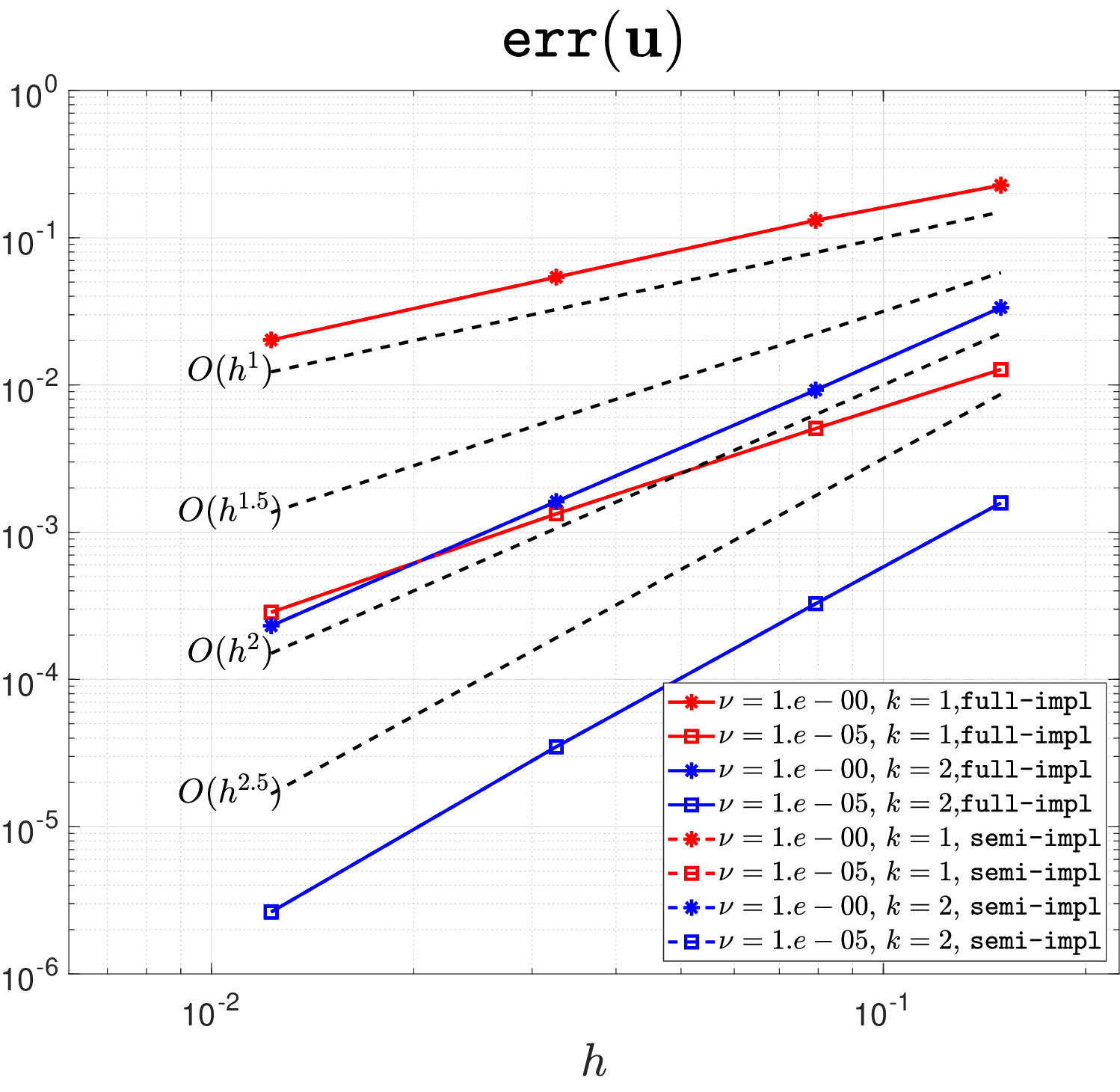} &
\includegraphics[width=0.45\textwidth]{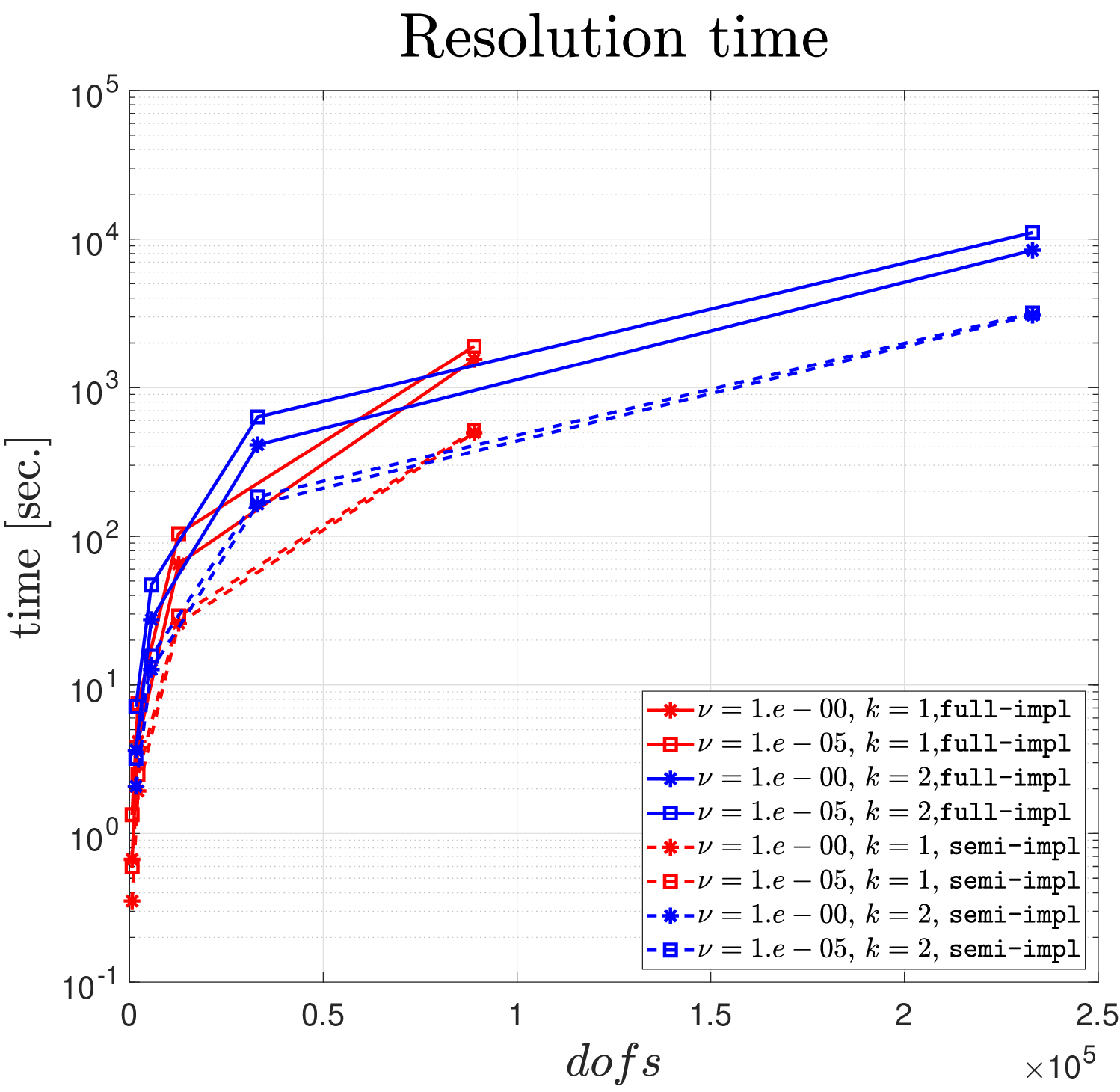}
\end{tabular}
\caption{\red{Convergence rates of the fully implicit and the semi-implicit schemes for the error~$\errU$ varying~$h$, 
considering different values of~$\nu$ and approximations of degree~$k=1$ and~$k = 2$, 
for the problem with exact solution in~\eqref{sol1}, left.
Resolution time in seconds for each simulation, right.}}
\label{fig:diffConvDom}
\end{figure}

We observe the expected convergence rates in all cases.
In particular, as stated in Corollaries~\ref{corol:final} \red{and \ref{corol:final:2}}, 
we observe an additional~$h^{1/2}$ gain in the convergence rate in the convection-dominated problems,
\red{i.e., when $\nu=10^{-5}$.
Moreover, the convergence lines of the \texttt{full-impl} and \texttt{semi-impl} schemes almost overlap,
indicating that there is no significant difference between these two approaches in terms of accuracy.
However, as it is shown in Figure~\ref{fig:diffConvDom}(right panel), 
the \texttt{semi-impl} method yields a substantial improvement in computation time. 
This is due to the fact that fixed-point iterations are performed only during the first time step, whereas in all subsequent steps, only a single linear system needs to be solved.}

\begin{figure}[!htb]
\centering
\includegraphics[width=0.8\textwidth]{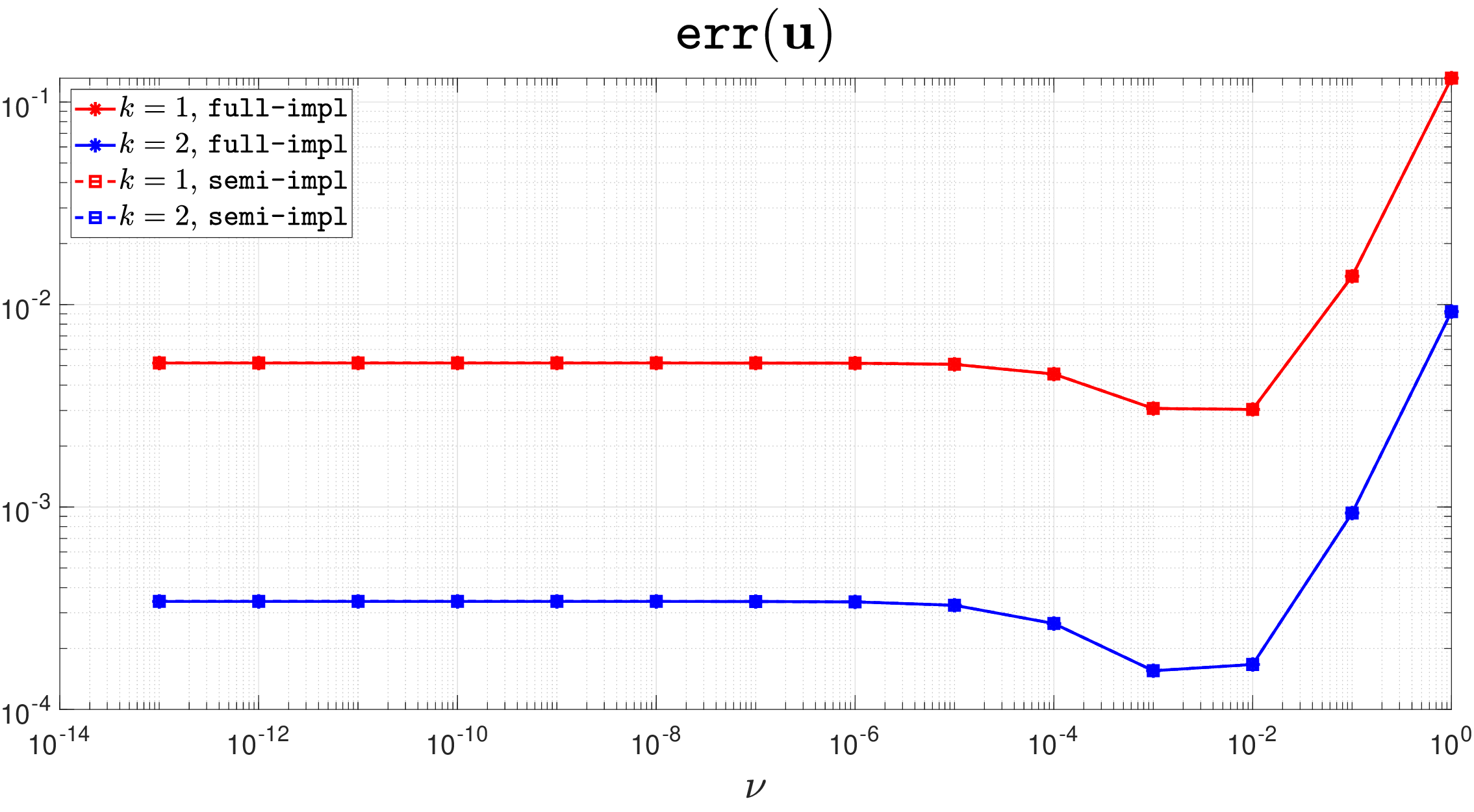}
\caption{\red{Behavior of the error~$\errU$ for the fully implicit and the semi-implicit schemes,
considering different values of~$\nu$ and approximations of degree~$k=1$ and~$k = 2$, 
for the problem with exact solution in~\eqref{sol1}.}}
\label{fig:nuStab}
\end{figure}

Next, we show that $\errU$ is Reynolds semi-robust 
meaning that the associated error constant remains bounded as~$\nu\to 0$.
To validate this numerically, 
we fix a space--time mesh and  solve the problem~\eqref{eq:model-problem} for different values of~$\nu$. 
In Figure~\ref{fig:nuStab}, we show the behavior of the error~$\errU$ for~$k=1$ and~$k = 2$, varying~$\nu$, and
using the space--time discretization~$(\texttt{mesh2},\texttt{inter2})$ \red{for both schemes}.
As expected, for~$\nu < 10^{-2}$, \red{these curves} become nearly horizontal,
thus providing numerical evidence that~$\errU$ is not affected by~$\nu$, 
in particular, as~$\nu\to 0$.
\red{Moreover, also in this case, the lines associated with the \texttt{full-impl} and \texttt{semi-impl} 
schemes are nearly indistinguishable, confirming that the accuracy of the two methods is effectively the same.}

Although not covered by our theory, in Figure~\ref{fig:linf}, 
we also observe (almost) optimal convergence rates for the 
error~$\| \eu \|_{L^\infty(0,T;L^2(\Omega)^d)}$ and for the~$L^2(\Omega)$-error of the pressure at the final time,
which decrease as~$\mathcal{O}(h^{k+1})$ and~$\mathcal{O}(h^k)$, respectively,
\red{for both the fully implicit and the semi-implicit schemes.}
Deriving error estimates for the pressure in strong norms is a nontrivial task and out of the scope of this manuscript. 
% but may be the focus of a future study.

\begin{figure}[!htb]
\centering
\begin{tabular}{cc}
\includegraphics[width=0.47\textwidth]{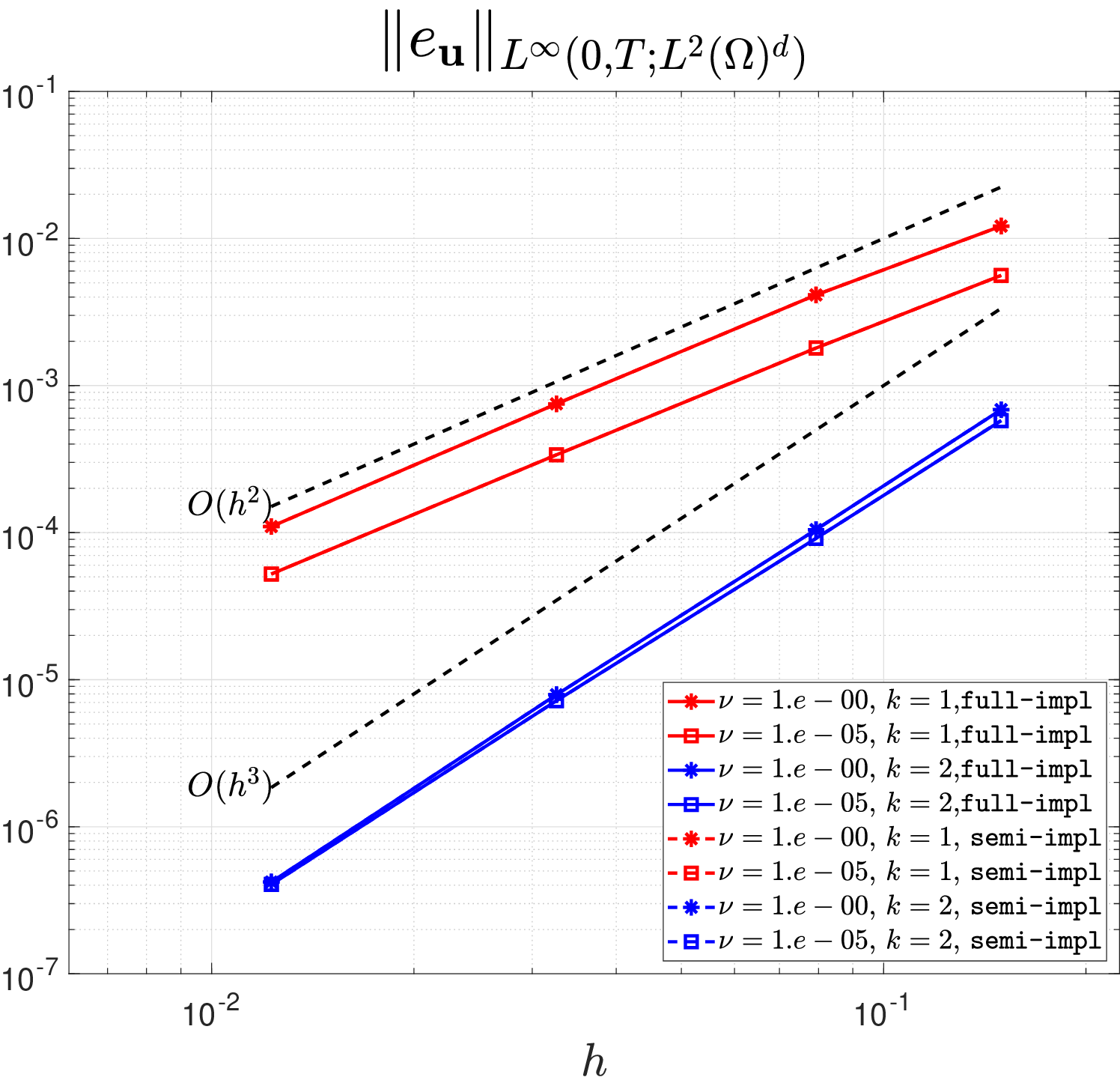} \hspace{0.2in}
\includegraphics[width=0.47\textwidth]{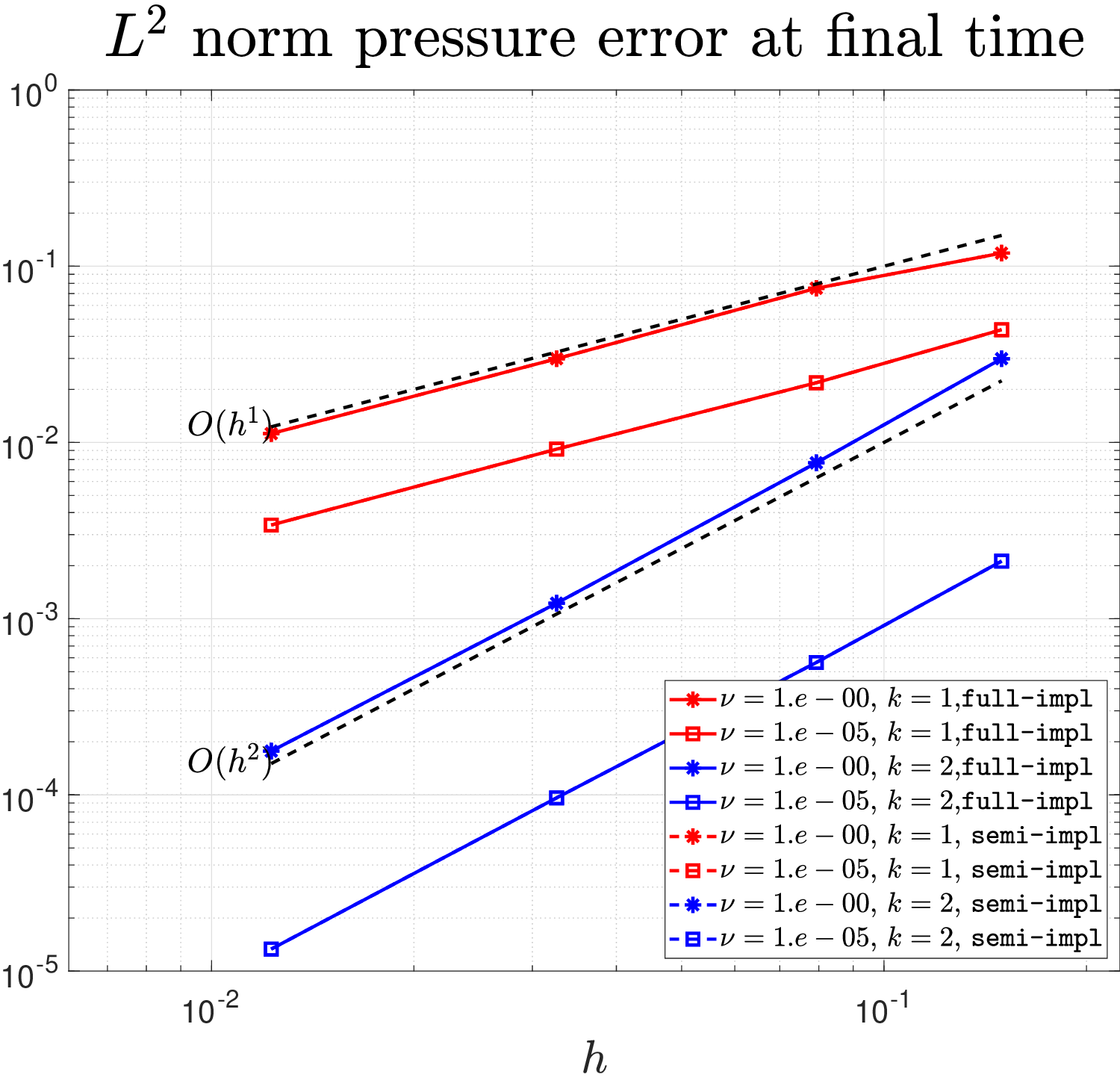}
\end{tabular}
\caption{\red{Behavior of the error~$\| \eu \|_{L^\infty(0,T;L^2(\Omega)^d)}$ (left) 
and of the $L^2(\Omega)$ pressure error at the final time (right) for the fully implicit and the semi-implicit schemes,
varying~$h$, considering different values of~$\nu$ and approximations of degree~$k=1$ and~$k = 2$, 
for the problem with exact solution in~\eqref{sol1}.}}
\label{fig:linf}
\end{figure}

\begin{comment}
\subsubsection{Unsymmetrical space time degree}

In this section, we consider the same problem as Section~\ref{sec:conv},
but now we take different polynomial approximation degrees for space and time.
Specifically, we will use a space approximation degree equal to 2, 
while the time approximation degree is 1.

Since the velocity field does not have high oscillation in time and
that considering a time approximation degree equal to 1,
the time discretisation error have a trend $\tau^2$,
we expect that the time discretisation error does not spoil the convergence of the solution, i.e.,
we have to see trends similar to the case $k=2$ of Figure~\ref{fig:diffConvDom}.

In Figure~\ref{fig:diffConvDomUnsym}, we show the values of $\errU$, left, and $\| \eu \|_{L^\infty(0,T;L^2(\Omega)^d)}$ varying the diffusive coefficient~$\nu$.

\begin{figure}[!htb]
\centering
\begin{tabular}{cc}
 &
\end{tabular}
\caption{{\color{red} Trend of the error $\errU$ (left) and $\| \eu \|_{L^\infty(0,T;L^2(\Omega)^d)}$ (right)
varying $h$ and considering different values of $\nu$, with space degree 2 and time degree 1.}}
\label{fig:diffConvDomUnsym}
\end{figure}
\end{comment}

%%%%%%%%%%%%%%%%%%%%%%%%%%%%%%%%%%%%%%%%%%%%%%%%% sei arrivato qui!

\subsection{Convergence test in~\texorpdfstring{$\tau$}{tau} and pressure-robustness}\label{sec:convT}

In this section, we perform a convergence analysis focused uniquely on the time discretization, 
which we also use to show  that the error bound in Corollary~\ref{corol:final} remains unaffected by 
both~$\nu$ and the~pressure~$p$.

To achieve this goal, we set a right-hand side, Dirichlet boundary conditions, and an initial condition, %of the IBVP~\eqref{eq:model-problem} 
such that the exact solution is given by
\begin{equation}
\label{sol2}
\begin{split}
\mathbf{u}(x,y, t) &= \begin{bmatrix}
          \cos(2\pi t)\,y\\
          \cos(2\pi t)\,x
\end{bmatrix}\,, \qquad 
p(x,y,t)  =  \cos(2\pi t)(\sin(\pi(x-0.5))-\sin(\pi(y-0.5)))\,.
\end{split}
\end{equation}
% Nonhomogeneous Dirichlet boundary conditions are treated in a standard way.

%It is important to note that 
Since the vector field~$\mathbf{u}$ is a polynomial of degree~$1$ in space, %and
%a trigonometric function in time.
%Therefore, since polynomials of degrees 1 and 2 in the spatial variables are included in the discrete space, 
the contribution of the spatial discretization to the error is negligible. 
Therefore, recalling also that due to the pressure robustness the velocity error is not affected by the pressure, the error decay for the velocities should depend only on the time discretization.

In order to practically verify such observation we conduct a time convergence analysis, considering the following sequence of space--time meshes:
$$
(\texttt{mesh2},\texttt{inter1}),\qquad(\texttt{mesh2},\texttt{inter2}),\qquad (\texttt{mesh2},\texttt{inter3}),\qquad\text{and}\qquad(\texttt{mesh2},\texttt{inter4}),
$$
where note that the spatial mesh is fixed.
%which involve keeping the spatial discretization fixed.

In Figure~\ref{fig:timeConv}, we show the behavior of~$\errU$ and~$\| \eu \|_{L^\infty(0,T;L^2(\Omega)^d)}$, where we observe 
%The convergence trends of 
optimal convergence rates of order~$\mathcal{O}(\tau^{k + 1})$ in both cases, 
%are observed for both~$\errU$ and~$\| \eu \|_{L^\infty(0,T;L^2(\Omega)^d)}$, 
which is in agreement with Corollary~\ref{corol:final}.
%align with the estimate in Equation~\eqref{eqn:finalEst},
%while those associated with $\| \eu \|_{L^\infty(0,T;L^2(\Omega)^d)}$ follow the expected trend. 
Moreover, in both cases, the errors remain unaffected by~$\nu$, i.e., 
for each degree~$k$, the convergence lines associated with the values~$\nu=1$ and $\nu=10^{-5}$ almost overlap.
\red{On the other hand, one cannot avoid noticing the high error values obtained by the semi-implicit scheme (for $k=2$) in the convection dominated case whenever the time step $\tau$ is large (note that here we purposefully tested the semi-implicit scheme also for smaller values of $\tau$ in order to show that for small time steps the method behaves correctly). Such a behavior seems to indicate an instability and is not in contrast with our theoretical results, which indeed require a sufficiently small time step for both methods (c.f. Theorems \ref{prop:final-semi-implicit} and \ref{prop:final}). Our intuition is that, while in practice the fully implicit method is more robust in this respect, the (more efficient) semi-implicit scheme pays a price in terms of sensibility: if the time step is not sufficiently small the method can really deliver bad results. Further studies on this aspect (and possible cures) may be the objective of further communications.}

\begin{figure}[!htb]
\centering
\begin{tabular}{cc}
% \includegraphics[width=0.47\textwidth]{imm/exe2ErrU_fs.eps}
% \hspace{0.2in}
% \includegraphics[width=0.47\textwidth]{imm/exe2LInf_fs.eps}\\
\includegraphics[width=0.47\textwidth]{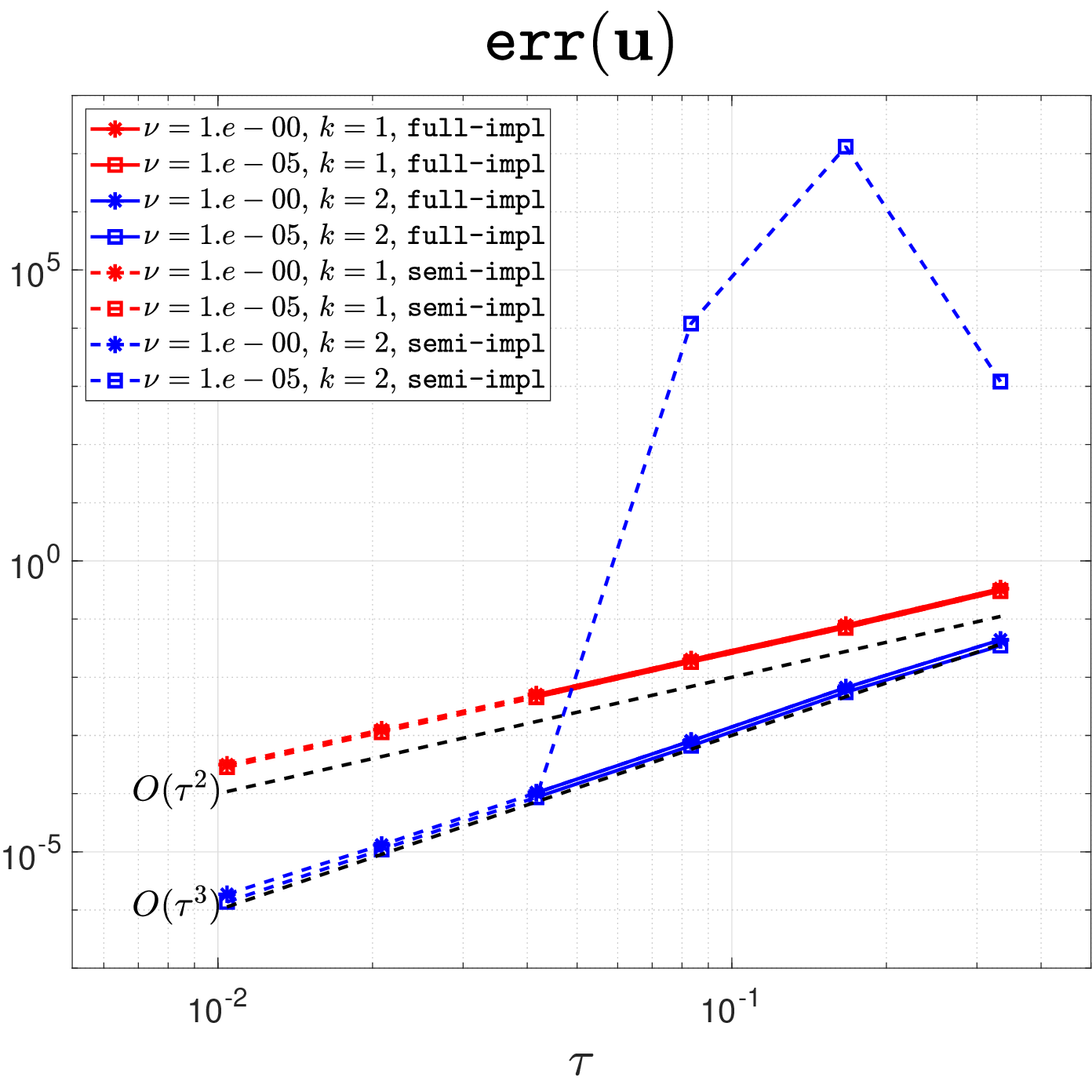}
\hspace{0.2in}
\includegraphics[width=0.47\textwidth]{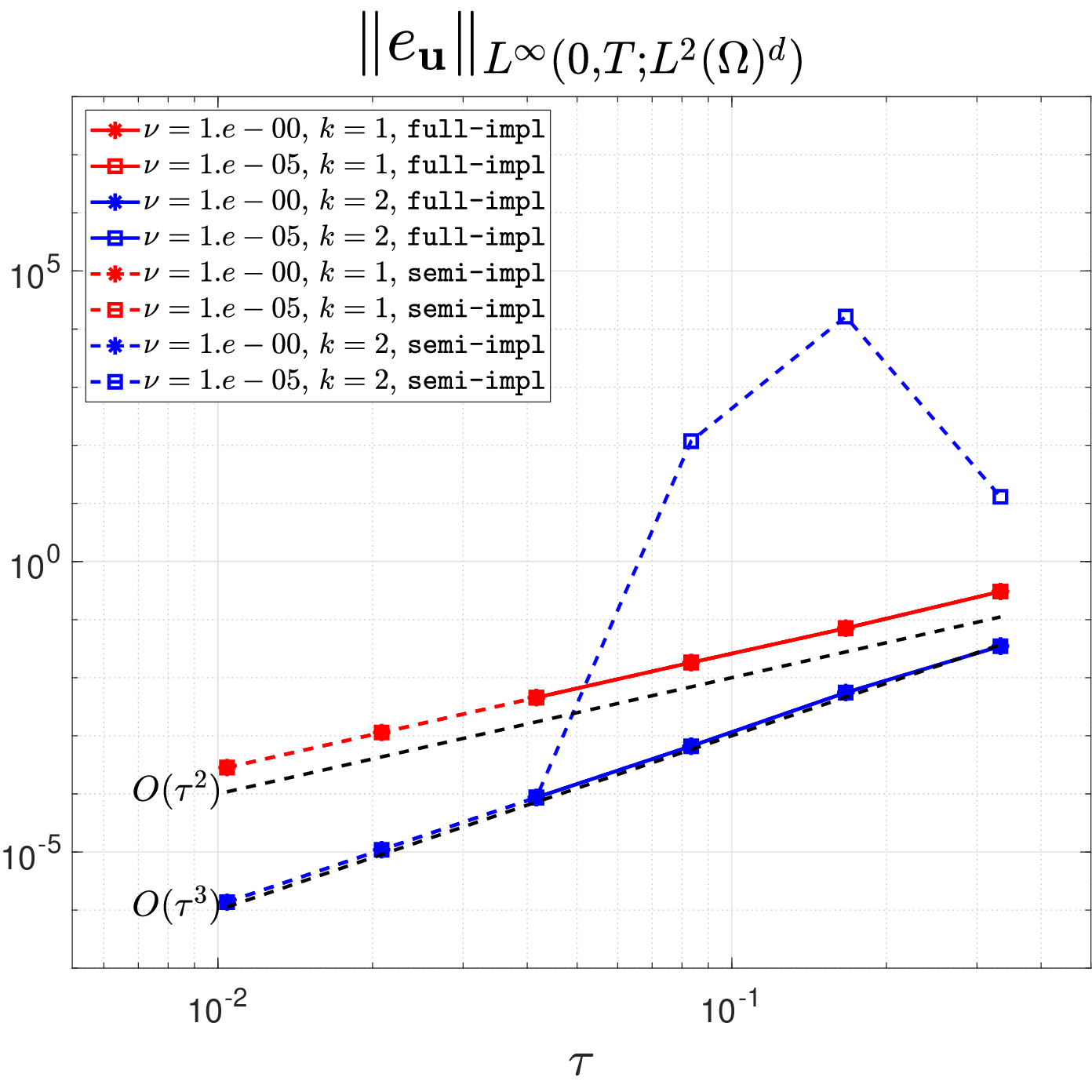}
\end{tabular}
\caption{\red{Convergence rates of the errors~$\errU$ (left) and~$\| \eu \|_{L^\infty(0,T;L^2(\Omega)^d)}$ (right) 
for the fully implicit and the semi-implicit schemes, 
varying~$\tau$, considering different values of $\nu$ and approximations of degree~$k=1$ and~$k = 2$, for the problem with exact solution in~\eqref{sol2}.}}
\label{fig:timeConv}
\end{figure}

In order to further illustrate the pressure-robustness of the method, 
we have also tested the same identical problem but with velocity solution ${\bf u}$ linear also in space, keeping the same pressure as above
\begin{equation}
\label{sol3}
\mathbf{u}(x,y, t) =\begin{bmatrix}
          y\,t\\
          x\,t
\end{bmatrix}\,.
% \label{eqn:sol3}
\end{equation}
\red{As expected, for the fully implicit scheme, 
the resulting~$\errU$ is comparable to the tolerance set for the fixed-point iterations (see Figure~\ref{fig:linke}).
The same holds for the semi-implicit scheme, since such a method is also affected by the error 
introduced by the fixed-point iterations at the first time step. These results provide further numerical evidence of the pressure-robustness of the scheme.}

\begin{figure}[!htb]
\centering
\begin{tabular}{cc}
\includegraphics[width=0.47\textwidth]{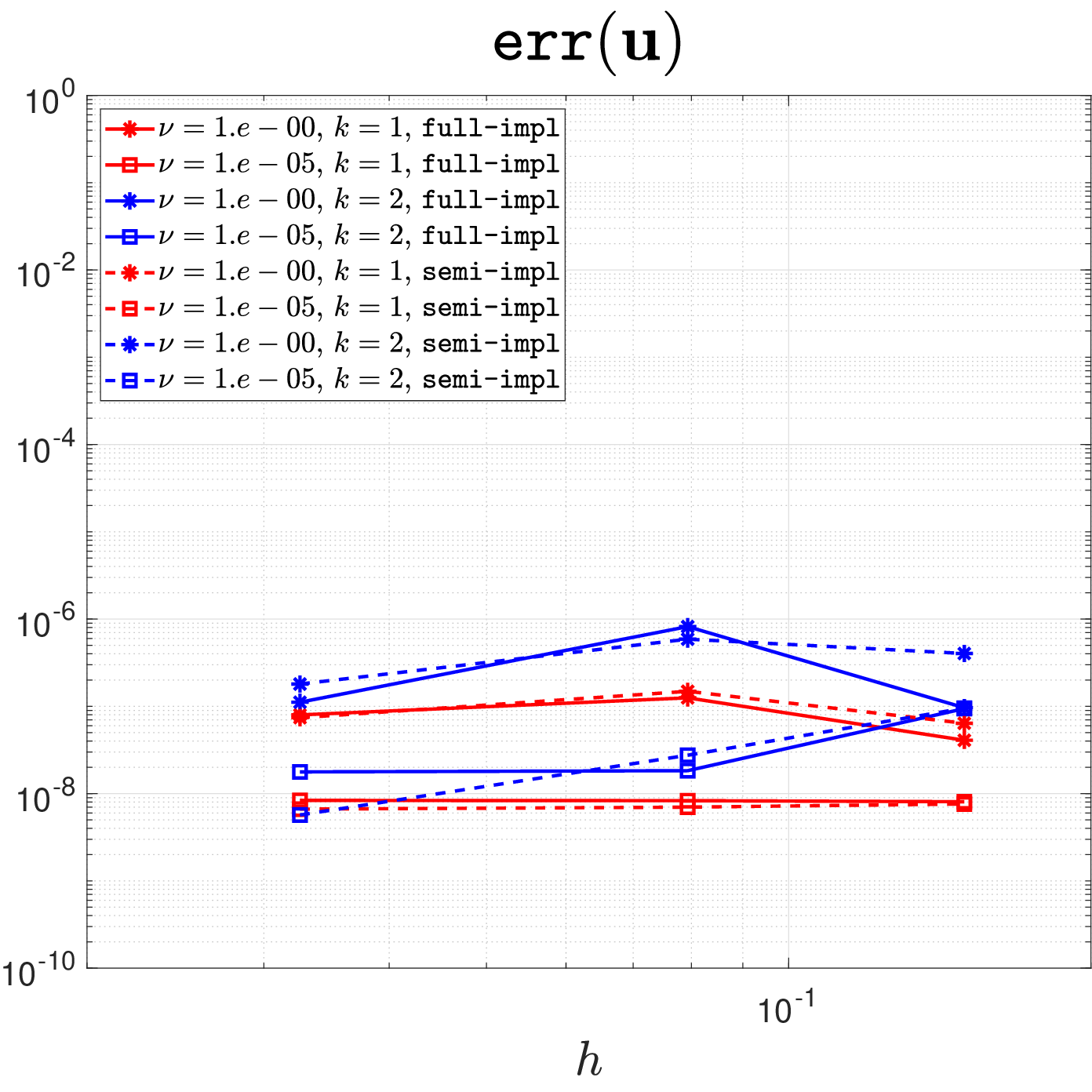}
\hspace{0.2in}
\includegraphics[width=0.47\textwidth]{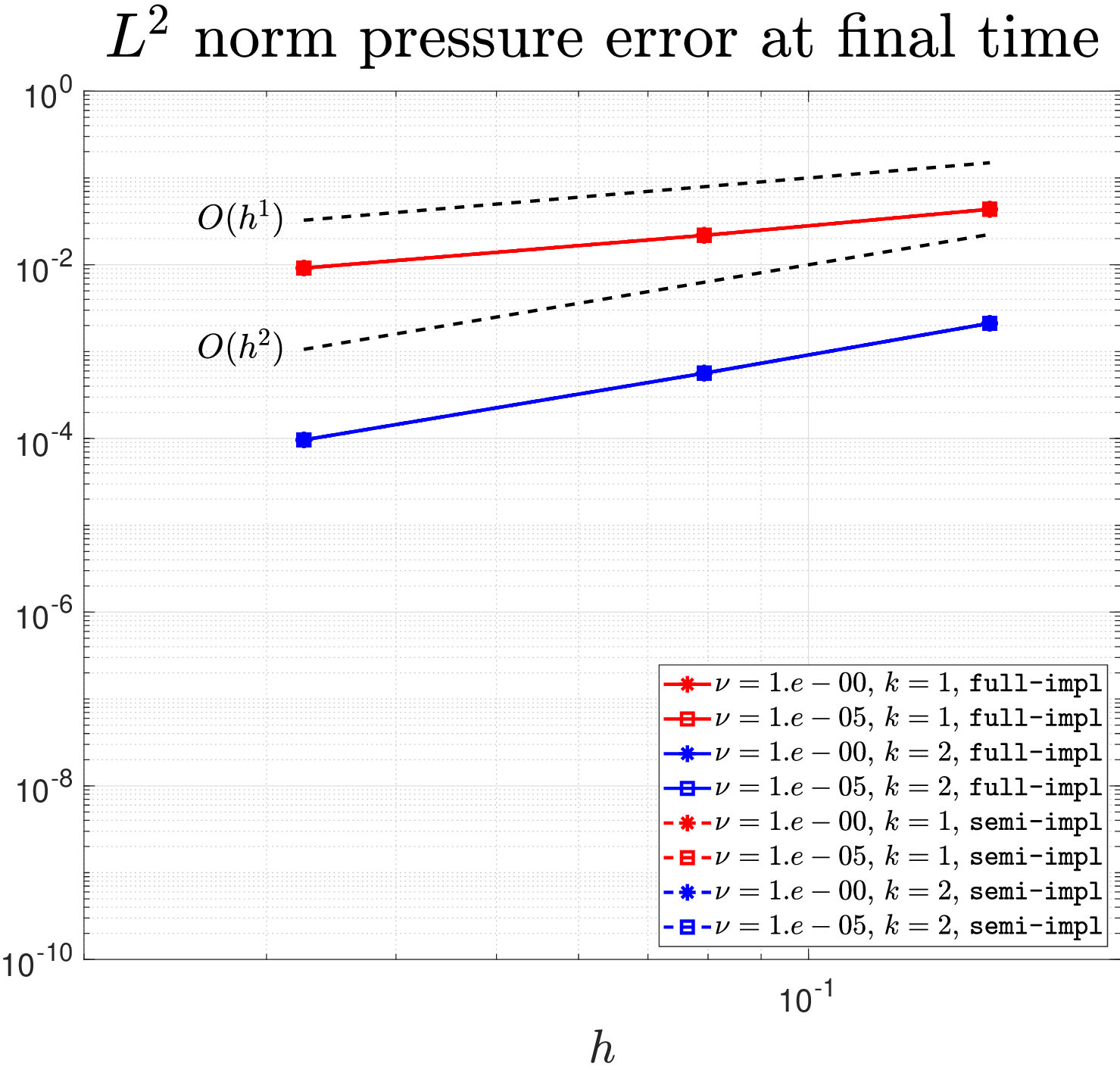}
\end{tabular}
\caption{\red{Behavior of the error $\errU$ (left) and of the~$L^2(\Omega)$-error of the pressure at the final time (right) for the fully implicit and the semi-implicit schemes, 
varying~$h$, considering different values of~$\nu$ and approximations of degree~$k=1$ and~$k = 2$, 
for the problem with exact solution in~\eqref{sol3}.}}
\label{fig:linke}
\end{figure}

% ---------------------------------------------------------
%               ACKNOWLEDGEMENTS
% ---------------------------------------------------------
\section*{{Acknowledgments}}

All authors where partially supported by the European Union (ERC Synergy, NEMESIS, project number 101115663). Views and opinions expressed are however those of the authors only and do not necessarily reflect those of the EU or the ERC Executive Agency.
All authors are also member of the INdAM-GNCS group.

% ---------------------------------------------------------
%               BIBLIOGRAPHY
% ---------------------------------------------------------

\end{document}